\documentclass[12pt]{amsart}

\usepackage{hyperref}
\usepackage{amssymb}
\usepackage[bbgreekl]{mathbbol}
\usepackage{graphicx}
\usepackage{ifthen}
\usepackage{pict2e}
\usepackage{xargs}
\usepackage{xspace}
\usepackage{xcolor}
\usepackage{pgf,tikz}
\usepackage{pgfplots}
\usepackage{environ}
\usepackage{caption}
\usetikzlibrary{arrows}
\usepackage{enumitem}
\usepackage{xifthen}
\usepackage[a4paper,left=2cm,right=2cm,top=1.5cm,bottom=1.5cm, includefoot]{geometry}
\usepackage{multicol}
\usepackage{multirow}


\newcounter{dummy}
\makeatletter
\newcommand\myitem[1][]{\item[#1]\refstepcounter{dummy}\def\@currentlabel{#1}}
\makeatother

\makeatletter
\newsavebox{\measure@tikzpicture}
\NewEnviron{scaletikzpicturetowidth}[1]{%
	\def\tikz@width{#1}%
	\begin{lrbox}{\measure@tikzpicture}%
		\BODY
	\end{lrbox}%
	\pgfmathparse{#1/\wd\measure@tikzpicture}%
	\BODY
}
\makeatother

\DeclareSymbolFontAlphabet{\mathbb}{AMSb}

\newcommand{\thistheoremname}{}
\newtheorem*{genericthm*}{\thistheoremname}
\newenvironment{namedthm*}[1]
{\renewcommand{\thistheoremname}{#1}%
	\begin{genericthm*}}
	{\end{genericthm*}}

\newcommand{\Bairespace}[1][]{
	\ifthenelse{\equal{#1}{}}{\functions{\N}{\N}}{\functions{#1}{\N}}
}
\newcommand{\bbL}{\mathbb{L}}
\newcommand{\bbX}{\mathbb{X}}

\newcommand{\Cantorspace}[1][]{
	\ifthenelse{\equal{#1}{}}{\functions{\N}{2}}{\functions{#1}{2}}
}

\newcommandx{\concatenation}[2][1 = undefined, 2 = undefined]{
	\ifthenelse{\equal{#1}{undefined}}{{}\smallfrown}{
		\ifthenelse{\equal{#2}{undefined}}{\bigoplus #1}{\bigoplus_{#1} #2}
	}
}

\newcommandx{\functions}[3][3 =]{
	\ifthenelse{\equal{#3}{}}{#2^{#1}}{#2_{#3}^{#1}}
}

\newcommand{\Gzero}[1][]{
	\ifthenelse{\equal{#1}{}}
	{\mathbb{G}_0}
	{\mathbb{G}_{0,n}}
}
\newcommandx{\Hzero}[2][2 = undefined]{
	\ifthenelse{\equal{#2}{undefined}}
	{\mathbb{H}_{#1}}
	{\mathbb{H}_{#1, #2}}
}
\newcommandx{\intersection}[2][1 =, 2 =]{
	\ifthenelse{\equal{#1}{}}{\cap}{
		\ifthenelse{\equal{#2}{}}{\bigcap #1}{{\bigcap_{#1} #2}}
	}
}
\newcommand{\Lzero}[1][]{\ifthenelse{\equal{#1}{}}{\bbL_0}{L_{0, #1}}}
\newcommand{\Lzerospace}[1][]{\ifthenelse{\equal{#1}{}}{\bbX_0}{X_{0, #1}}}

\newcommand{\modulo}[1]{\ (\text{mod } 2)}
\newcommand{\N}{\mathbb{N}}

\newcommandx{\product}[2][1 =, 2 =]{
	\ifthenelse{\equal{#1}{}}{\times}{
		\ifthenelse{\equal{#2}{}}{\prod #1}{{\prod_{#1} #2}}
	}
}

\newcommandx{\sequence}[2][2 = undefined]{
	\ifthenelse{\equal{#2}{undefined}}{(#1)}{
		(#1)_{#2}
	}
}

\newcommandx{\set}[2][2 = undefined]{
	\ifthenelse{\equal{#2}{undefined}}{\{ #1 \}}{
		\{ #1 \suchthat #2 \}
	}
}
\newcommandx{\sets}[3][3 =]{
	\ifthenelse{\equal{#3}{}}{[#2]^{#1}}{[#2]^{#1}_{#3}}
}

\newcommand{\suchthat}{\mid}

\newcommandx{\union}[2][1 =, 2 =]{
	\ifthenelse{\equal{#1}{}}{\cup}{
		\ifthenelse{\equal{#2}{}}{\bigcup #1}{{\bigcup_{#1} #2}}
	}
}

\newtheorem{theorem}{Theorem}[section]
\newtheorem{lemma}[theorem]{Lemma}

\newtheorem{claim}[theorem]{Claim}

\newtheorem{proposition}[theorem]{Proposition}

\newtheorem{problem}[theorem]{Problem}

\newtheorem{definition}[theorem]{Definition}
\newtheorem{example}[theorem]{Example}
\newtheorem{remark}[theorem]{Remark}

\numberwithin{equation}{section}

\newcommand{\bd}{\begin{definition}}
	\newcommand{\ed}{\end{definition}}

\DeclareMathOperator{\dom}{dom}

\DeclareMathOperator{\dist}{dist}
\DeclareMathOperator{\didistance}{didist}
\newcommand{\mc}{\mathcal}

\newcommand{\id}{id}
\newcommand{\distance}[3]{\ifthenelse{\isempty{#3}}{\dist(#1,#2)}{\dist^{#3}(#1,#2)}}
\newcommand{\didist}[3]{\ifthenelse{\isempty{#3}}{\didistance(#1,#2)}{\didistance^{#3}(#1,#2)}}
\newcommand{\digraph}[3]{\ifthenelse{\equal{#1}{b}}{\mathbb{#2}_{#3}}
	{{#2}_{#3}}}
\newcommand{\linegraph}[3]{\ifthenelse{\equal{#1}{b}}{\mathbb{#2}_{#3}}
	{#2_{#3}}}

\newcommand{\underlyingspace}[3]{\ifthenelse{\equal{#1}{b}}{\mathbb{#2}_{#3}}
	{#2_{#3}}}
\newcommand{\distanceset}[2]{\ifthenelse{\isempty{#2}}{D(#1)}{D^{#2}(#1)}}

\pgfplotsset{soldot/.style={color=blue,only marks,mark=*}}
\usetikzlibrary{math}
\usetikzlibrary{arrows.meta}
\usetikzlibrary{shapes,decorations}
\usetikzlibrary{decorations.pathmorphing}
\usetikzlibrary{math}
\usetikzlibrary{arrows.meta}
\usetikzlibrary{shapes,decorations}
\usetikzlibrary{decorations.pathmorphing}
\usetikzlibrary{calc}
\definecolor{pastelred}{rgb}{1.0, 0.41, 0.38}
\definecolor{pastelblue}{rgb}{0.52, 0.63, 0.94}
\definecolor{pastelyellow}{rgb}{0.99, 0.99, 0.59}
\definecolor{pastelgreen}{rgb}{0.47, 0.87, 0.47}
\definecolor{pastelorange}{rgb}{1.0, 0.7, 0.28}

\newcommand{\dlocal}{\mathsf{DLOCAL}}
\newcommand{\rlocal}{\mathsf{RLOCAL}}

\newcommand{\fG}{\mathcal{G}}
\newcommand{\fB}{\mathcal{B}}

\newcommand{\fA}{\mathcal{A}}
\newcommand{\HOM}{\operatorname{\bf Hom}}

\newcommand{\Homed}{\HOM^{e}}

\newcommand{\schreier}[3]{\operatorname{Sch}(#1,#2,#3)}
\newcommand{\Cay}{\operatorname{Cay}}

\usepackage{framed}
\definecolor{shadecolor}{named}{lightgray}

\title{From descriptive to distributed}

\author{Jan Greb\'ik}

\address{Faculty of Informatics, Masaryk University, Botanicka 68A, 60200 Brno, Czech Republic}

\author{Zolt\'an Vidny\'anszky}
\address{E\"otv\"os Lor\'and University, Institute of Mathematics, P\'azm\'any P\'eter stny. 1/C, 1117 Budapest, Hungary}
\begin{document}

\begin{abstract}

    In the past couple of years a rich connection has been found between the fields of \emph{descriptive set theory} and \emph{distributed computing}. Frequently, and less surprisingly, finitary algorithms can be adopted to the infinite setting, resulting in theorems about infinite, definable graphs.     
    
    In this survey, we take a different perspective and illustrate how results and ideas from descriptive set theory provide new insights and techniques to the theory of distributed computing. We focus on the two classical topics from graph theory, \emph{vertex} and \emph{edge colorings}.
    After summarizing the up-to-date results from both areas, we discuss the adaptation of Marks' games method to the LOCAL model of distributed computing and the development of the multi-step Vizing's chain technique, which led to the construction of the first non-trivial distributed algorithms for Vizing colorings.
    We provide a list of related open problems to complement our discussion.
    Finally, we describe an efficient deterministic distributed algorithm for Brooks coloring on graphs of subexponential growth.
\end{abstract}

\maketitle

\section{Introduction}

{\bf Graph colorings} is a fundamental concept in combinatorics and theoretical computer science. Moreover, a myriad of practical problems can be also reformulated in terms of graphs and their colorability (e.g., scheduling and clustering problems, register allocation, etc.\footnote{
	For more on the historical evolution of these problems, and in particular, the amusing history of the four coloring problem see \cite{wilson2014four} or \href{https://en.wikipedia.org/wiki/Graph_coloring}{https://en.wikipedia.org/wiki/Graph\textunderscore coloring}.}), thus, these concepts have numerous applications beyond their theoretical usefulness. A basic question about graph colorings asks for the smallest number of colors needed to color the vertices (or edges) of a given graph $G$ so that adjacent vertices (or edges) have different colors.
In the case of vertex coloring this number is called the {\bf chromatic number} of $G$, denoted by $\chi(G)$, and in case of edge coloring it is called the {\bf chromatic index} of $G$, denoted by $\chi'(G)$.
Classical theorems of \emph{Brooks} and \emph{Vizing} give an upper bound on $\chi(G)$ and $\chi'(G)$, respectively, in terms of the maximum degree $\Delta(G)$ of the graph $G$.
Colorings matching these upper bounds can be found in time polynomial in the size of the graph.
In contrast, deciding (in the case of vertex colorings even approximately) the chromatic number or chromatic index of a graph is an NP-complete problem.

In this survey, we discuss algorithmic aspects of these classical problems from the perspectives of the {\bf LOCAL \emph{model of distributed computing}} and {\bf \emph{measurable graph combinatorics}}.
At first glance, these fields may seem to be at the opposite ends of the spectrum: one focuses on distributed algorithms for large but finite networks, while the other investigates uncountable graphs, and stems from pure measure theory and logic.
However, the recently discovered formal connections, originated in the work of Bernshteyn \cite{Bernshteyn2021LLL} and Elek \cite{elek2018qualitative}, and the ensuing surge of activity proved otherwise.
In fact, researchers in both areas have very often studied \emph{exactly} the same problems, and, now, there are black-box translations of results from one field to the other.
As we only briefly comment on this beautiful correspondence in Section~\ref{subsec:Correspondence}, we refer the reader to the surveys of Bernshteyn \cite{BernshteynSurvey} and Pikhurko \cite{pikhurko2021descriptive_comb_survey} for more comprehensive treatment.

What phenomenon links these fields?
At a high level, in both areas the combinatorial core of the questions of interest boils down to the understanding when a solution to local graph coloring problems can be produced \emph{efficiently in a decentralized way}.
In the seminal paper \cite{Bernshteyn2021LLL}, Bernshteyn proved that efficient distributed algorithms solving a given coloring problem on finite graphs yield automatically the existence of such a measurable coloring on measurable graphs.
As a consequence, many results from the distributed computing literature immediately yield new results in the relatively younger field of measurable graph combinatorics.
While this fundamental result follows the typical direction, namely, ideas from theoretical computer science are used to solve major problems in pure mathematics, it has also opened the gates for myriads of follow up questions. The most important one for this survey might be the following:
\begin{quote}
    \emph{Can measure theoretic techniques have a direct impact on the evolution of distributed algorithms on finite graphs?}
\end{quote}

In this paper, we aim to answer this question in the positive.
Our goal is to complement the views presented in the survey papers by Bernshteyn \cite{BernshteynSurvey} and Pikhurko \cite{pikhurko2021descriptive_comb_survey} by describing the flow of ideas from the infinite world to the finite.
The two main results from measurable graph combinatorics and their application in the theory of distributed computing that we discuss are (a) the {\bf determinacy method of Marks} for showing the \emph{inexistence} of Borel vertex colorings, and (b) {\bf measurable Vizing's theorem}, which proves the \emph{existence} of measurable edge colorings that satisfy the Vizing's bound.
\vspace{+0.2cm}

The paper is structured as follows.
Sections~\ref{sec:Classical}--\ref{sec:Summary} serve as an introduction and a comparison of graph coloring results in the classical, distributed and measurable setting.
We discuss the adaptation of Marks' technique to the LOCAL model of distributed computing in Section~\ref{sec:Marks} and the measurable and distributed analogs of Vizing's theorem in Section~\ref{sec:Vizing}. 
In Section~\ref{sec:Subexp}, we describe the efficient distributed algorithms for Brooks' coloring for graphs of subexponential growth and, finally, in Section~\ref{sec:Problems} we finish the survey with a collection of related and fascinating open problems.

\subsection*{Acknowledgments} 
We are very grateful to Jukka Soumela for letting us use Figure~\ref{fig:GlobalLOCAL}. We are also grateful to Anton Bernshteyn, Yi-Jun Chang, Clinton Conley, Alexander Kechris, Andrew Marks, Oleg Pikhurko, V\'{a}clav Rozho\v{n}, Felix Weilacher, and the anonymous referee for their comments on the first version of this manuscript.

JG was supported by MSCA Postdoctoral Fellowships 2022 HORIZON-MSCA-2022-PF-01-01 project BORCA grant agreement number 101105722. ZV was supported by Hungarian Academy of Sciences Momentum Grant no. 2022-58 and National Research, Development and Innovation Office (NKFIH) grants no.~113047, ~129211.

\section{Classical results and notation}\label{sec:Classical}

We start by fixing a standard graph theoretic notation that will be used throughout the paper. A \emph{graph} $G$ is a pair $(V,E)$, where $V$ is a set and $E\subseteq \binom{V}{2}$ is a set of unordered pairs of $V$, in particular, we do not allow multiple edges or loops. We will also use the notation $V(G)$ for the set of vertices of $G$ and $E(G)$ or just simply $G$ for the set of edges. 

To distinguish finite and countable graphs from the usually uncountable graphs coming from the measurable context we use $G$ for the former and $\fG$ for the latter.
In case $G$ is finite, we set $n:=|V|$ and $m:=|E|$.
The \emph{maximum degree of $G$} is denoted by $\Delta(G)$.
Throughout the paper we assume that $\Delta(G)\in \mathbb{N}$, in fact, most of the time we treat $\Delta(G)$ as a constant that is fixed in advance and independent of $n$.
The line graph of a graph $G=(V,E)$, denoted by $L(G)$ is defined as the pair $(E,F)$, where $\{e,e'\}\in F$ if $e\cap e'\not=\emptyset$.
We write $N_G(v)$ for the set of neighbors of $v\in V$ in $G$.
We set $\deg_G(v)=|N_G(v)|$ and $\deg_{G}(v,A)=|N_G(v)\cap A|$.
We denote as $B_G(v,t)$ the $t$-neighborhood of $v$ in $G$, that is, the restriction of $G$ to vertices of graph distance at most $t$ from $v$, that includes additional labeling that might be defined on $G$.

Standard conventions are used to denote asymptotics, that is, if $f, g: \N \to \N$ are functions, we say that $f \in o(g)$, $f \in O(g)$, $f\in \Omega(g)$ and $f \in \omega(g)$ if $\lim_{n \to \infty} f(n)/g(n)=0$, $\limsup_{n \to \infty} f(n)/g(n)<\infty$, $g \in O(f)$ and $g \in o(f)$, respectively.
Additionally, we write $f \in \Theta(g)$ if $f\in O(g)$ and $g \in O(f)$.
\vspace{+0.2cm}

For $k\in \mathbb{N}$, we write $[k]=\{1,\dots,k\}$.
A \emph{vertex coloring of $G$ (with $k\in \mathbb{N}$ colors)} or simply \emph{$k$-coloring} is a map $c:V\to [k]$ such that $c(v)\not= c(w)$ whenever $\{v,w\}\in E$.
An \emph{edge coloring of $G$} (with $\ell\in \mathbb{N}$ colors) is a map $d:E\to [\ell]$ such that $d(e)\not=d(f)$ whenever $e\not=f$ and $e\cap f \not=\emptyset$.
The {\bf chromatic number of $G$}, denoted by $\chi(G)$, is the minimum $k\in \mathbb{N}$ such that there is a vertex coloring of $G$ with $k$ colors and the {\bf chromatic index of $G$}, denoted by $\chi'(G)$, is the minimum $\ell\in \mathbb{N}$ such that there is an edge coloring of $G$ with $\ell$ colors.
Note that we have $\chi(L(G))=\chi'(G)$ by the definition.

A \emph{partial vertex coloring} (with $k\in \mathbb{N}$ colors) is a partial map $c:V\rightharpoonup [k]$ that is a vertex coloring of the graph induced from $G$ on $\dom(c)\subseteq V$.
We set $U_c=V\setminus \dom(c)$ to denote the set of uncolored vertices.
A \emph{partial edge coloring} $d:E\rightharpoonup [k]$ and the set of uncolored edges $U_d$ is defined analogously.

\subsection{Greedy algorithm}

We start with the classical greedy upper bound.

\begin{theorem}[Greedy coloring]
    Let $G$ be a graph such that $\Delta(G)\in \mathbb{N}$.
    Then $\chi(G)\le \Delta(G)+1$ and $\chi'(G)\le 2\Delta(G)-1$.    
\end{theorem}

The proof of this theorem for vertex colorings is very often presented in the form of a sequential algorithm.
Namely, after enumerating the vertices of a finite graph $G$ as $(v_1,\dots,v_{|V|})$, we go through the list inductively and color the current vertex with any of the available colors at this step.
This is always possible as we are allowed to use strictly more colors than is the maximum degree of $G$.
The argument for $\chi'(G)$ is analogous after, for example, passing to $L(G)$.

Note that the sequential algorithm runs in time $O(m)$.
In case that we treat $\Delta$ as a constant, which we usually do, the running time is $O(n)$.
It is one of the basic results in the LOCAL model, which is discussed in the next section, that there is in fact an efficient distributed algorithm producing a greedy coloring, see Theorem~\ref{thm:DistributedGreedy}.

\subsection{Brooks' and Vizing's theorems}

The greedy upper bound for vertex colorings is tight in general.
Namely, if $n\in \mathbb{N}$ and $K_n$ is the complete graph on $[n]$ vertices, then $\chi(K_n)=n$ and $\Delta(K_n)=n-1$.
A classical result of Brooks \cite{Brooks} states that apart from this situation and the case of odd cycles, $\Delta(G)$ colors are enough.

\begin{theorem}[Brooks coloring]
    Let $G$ be a graph such that $\Delta(G)\in \mathbb{N}$ and assume that $G$ does not contain a copy of $K_{\Delta(G)+1}$, or an odd cycle in case $\Delta(G)=2$.
    Then $\chi(G)\le \Delta(G)$.
\end{theorem}

Brooks coloring can be found in time $O(m)$ as was proved by Skulrattanakulchai \cite{BrooksAlg}, however, and that is one of the main results discussed in this paper, there is no efficient deterministic distributed algorithm even if we assume that the graphs that we consider have high {\bf girth}, i.e., they look locally like trees.
Recall that a classical theorem of Erd\H{o}s \cite{ErdosProbMethod}, proven by the probabilistic method, states that there exist graphs of arbitrarily high girth and chromatic number.
Recall that the \emph{girth} of a graph $G$ is the size of the shortest cycle in $G$.
More related to our investigation is the same result of Bollob\'{a}s \cite{bollobas} for $\Delta$-regular random graphs: random $\Delta$ regular graphs have girth tending to $\infty$ with positive probability as $n\to \infty$ and the chromatic number of such graphs is almost surely $\Theta(\Delta/\log(\Delta))$ as $\Delta\to \infty$.
This gives a lower bound on the number of colors needed for \emph{any} efficient distributed algorithm on locally tree-like graphs.
\vspace{+0.2cm}

In the case of edge colorings the situation is much simpler as was proved by Vizing \cite{Vizing} and independently by Gupta \cite{Gupta}.

\begin{theorem}[Vizing coloring]\label{thm:Vizing}
    Let $G$ be a graph such that $\Delta(G)\in \mathbb{N}$.
    Then $\chi'(G)\in \{\Delta(G),\Delta(G)+1\}$.    
\end{theorem}

The standard proof of Vizing's theorem, which is based on the augmenting chain technique, implies immediately that edge coloring with $\Delta(G)+1$ colors can be found in time polynomial in $m$.
Currently, the fastest deterministic algorithm for Vizing coloring for general graphs due to Sinnamon \cite{Fast} runs in time $O(m\sqrt{n})$, and the fastest randomized one due to Bhattacharya, Carmon, Costa, Solomon, and Zhang \cite{Faster} runs in time $O(mn^{1/3})$.
We refer the reader to \cite{BernshteynVizing2,Faster,BernshteynVizing3} for more details on the current development and better bounds under some additional assumptions on the maximum degree.
While it is known that there is no efficient deterministic distributed algorithm, the existence of a randomized one is one of the main open problems in the theory of distributed computing, see Section~\ref{sec:Vizing} and Problem~\ref{pr:SublogVizing}.
\vspace{+0.2cm}

We formalize the \emph{augmenting chain technique} as it plays a crucial role both in our investigation of edge colorings in Section~\ref{sec:Vizing} and in our construction of an efficient deterministic distributed algorithm for Brooks' theorem on graphs of subexponential growth in Section~\ref{sec:Subexp}.

\begin{definition}[Augmenting subgraph]
    Let $G=(V,E)$ be a graph, $d:E\rightharpoonup [k]$ be a partial edge coloring and $e\in U_d$ be an uncolored edge.
    A subgraph $H\subseteq G$ is called \emph{{\bf augmenting} for $e$ (and $d$)} if it is connected, $E(H)\setminus \{e\}\subseteq \dom(d)$ and there is a partial edge coloring $d':E\rightharpoonup [k]$ such that $\dom(d')=\dom(d)\cup\{e\}$ and $d(f)=d'(f)$ for every $f\in \dom(d)\setminus E(H)$.

    An augmenting subgraph for partial vertex colorings and an uncolored vertex is defined analogously.
\end{definition}

Vizing's theorem for finite graphs then follows from the fact that there is an augmenting subgraph for every uncolored edge $e\in E$ and any partial edge coloring of $G$ with $\Delta(G)+1$ colors.

In order to produce Vizing coloring by a distributed algorithm, or measurably, we need to control the size and structure of the augmenting subgraphs.
Already the standard proof of Vizing's theorem provides some information in this direction, the augmenting subgraph is of the form \emph{fan} and \emph{alternating path}, see Figure~\ref{fig:Vizing} in Section~\ref{sec:Vizing}.
A detailed explanation of the recent progress on this problem is done in Section~\ref{subsec:ApplicationsVizing}.
We finish this section by stating a purely graph-theoretic consequence of this line of research that was proved by Christiansen \cite{Christiansen}.

\begin{theorem}[Small augmenting subgraphs for edge colorings, Christiansen \cite{Christiansen}]\label{thm:ClassicChristiansen}
    Let $G$ be a graph, $d:E\rightharpoonup [\Delta(G)+1]$ be a partial edge coloring and $e\in U_d$ be an uncolored edge.
    Then there exists an augmenting subgraph $H$ for $e$ such that $|E(H)|\in O(\Delta(G)^7 \log(n))$.    
\end{theorem}

We remark that up to $\Delta$ factors the bound in the theorem matches the lower bound of Chang, He, Li, Pettie, and Uitto \cite{ChangHLPU20} who found an instance of partial edge coloring and an uncolored edge such that every augmenting subgraph has size $\Omega(\Delta(G) \log(n/\Delta(G)))$.

An analogous result for Brooks' coloring was proved by Panconesi and Srinivasan \cite{panconesi-srinivasan,PanconesiSrinivasan}.
Namely, every uncolored vertex in a partially vertex colored graph $G$ with $\Delta(G)\ge 3$ colors is contained in an augmenting path of length bounded by $O(\log_{\Delta(G)}(n))$.
They also demonstrated that this bound is tight.
We state this result formally in Section~\ref{sec:Subexp}, where we use it in a black-box manner in our construction of an efficient deterministic distributed algorithm for Brooks coloring on graphs of subexponential growth.

\section{LOCAL model}
\label{sec:local}

The definition of the {\bf LOCAL model} of distributed computing by Linial~\cite{linial92LOCAL} was motivated by the desire to understand \emph{distributed algorithms} in huge networks, see the book of Barenboim and Elkin \cite{barenboimelkin_book} for a historical account. 
As an example, consider a huge network of wifi routers. Let us think of two routers as connected by an edge if they are close enough to exchange messages. 
It is desirable that such close-by routers communicate with user devices on different channels to avoid interference.
In graph-theoretic language, we want to properly color the underlying network. 
Observe that producing the greedy coloring, which was easily solved by a sequential algorithm in the previous section, remains a highly interesting and non-trivial problem in the distributed setting, as, ideally, each vertex decides on its output color after only a few rounds of communication with its neighbors.

The LOCAL model of distributed computing formalizes this setup: we have a large network, where each vertex knows the size of the network $n$, and perhaps some other parameters like the maximum degree $\Delta$. In the case of randomized algorithms, each vertex has access to a private random bit string, while in the case of deterministic algorithms, each vertex is equipped with a unique identifier from a range polynomial in the size $n$ of the network. 
In one round, each vertex can exchange any message with its neighbors and can perform an arbitrary computation.
The goal is to find a solution to a given problem, in our case a vertex or edge coloring of the network, in as few communication rounds as possible. 
As the allowed message size is unbounded, a $t$-round distributed algorithm can be equivalently described as a function that maps $t$-neighborhoods to outputs---the output of a vertex is then simply the output of this function applied to the $t$-neighborhood of this vertex.
An algorithm is correct if and only if the collection of outputs at all vertices constitutes a correct solution to the problem. 

The theory of distributed algorithms is extremely rich and the landscape of possible complexities of local graph coloring problems is very well understood.
It should not come as a surprise that vertex and edge colorings play a prominent role in this endeavor.
However, it might not be immediately clear why \emph{any} non-trivial problem can be solved by a distributed algorithm.

Before we formalize the definitions, we describe a simple, yet profound, procedure that is due to Cole and Vishkin \cite{cole86}, which modifies a given vertex coloring of a graph $G$ with $k$ colors to a vertex coloring with roughly $\log(k)^{\Delta(G)}$ colors in one communication round.

\begin{example}[Color reduction algorithm, Section~3.5 in \cite{barenboimelkin_book}]\label{ex:ColorRedAlg}
    Let $G=(V,E)$ be a graph of maximum degree $\Delta\in \mathbb{N}$ and $c:V\to [k]$ be a vertex coloring.
    We describe how to produce a new vertex coloring with $(2\lceil\log(k)\rceil)^\Delta$ colors in one round of communication having $c$ as an input at every vertex of $G$.

    Let us view the color $c(v)$ as a $\{0,1\}$-string of length $\lceil\log(k)\rceil$.
    Given $v\in V$, we list for each $w\in N_G(v)$ the pair that encodes the position of the minimal bit where $c(v)$ and $c(w)$ differ together with the value of $c(v)$ at that bit.
    That is, each $v$ is assigned a string $c'(v)$ of at most $\Delta(G)$ many pairs from the set $\lceil\log(k)\rceil\times \{0,1\}$.
    We leave as an exercise to verify that $c'$ is indeed a proper vertex coloring of $G$ with $(2\lceil\log(k)\rceil)^\Delta$ colors.
\end{example}

A more complicated combinatorial argument, which is due to Linial \cite{linial92LOCAL}, decreases the number of colors to $4(\Delta+1)^2\log^2(k)$ in one round of communication, see Hirvonen and Suomela \cite[Section~4.8]{hirvonen2021distributed} or Rozho\v{n} \cite[Section 2.1.1]{rozhovn2024invitation} for a detailed explanation.

\subsection{Formal definition and basic results}

Our main references are the books of Barenboim and Elkin \cite{barenboimelkin_book} and of Hirvonen and Suomela \cite{hirvonen2021distributed}, and the recent survey of Rozho\v{n} \cite{rozhovn2024invitation}.

When talking about distributed algorithms, we always assume that there is a fixed class of finite graphs on which they operate.
Prominent examples of classes studied in the literature include all finite graphs, graphs of degree bounded by $\Delta\in \mathbb{N}$, acyclic graphs, graphs of bounded neighborhood growth, or graphs that look locally like a fixed graph, e.g., grid graphs, high girth graphs etc.

In this survey, we work, unless stated otherwise, with the class of \emph{graphs of degree bounded by $\Delta\in \mathbb{N}$}.
In particular, the dependence on $\Delta$ is hidden in the asymptotic notation.
\vspace{+0.2cm}

In the high-level overview, we mentioned that a $t$-round distributed algorithm assigns a color to a vertex based on its $t$-neighborhood and every vertex must run the same algorithm.
In particular, if there are vertices with isomorphic $t$-neighborhoods, they should receive the same color.
In order to make the model non-trivial, we need to introduce a way how to \emph{break symmetries} in graphs that are highly symmetric.
In the case of \emph{deterministic distributed algorithms} this is done by assigning a {\bf unique identifier} to each vertex from a range that is polynomial in the size of $V$.
In other words, when we talk about deterministic distributed algorithms, we always assume that the graph $G$ comes equipped with an adversarial injective map $\id:V\to n^c$, where $c\ge 1$ is a fixed constant.
In the case of \emph{randomized distributed algorithms}, each vertex generates independently uniformly an {\bf infinite sequence of random bits}.
Formally, the input graph comes equipped with an adversarial map $i:V\to \{0,1\}^\mathbb{N}$.

\begin{definition}[Distributed algorithm]
        A {\bf deterministic distributed algorithm} $\fA=(\fA_n)_n$ of complexity $(t_n)_n\in \mathbb{R}^\mathbb{N}$ is a function such that for every graph $G$ of size $n$ and every assignment of unique identifiers $\id:V\to n^c$ outputs for every $v\in V$ a color
        $$v\mapsto \fA_n(B_G(v,t_n))$$
        that only depends on the $t_n$-neighborhood of $v$, that is, on $B_G(v,t_n)$ together with $\id\upharpoonright B_G(v,t_n)$.

        A {\bf randomized distributed algorithm} is defined analogously by replacing an assignment of unique identifiers by an assignment of infinite sequences of random bits.        
\end{definition}

Note that we implicitly assume that $n$, the size of the vertex set of $G$, is part of the input of $\fA$.
In particular, the functions in the sequence $\fA=(\fA_n)_n$ can be completely independent.
We also remark that under a mild assumption on the class of graphs under consideration, the constant $c\ge 1$ can be freely thought of as being equal to $1$.
For instance, it is trivial to check that if there is an algorithm for $c\ge 1$, then there is one for any $c\ge c'\ge 1$, and if the class of graphs is closed under enlarging the vertex set the same holds for any $c'\ge c$. 

As it is always clear from the context whether we talk about deterministic or randomized algorithms, we suppress the mentioning of unique identifiers or random bits and simply say that we apply $\fA$ on $G$.
We emphasize again that when constructing distributed algorithms one should think of these assignments as being given to us by an adversary and our algorithm should perform the given task successfully for every such assignment in the case of deterministic algorithms, and with high-probability in case of randomized algorithms.
\vspace{+0.2cm}

As our main goal is to discuss graph coloring problems and to keep the notation simple, we formalize the definition of \emph{LOCAL complexity} for vertex colorings only.
The definition of LOCAL complexity for edge colorings, or more generally for the so-called \emph{locally checkable labeling (LCL) problems} that are discussed in Section~\ref{subsec:GlobalLOCAL}, can be done completely analogously. 
We also remind the reader that we implicitly work with the class of graphs of degree bounded by $\Delta\in \mathbb{N}$, but the definition can be generalized to any class of finite graphs.
 
	\begin{definition}[LOCAL complexity of graph colorings]
		\label{def:local_complexity}
		The \emph{vertex coloring problem with $k\in \mathbb{N}$} colors has a {\bf deterministic LOCAL complexity} $(t_n)_n$ if there is a deterministic distributed algorithm $\fA=(\fA_n)_n$ of complexity $(t_n)_n$ such that for every graph $G$ the output coloring
        $$v\mapsto \fA_n(B_G(v,t_n))$$
        produced by $\fA=(\fA_n)_n$ is a vertex coloring of $G$ with $k$ colors,
		
		The {\bf randomized LOCAL complexity} is defined analogously, with the main difference being that we demand that the output coloring is a vertex coloring of $G$ with $k$ colors with probability at least $1 - 1/n^d$, where $d>0$ is a fixed constant. 
	\end{definition}

Note that the deterministic complexity is an upper bound on the randomized complexity as interpreting the first $\lfloor \log(n^{c})\rfloor$ bits as an element of the space of unique identifiers $n^{c}$ gives an injective map with probability at least $1-\frac{1}{n}$ provided that $c\ge 3$.
\vspace{+0.2cm}

In Example~\ref{ex:ColorRedAlg}, we discussed how to decrease the number of colors of a vertex coloring in one round of communication.
Once we are familiar with this one step procedure, nothing can stop us from iterating it.
For example, if $G$ is a $2$-regular graph and we start with a vertex coloring with $k=2^{2^{2^{2}}}\approx 2\cdot 10^{19729}$ many colors, then this procedure decreases the number of colors in $4$ rounds of communication as follows 
$$2\cdot 10^{19729}\mapsto 17179869184\mapsto 4624\mapsto 676\mapsto 400.$$
It is not coincidental that it took us $4$ steps to reduce a $4$-times iterated exponential to a reasonably small number.
A general argument (using the better bound discussed after Example~\ref{ex:ColorRedAlg} and treating $\Delta$ as a constant) shows that starting from vertex coloring with $n^c$ colors, where $c>0$ is a constant, we arrive to a coloring with $O(1)$ colors after at most $O(\log^*(n))$ many communication rounds.\footnote{Recall that $\log^*(n)$ is defined as the minimum number of times we need to apply $\log$ to decrease $n$ to a value that is at most $2$.}
A particular instance of a starting coloring might be the assignment of unique identifiers that is given to us in the setting of deterministic algorithms.
Combining these ideas allows to prove the upper bound in the following fundamental result.
Note that the last step, that is, getting from constant colors to $\Delta(G)+1$, can be done by a simple distributed greedy algorithm by iterating over the constantly many color classes and simulating the sequential greedy algorithm locally, see again \cite[Section~4.8]{hirvonen2021distributed} and \cite[Section 2.1.1]{rozhovn2024invitation} for more details.
We also refer the reader to \cite[Section~3.6-3.10]{barenboimelkin_book} for a more detailed treatment of the dependency on $\Delta$.

\begin{theorem}[Distributed Greedy coloring\footnote{This is typically referred to as $\Delta$+1-coloring, rather than ``greedy".}, Cole--Vishkin \cite{cole86}, Goldberg--Plotkin--Shannon \cite{goldberg88}, Linial \cite{linial92LOCAL}]\label{thm:DistributedGreedy}
    Let $\Delta\in \mathbb{N}$.
    \begin{itemize}
        \item There is a deterministic distributed algorithm of complexity $O(\log^*(n))$ that produces a vertex coloring with $\Delta+1$ colors on the class of graphs of degree bounded by $\Delta$.
        \item There is a deterministic distributed algorithm of complexity $O(\log^*(n))$ that produces an edge coloring with $2\Delta-1$ colors on the class of graphs of degree bounded by $\Delta$.
    \end{itemize}
    Moreover, in both cases $\Omega(\log^* (n))$ rounds are necessary.
\end{theorem}

For the proof of the lower bound, which is originally due to Linial \cite{linial92LOCAL}, together with a description of randomized distributed algorithm for these colorings that work \emph{without knowledge of the size of the graph}, we refer the reader to the paper of Holroyd, Schramm and Wilson \cite[Sections 1 and 2]{HolroydSchrammWilson2017FinitaryColoring}, where the connection to \emph{random processes}, in particular, to \emph{finitary factors of iid}, is discussed.
\vspace{+0.2cm}

In the light of Theorem~\ref{thm:DistributedGreedy}, we are compelled to raise the following questions.
Can we decrease the number of colors?
Are there similarly efficient distributed algorithms for Brooks or Vizing colorings?
Does randomness give us any advantage?

In order to answer these questions, it is useful to discuss the global picture of possible LOCAL complexities.

\subsection{Global picture of the LOCAL model}\label{subsec:GlobalLOCAL}
The landscape of possible LOCAL complexities of graph coloring problems where the correctness of a solution can be checked locally, so-called \emph{locally checkable labeling (LCL) problems} (see, e.g., \cite[Section 2.4.2]{rozhovn2024invitation} for a formal definition),\footnote{Apart from graph colorings, the class of LCL problems include for instance the perfect matching problem, the sinkless orientation problem and various versions of the unfriendly coloring problem.} is very well understood.
This understanding has been a major success of several recent lines of work on local algorithms, including \cite{naorstockmeyer, brandt19automatic,chang_kopelowitz_pettie2019exp_separation,chang2019time, ghaffari2017complexity, ghaffari2018derandomizing,balliu2018new_classes-loglog*-log*, RozhonG19, ghaffari2019distributed, ghaffari_grunau_rozhon2020improved_network_decomposition,chang2020n1k_speedups,balliu2020almost_global_problems,brandt_grids,brandt_grunau_rozhon2021classification_of_lcls_trees_and_grids} and many others.
For the class of graphs with degree bounded by $\Delta$, a clean picture emerges, see Figure~\ref{fig:GlobalLOCAL}.\footnote{The picture is slightly simplified, in particular, $\log(n)$ and $\log(\log (n))$ complexities are only conjectured \cite{chang2019time}, the best up-to-date general upper bounds in the corresponding regimes are $O(\log^3(n))$ and $O(\log^3(\log (n)))$.}

The two most interesting features of this classification is that (i) there is only one regime when randomness helps, (ii) there are automatic \emph{speed-up} results that often turn a proposed distributed algorithm into a faster one in a completely black-box manner.
In particular, item (ii) automatically answers question of the form, is there an LCL problem that has (deterministic) LOCAL complexity $\Theta(\sqrt{\log(n)})$?

The fundamental result of Bernshteyn \cite{Bernshteyn2021LLL}, mentioned in the introduction, turns colorings produced by efficient distributed algorithms into measurable colorings.
Here, \emph{efficient distributed algorithms} refers to the $o(\log(n))$ regime, see Figure~\ref{fig:GlobalLOCAL}.
In particular, one needs to specify if randomness is allowed as this is the regime, where it gives an advantage.
\vspace{+0.2cm}

\begin{figure}
    \centering
        \includegraphics[scale=0.75]{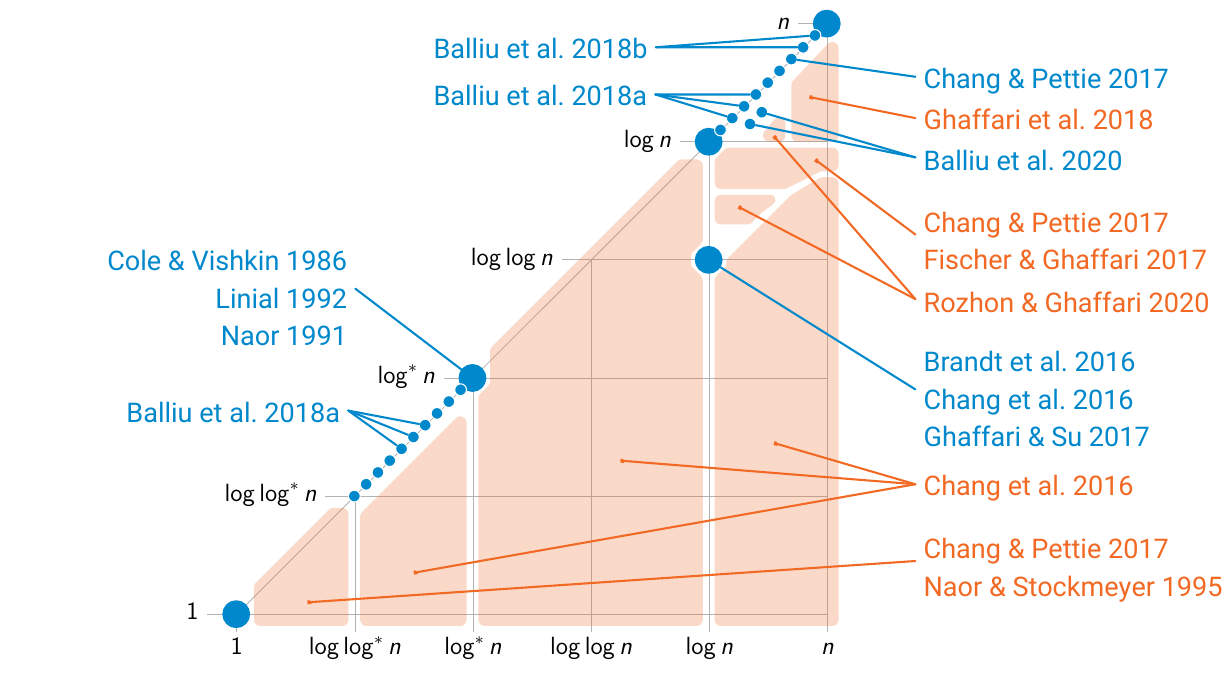}
		\caption{(Courtesy of Jukka Soumela)
        The blue dots represent classes of LCL problems (size of the dots suggest the importance of the class), that is, every LCL problem, e.g., vertex coloring with $\Delta$ colors, belongs to one of the dots.
                \emph{Deterministic LOCAL complexity} is given by the projection of the dot to the horizontal axis and \emph{randomized LOCAL complexity} is given by the projection to the diagonal.
                Light salmon colored areas do not contain any complexity class of LCL problems, in particular, they represent the \emph{speed-up} results.
                }
		\label{fig:GlobalLOCAL}
\end{figure}

To understand the LOCAL complexities of our graph coloring problems, we need to locate the corresponding blue dots in Figure~\ref{fig:GlobalLOCAL}.
We have seen in Theorem~\ref{thm:DistributedGreedy} that the LOCAL complexity of \emph{greedy coloring}, both deterministic and randomized, is $\Theta(\log^* (n))$, the class of LCL problems with this LOCAL complexity is sometimes referred to as the class of \emph{symmetry breaking problems}.

Deterministic lower bounds $\Omega(\log(n))$ for Brooks and strong version of Vizing coloring are fundamental results in the theory of the LOCAL model.
The original argument uses the \emph{round elimination method} of Brandt, see \cite{brandt19automatic} for more details.
In Figure~\ref{fig:GlobalLOCAL}, we see that this automatically implies the randomized lower bound $\Omega(\log(\log(n)))$.
Rather surprisingly the lower bounds also follow from the groundbreaking work of Marks \cite{DetMarks} in Borel combinatorics by applying Bernshteyn's correspondence \cite{Bernshteyn2021LLL}.
We remark that Marks' argument involves infinite games and one of the deepest results in mathematical logic, the Borel determinacy theorem of Martin \cite{martin}.
In Section~\ref{sec:Marks}, we show how to adapt Marks' game technique directly in the LOCAL model and get an alternative and very simple proof of these fundamental lower bounds based on determinacy of finite games.
The following theorem should be compared with Theorem~\ref{thm:DetMarks}.

\begin{theorem}[Distributed lower bounds for colorings, \cite{brandt_etal2016LLL,ChangHLPU20}]\label{thm:LowerBoundsDistributed}
    Let $\Delta\ge 3$.
    \begin{itemize}
        \item Vertex coloring with $\Delta$ colors on the class of finite graphs of degree bounded by $\Delta$ that do not contain $K_{\Delta+1}$ needs $\Omega(\log(n))$ rounds in the deterministic LOCAL model.
        \item Edge coloring with $2\Delta-2$ colors on the class of finite graphs of degree bounded by $\Delta$ needs $\Omega(\log(n))$ rounds in the deterministic LOCAL model.
    \end{itemize}
\end{theorem}

\subsection{Upper bound for Brooks and Vizing colorings}
Finding optimal distributed algorithms for graph colorings and decreasing the number of colors, ideally matching the Brooks and Vizing bounds, has been a major challenge in the theory of distributed computing.
In this section, we give a brief overview of the up-to-date results.
Recall that for simplicity we treat $\Delta$ as a constant, see the survey of Rozho\v{n} \cite[Section~1.6]{rozhovn2024invitation} for references, results and open problems in the regime when the complexity depends on $\Delta$ as well.
\vspace{+0.2cm}

The local complexity, both deterministic and randomized, of the greedy vertex coloring is $\Theta(\log^*(n))$ by Theorem~\ref{thm:DistributedGreedy}.
The first non-trivial result concerning distributed Brooks colorings, that remained the state of the art for more than two decades, was achieved by Panconesi and Srinivasan \cite{panconesi-srinivasan,PanconesiSrinivasan}.
The combinatorial core of their algorithm is based on the augmenting graph result for Brooks' theorem that is discussed after Theorem~\ref{thm:ClassicChristiansen}, and yields both a deterministic and a randomized distributed algorithm of local complexity $O(\log^3(n))$.

In the special case of trees of degree bounded by $\Delta\ge 3$, Barenboim and Elkin \cite{BarenboimElkinPaper} described a deterministic distributed algorithm of local complexity $O(\log(n))$ and Chang, Kopelowitz and Pettie \cite{chang_kopelowitz_pettie2019exp_separation} gave a randomized one of local complexity $O(\log(\log(n)))$, when $\Delta$ is sufficiently large.
Importantly, these upper bounds exactly match the lower bounds of Brandt, Fischer, Hirvonen, Keller, Lempi{\"a}inen, Rybicki, Suomela and Uitto \cite{brandt_etal2016LLL} that were proved for the class of graphs of girth $\Omega(\log(n))$, in particular, they also apply to the class of finite trees.

The following are the up-to-date results of Ghaffari, Hirvonen, Kuhn and Maus \cite{ghaffari2020Delta_coloring}, see also \cite{FastDistributedBrooksTheorem}.
Note that these results are only one logarithmic factor away from the lower bounds.

\begin{theorem}[Distributed Brooks coloring, \cite{ghaffari2020Delta_coloring}]\label{thm:DistributedBrooks}
    Let $\Delta\ge 3$.
    \begin{itemize}
        \item Vertex coloring with $\Delta$ colors on the class of finite graphs of degree bounded by $\Delta$ that do not contain $K_{\Delta+1}$ has deterministic local complexity $O(\log^2(n))$.
        \item Vertex coloring with $\Delta$ colors on the class of finite graphs of degree bounded by $\Delta$ that do not contain $K_{\Delta+1}$ has randomized local complexity $O(\log^2(\log(n)))$.
    \end{itemize}
\end{theorem}

For results that concern decreasing the number of colors even further, we refer the reader to the papers of Bamas and Esperet \cite{BamasEsperet} for matching optimal coloring and Chung, Pettie, and Su \cite{chung2017LLL} for colorings of graphs of girth at least $3$, $4$ or $5$.
\vspace{+0.2cm}

Applying the vertex coloring results to the line graph of a given graph $G$ implies immediately the same upper bounds for edge colorings of $G$.
Namely, the greedy edge coloring with $2\Delta-1$ colors have a local complexity $\Theta(\log^*(n))$.
Similarly, the same bounds as in Theorem~\ref{thm:DistributedBrooks} apply to edge colorings with $2\Delta-2$ colors.
Chang, He, Li, Pettie, and Uitto \cite{ChangHLPU20} showed that edge coloring with $2\Delta-2$ colors have a deterministic lower bound $\Omega(\log(n))$ even when restricted to the class of finite trees of degree bounded by $\Delta$.
In the same paper \cite{ChangHLPU20}, the authors developed a randomized algorithm of local complexity $O(\operatorname{poly}(\log(\log(n)))$ for edge coloring with $\Delta+\sqrt{\Delta}\operatorname{poly}(\log(n))$ colors.

Getting closer to the Vizing bound, Su and Vu \cite{SuVu} deviced a randomized distributed algorithm for edge coloring with $\Delta+2$ colors of local complexity $\operatorname{poly} (\log(n))$.
Finally, Bernshteyn \cite{BernshteynVizing} matched the Vizing bound by both deterministic and randomized distributed algorithms of local complexity $\operatorname{poly} (\log(n))$.
This was subsequently improved by Christiansen \cite{Christiansen}, and Bernstehyn and Dhawan \cite{BernshteynVizing2} to the current state of art, see Section~\ref{sec:Vizing} for more details.

\begin{theorem}[Distributed Vizing coloring, \cite{BernshteynVizing2}]\label{thm:DistributedVizing}
    Let $\Delta\in \mathbb{N}$.
    \begin{itemize}
        \item Edge coloring with $\Delta+1$ colors on the class of finite graphs of degree bounded by $\Delta$ has deterministic local complexity $O(\operatorname{poly}(\log(\log(n)))\log^5(n))$.
        \item Edge coloring with $\Delta+1$ colors on the class of finite graphs of degree bounded by $\Delta$ has randomized local complexity $O(\operatorname{poly}(\log(\log(n)))\log^2(n))$.
    \end{itemize}    
\end{theorem}

\section{Measurable combinatorics}
\label{sec:measurable}

	Generalizing concepts from finite graph theory to uncountable graphs in the most straightforward way often results in counter-intuitive behavior: the most spectacular example of this is probably the Banach-Tarski paradox. It states that a three-dimensional ball of unit volume can be decomposed into finitely many pieces that can be moved by isometries (distance preserving transformations such as rotations and translations) to form two three-dimensional balls each of them with unit volume(!).
	The graph theoretic problem lurking behind this paradox is the following: fix finitely many isometries of $\mathbb{R}^3$ and then consider a graph where $x$ and $y$ are connected if there is an isometry that sends $x$ to $y$.
	Then our task becomes to find a perfect matching in the appropriate subgraph of this graph -- namely, the bipartite subgraph where one partition contains points of the first ball and the other contains points of the other two balls. 
	Banach and Tarski have shown that, with a suitably chosen set of isometries, the Axiom of Choice implies the existence of such a matching. 
	In contrast, since isometries preserve the Lebesgue measure, the pieces in the decomposition cannot be Lebesgue measurable. Such observations lead to the following heuristic: if we want to use (uncountably) infinite graphs to model the behavior of large finite graphs, we must impose measurability constraints on the objects we are investigating.  
	
	Recently, results in this direction that lie on the border of combinatorics, logic, group theory, and ergodic theory led to an emergence of a new field often called \textit{descriptive} or \textit{measurable combinatorics}. 
	The field focuses on the connection between the discrete and continuous and is largely concerned with the investigation of graph-theoretic concepts. See \cite{laczk,doughertyforeman, marksunger,measurablesquare,marks2016baire,gaboriau,KST,DetMarks,millerreducibility,AsymptoticDim,csokagrabowski,Bernshteyn2021LLL} for some of the most important results, and \cite{kechris_marks2016descriptive_comb_survey, pikhurko2021descriptive_comb_survey,marks2022measurable} for surveys.
	
	The usual setup in descriptive combinatorics is that we have a graph with uncountably many connected components, each being a countable graph of bounded degree. 
	For example, in case of the Banach-Tarski paradox, the vertices of the underlying graph are the points of the three balls, edges correspond to isometries, and the degree of each vertex is bounded by the number of chosen isometries. 
    \vspace{+0.2cm}
    
    Before we formally introduce Borel graphs and their colorings, we illustrate the subject on a classical -and impossible to omit- example from ergodic theory that simulates nicely the type of behavior that one should expect when dealing with measurable sets, and how such behaviors can be dramatically different from the classical case.

\begin{example}[Irrational rotations]\label{ex:IrrationalRotation}
Let $\alpha \in [0,\pi]$ be such that $\frac{\alpha}{\pi}$ is irrational and denote by $T_\alpha$ the rotation of the circle $\mathbb{S}^1$ by $\alpha$.
Define the graph $\fG_\alpha$ with vertex set $\mathbb{S}^1$ as follows.
We say that $x,y\in \mathbb{S}^1$ form an edge in $\fG_\alpha$ if $T_\alpha(x)=y$ or $T_\alpha(y)=x$.

Observe that $\fG_\alpha$ is $2$-regular and, as $\frac{\alpha}{\pi}$ is irrational, acyclic.
In particular, there are uncountably many connectivity components in $\fG_\alpha$ each isomorphic to the canonical graph on $\mathbb{Z}$.
As each of these connectivity components has chromatic number $2$, it follows that $\chi(\fG_\alpha)=2$.
In other words, we can write $\mathbb{S}^1=B_0\sqcup B_1$ such that $T_\alpha(B_0)=B_1$, $T_\alpha(B_0)\cap B_0=\emptyset$ and $T_\alpha(B_1)\cap B_1=\emptyset$.

A typical question in measurable combinatorics asks if the sets $B_0$ and $B_1$ can be Lebesgue measurable, that is, if the \emph{measurable chromatic number of $\fG_\alpha$} is equal to $2$ as well.
We demonstrate that this is \emph{not} possible.
Suppose for a contradiction that $B_0$ and $B_1$ are Lebesgue measurable.
Then by the properties above combined with the fact that $T_\alpha$ preserves measure, we conclude that $\lambda(B_0)=\lambda(B_1)=1/2$.
By the Lebesgue density point theorem, there is a nonempty open interval $U$ with the property that $\lambda(U\setminus B_0)<\frac{1}{3}\lambda(U)$.
But then as $2\alpha/\pi$ is also irrational, there is an odd $n\in\mathbb{N}$ such that $\lambda(T^n_{\alpha}(U) \cap U)>\frac{2}{3}\lambda(U)$, where $T^n_\alpha$ denotes the application of $T_\alpha$ $n$-times.
Note that $T^n_\alpha(B_0)=B_1$ as this holds for every odd $n\in \mathbb{N}$.
An easy volume argument in $U$ then implies that $B_1\cap B_0=T^n(B_0) \cap B_0 \neq \emptyset$, a contradiction.

In a stark contrast, a vertex coloring with $3$ colors can be achieved with Lebesgue measurable sets, see Example~\ref{ex:BasicBorelColoring}.
\end{example}

\subsection{Definitions and basic results}

The connection between the local and global structure of the graphs that are studied in measurable combinatorics is formalized through classical descriptive set theoretic notions, see \cite{kechrisclassical} for introduction to the subject.
Namely, the vertex set of a graph comes equipped with a \emph{$\sigma$-algebra} of measurable, or definable, subsets.
The structure and all the operations on the graphs that we are interested in, for example, graph colorings or partitions, have to be measurable with respect to this $\sigma$-algebra or its powers.

To avoid any pathological behavior, we always assume that the $\sigma$-algebra turns the vertex set into a {\bf standard Borel space}.
Formally, a standard Borel space is a pair $(X,\fB)$, where $X$ is a set and $\fB$ is a $\sigma$-algebra of subsets of $X$ that coincide with a $\sigma$-algebra of Borel sets for some complete separable metric on $X$.
Slightly abusing the notation, we refer to elements of $\fB$ as \emph{Borel sets}.
By the \emph{Borel isomorphism theorem}, all uncountable standard Borel spaces are Borel isomorphic, so one may safely assume that the vertex set of every Borel graph in this paper is, for example, the unit interval $[0,1]$ endowed with the $\sigma$-algebra of Borel sets.
If $(X,\fB)$ is a standard Borel space, then the product $X\times X$ and the set of all unordered pairs $\binom{X}{2}$ is naturally a standard Borel space.

\begin{definition}
    A {\bf Borel graph $\fG$} is a triple $(V,E,\fB)$, where $(V,\fB)$ is a standard Borel space, $(V,E)$ is a graph and $E\subseteq \binom{V}{2}$ is a Borel subset of the standard Borel space $\binom{V}{2}$ of unordered pairs of $V$.
\end{definition}

Borel graphs often appear in the nature.
The most relevant examples for our purposes are so called Schreier graphs of group actions.
These are frequently investigated in the theory of random processes, measured group theory, ergodic theory and dynamics. 

\begin{example}[The Schreier graph $\schreier{\Gamma}{S}{X}$]\label{ex:Schreier}
Let $\Gamma$ be a countable group and $S \subseteq \Gamma$ be a generating set of $\Gamma$.
Assume that $\Gamma\curvearrowright^\cdot X$ is an action of $\Gamma$ on a set $X$.
The \emph{Schreier graph} $\schreier{\Gamma}{S}{X}$ of the action is a graph on the set $X$, where $x \neq x'$ form an edge if there is $\gamma \in S$ such that $\gamma\cdot x=x'$ or $\gamma\cdot x'=x$.

If $(X,\mathcal{B})$ is a standard Borel space and the action $\Gamma\curvearrowright^\cdot X$ is Borel, meaning that $\gamma\cdot\_:X\to X$ is a Borel isomorphism for every $\gamma\in \Gamma$, then $\schreier{\Gamma}{S}{X}$ is a Borel graph.
\end{example}

The simplest example of a Schreier graph is the (right) \emph{Cayley graph of $\Gamma$ with respect to $S$}, denoted by $\Cay(\Gamma,S)$, that is, the Schreier graph that comes from the right multiplication action of $\Gamma$ on itself. 
Observe that the graph in Example~\ref{ex:IrrationalRotation} is of the form $\schreier{\Gamma}{S}{X}$, where $\Gamma=\mathbb{Z}$, $S=\{1\}$ and $X=\mathbb{S}^1$.

Note that if $S$ is finite, then the degree of $\schreier{\Gamma}{S}{X}$ is bounded by $2|S|$.
In particular, the degree is uniformly bounded.
Unless explicitly stated otherwise, we are exclusively concerned with graphs with uniformly bounded degree.

\begin{definition}
    Let $\mathcal{G}=(V,E,\fB)$ be a Borel graph.
    The \emph{Borel chromatic number of $\mathcal{G}$}, denoted by $\chi_\fB(\fG)$, is the minimal $k \in \mathbb{N}$ such that there is $c:V\to [k]$ that is a vertex coloring with $k$ colors and $c^{-1}(\{i\})\in \fB$ for every $i\in [k]$.
    In other words, there is a partition $V=B_1\sqcup \dots,\sqcup B_k$ into Borel $\fG$-independent sets, where $X\subseteq V$ is called \emph{$\fG$-independent} if it does not span any edge in $\fG$.
    
    Analogously, we can define the \emph{Borel chromatic index of $\fG$} and denote it as $\chi'_\fB(\mc{G})$.
\end{definition}

Is there an analogue of greedy colorings in the Borel context?
As we most of the time implicitly assume that the underlying vertex set $(V,\fB)$ is uncountable, it is a priori not even clear that there  is a countable decomposition $V=X_1\sqcup X_2\sqcup \dots$ into $\fG$-independent Borel sets.
Let us remark that the main result of the seminal paper of Kechris-Solecki-Todor\v{c}evi\'c \cite{KST}, which initiated the study of Borel chromatic numbers, was to give a complete characterization of the Borel graphs that admit such a countable decomposition.
In particular, one of their results states that in the bounded degree case such a decomposition always exists.
Compare the following example with Example~\ref{ex:ColorRedAlg}.

\begin{example}[Countable colorings from finite degrees]\label{ex:BasicBorelColoring}
    Let $\fG=(V,E,\fB)$ be a Borel graph that is locally finite, that is, $\deg_\fG(v)\in \mathbb{N}$ for every $v\in V$.
    By the Borel isomorphism theorem, we may assume that $V=2^\mathbb{N}$, that is, every $v\in V$ can be represented by an infinite sequence of $0$s and $1$s.
    The crucial observation is that for every $v\in V$ and a finite set $A\subseteq V\setminus \{v\}$, there is $\ell\in \mathbb{N}$ such that $v\upharpoonright [\ell]\not=x\upharpoonright [\ell]$ for every $x\in A$.
    Hence, for every $v\in V$ we may set $c'(v)=v\upharpoonright [k]$, where $k\in \mathbb{N}$ is the minimal number that have this property with respect to the finite set $N_\fG(v)$.
    It is a matter of routine work with Boolean operations to check that the map $c':V\to 2^{<\mathbb{N}}$, where $2^{<\mathbb{N}}$ is the set of \emph{finite} sequences of $0$s and $1$s, is a Borel map that has the property that $c'(v)\not=c'(w)$ for every $\{v,w\}\in E$. In particular, composing $c'$ with \emph{any} bijection between the countable sets $2^{<\mathbb{N}}$ and $\mathbb{N}$ defines a Borel map $c:V\to \mathbb{N}$ that is a vertex coloring with $\mathbb{N}$ colors, that is,     
    $$V=\bigsqcup_{i\in \mathbb{N}}c^{-1}(\{i\})$$
    is a decomposition into Borel $\fG$-independent sets. 
\end{example}

Example~\ref{ex:BasicBorelColoring} serves as a basic symmetry breaking device in the measurable context that allows to simulate local algorithms.
It is an interesting philosophical problem if the computational power of Borel constructions consists merely of transfinite iteration of this process possibly on increasing graph powers of the underlying graph.

Similarly as in the proof of the upper bound from Theorem~\ref{thm:DistributedGreedy}, it is possible to modify the countable coloring above into a greedy coloring by simulating locally the sequential greedy algorithm.

\begin{theorem}[Borel greedy colorings, \cite{KST}]\label{thm:BorelGreedy}
    Let $\fG=(V,E,\fB)$ be a Borel graph such that $\Delta(\fG)\in \mathbb{N}$.
    Then we have $\chi_\fB(\mathcal{G})\leq \Delta(\fG)+1$ and $\chi'_\fB(\mathcal{G})\leq 2\Delta(\fG)-1$.
\end{theorem}

In a groundbreaking paper \cite{DetMarks}, Marks proved that the Borel greedy colorings are optimal even for Borel graphs that are regular and acyclic.
The proof uses one of the deepest theorems in descriptive set theory, the determinacy of two player games with a Borel pay-off set, proven by Martin \cite{martin}.

\begin{theorem}[Borel lower bounds for colorings, \cite{DetMarks}]\label{thm:DetMarks}
    Let $\Delta>2$.
    \begin{itemize}
        \item There is an acyclic $\Delta$-regular Borel graph $\fG$ that does not admit Borel vertex coloring with $\Delta$ colors.
        \item There is an acyclic $\Delta$-regular Borel graph $\fG$ that does not admit Borel edge coloring with $2\Delta-2$ colors.
    \end{itemize}
\end{theorem}

One should compare the above statement with Theorem \ref{thm:LowerBoundsDistributed}. We will discuss the proof of this result in some detail in Section \ref{sec:Marks}.

\subsection{Measure, as the analogue of randomness}
In ergodic theory, measured group theory or the theory of random processes on countable graphs, people study combinatorial problems on Borel graphs that are equipped with an additional structure, \emph{Borel probability measure} that satisfies the \emph{mass transport principle}.
This concept is also known under the names \emph{pmp (probability measure preserving) action}, \emph{graphing} or \emph{unimodular random network}.

Formally, let $\fG=(V,E,\fB)$ be a Borel graph and $\mu$ be a Borel probability measure on $(V,\fB)$.
We say that $\mu$ satisfies the \emph{mass transport principle} if
\begin{equation}\label{eq:MassTransport}
    \int_{A}\deg_{\fG}(x,B) \ d\mu=\int_{B}\deg_{\fG}(x,A) \ d\mu
\end{equation}
for every $A,B\in \fB$.
That is, counting edges from a set $A$ to $B$ and from $B$ to $A$ is the same.

Typical examples of Borel graphs that satisfies the mass transport principle are the Schreier graphs, introduced in Example~\ref{ex:Schreier}, under the additional assumption that the Borel action preserves a given Borel probability measure $\mu$.
We complement this family of examples with examples coming from the theory of random processes.

\begin{example}[Examples of mass transport]
    \emph{(a) iid graphings:}
    Let $\Gamma$ be a countable group and $S \subseteq \Gamma$ be a finite \emph{symmetric} generating set of $\Gamma$.
    The vertex set of the iid graphing $\fG_\Gamma$ consists of all maps $i:\Gamma\to\{0,1\}^\mathbb{N}$, and two such maps $i\not=j$ form an edge if there is $s\in S$, such that
    $$i(s^{-1}\cdot \gamma)=j(\gamma)$$
    for every $\gamma\in \Gamma$.
    If we endow the vertex set with the iid power of the coin-flip measure, we get that almost surely every connectivity component of $\fG_\Gamma$ is isomorphic to $\Cay(\Gamma,S)$.
    It is easy to check that $\fG_\Gamma$ is indeed a graphing.
    \vspace{+0.2cm}

    \emph{(b) percolation graph on $\mathbb{Z}^d$} (see \cite{aldouslyons2007} for an introduction to the subject):
    A more complicated example of a graphing can be described using the notion of \emph{site} (or vertex) \emph{percolation}.
    Consider the iid graphing $\fG_{\mathbb{Z}^d}$ and fix $p\in [0,1]$ with the property that the site percolation on $\mathbb{Z}^d$ contains infinite cluster with positive probability.
    Define a subset of the vertex set $A$ to consist of those functions $i:\mathbb{Z}^d\to\{0,1\}^\mathbb{N}$ such the the origin belongs to an infinite cluster.
    Define $\fG_{\mathbb{Z}^d,p}$ to be the graph induced in $A$.
    It can be again easily shown that, after re-normalizing the measure, $\fG_{\mathbb{Z}^d,p}$ is indeed a graphing.
    Observe that $\fG_{\mathbb{Z}^d,1}=\fG_{\mathbb{Z}^d}$.
\end{example}

The mass transport principle translates some of the intuition about finite graphs into the measurable setup.
In the following definition of a \emph{measured graph}, we work in bigger generality and do not require any condition on the Borel probability measure.

\begin{definition}[Measured graph and measurable chromatic numbers]

A \emph{measured graph} is a quadruple $\mathcal{G}=(V,E,\fB,\mu)$, where $(V,E,\fB)$ is a Borel graph and $\mu$ is a Borel probability measure on $(V,\fB)$.
A measured graph $\fG$ is called a \emph{graphing} if $\mu$ satisfies the mass transport principle.
\vspace{+0.2cm}

The \emph{$\mu$-measurable chromatic number of $\fG$}, denoted by $\chi_\mu(\fG)$ is the minimal $k\in\mathbb{N}$ such that there is a $\mu$-null set $N\in \fB$ and a decomposition $V=N\sqcup B_1\sqcup\dots\sqcup B_k$ such that $B_i$ is a Borel $\mathcal{G}$-independent set for every $i\in [k]$.
In other words, the Borel chromatic number is equal to $k$ off of a $\mu$-null set.

The \emph{$\mu$-measurable chromatic index of $\fG$} is defined analogously and is denoted by $\chi'_\mu(\fG)$.    
\end{definition}

It follows immediately that $\chi_\mu(\fG)\le \chi_\fB(\fG)$ for every measured graph $\mathcal{G}=(V,E,\fB,\mu)$, where we abuse the notation and define $\chi_\fB(\fG)$ to be the Borel chromatic number of the underlying Borel graph of $(V,E,\fB)$ of $\fG$.
We remark that it is one of the most intriguing problems in the area to find examples of local graph coloring problem that could be solved on graphings but not on measured graphs in general.

Does ignoring events that happen on null sets help with graph colorings, in particular, can the inequality $\chi_\mu(\fG)\le \chi_\fB(\fG)$ be strict?
Does the analogy between deterministic and randomized algorithms, and Borel and measurable colorings go deeper?
The answer to the first question is clear yes as we demonstrate next.
The answer to the second question is more complicated, see Remark~\ref{rem:HigherRandom}.
\vspace{+0.2cm}

Investigating measurable chromatic numbers on bounded degree Borel graphs has a rich history that connects the theory of random graphs, graph limits, random processes and measured group theory, see the survey of Kechris and Marks \cite[Section~6]{kechris_marks2016descriptive_comb_survey} for many results in this direction.
An illustrative example of this rich connection is the general lower bound $\Omega(\Delta/\log(\Delta))$ on measurable chromatic numbers of acyclic $\Delta$-regular graphs (where $\Delta$ is large) that follows from the seminal result of Bollob\'as \cite{bollobas} on the size of maximal independent set in random regular graphs combined with the theory of \emph{local-global} convergence developed by Hatami, Lov\'asz and Szegedy \cite{hatamilovaszszegedy}.

Complementing Theorem~\ref{thm:DistributedBrooks}, the measurable analogue of Brooks theorem was proved by Conley, Marks and Tucker-Drob \cite{conley2016brooks}.

\begin{theorem}[Measurable Brooks coloring, \cite{conley2016brooks}]\label{thm:MeasureBrooks}
    Let $\Delta\ge 3$ and $\fG=(V,E,\fB,\mu)$ be a measured graph that does not contain $K_{\Delta+1}$ as a subgraph and such that $\Delta(\fG)\le \Delta$.
    Then $\chi_\mu(\fG)\le \Delta$.    
\end{theorem}

Combining Theorem~\ref{thm:DetMarks} and Theorem~\ref{thm:MeasureBrooks} yields for every $\Delta\ge 3$ the existence of an acyclic $\Delta$-regular Borel graph $\fG$ such that $\chi_\mu(\fG)\le \Delta<\Delta+1=\chi_\fB(\fG)$ for every Borel probability measure $\mu$.\footnote{In fact, for large $\Delta$ the gap is much bigger as the measurable chromatic number of acyclic $\Delta$-regular graphs is bounded by $O(\Delta/\log(\Delta))$ by \cite[Theorem~3.9]{Bernshteyn2021LLL}.}
Note that this result exactly matches with Theorem~\ref{thm:LowerBoundsDistributed} and Theorem~\ref{thm:DistributedBrooks} from the theory of distributed computing.
While all four theorems were proven independently, it follows from the correspondence of Bernshteyn \cite{Bernshteyn2021LLL} that Theorem~\ref{thm:DetMarks} implies Theorem~\ref{thm:LowerBoundsDistributed}, and Theorem~\ref{thm:DistributedBrooks} implies Theorem~\ref{thm:MeasureBrooks} in a black-box manner.
\vspace{+0.2cm}

In the case of measurable edge colorings, Cs\'{o}ka, Lippner and Pikhurko \cite{csokalippnerpikhurko} showed that $\chi'_\mu(\fG)\le \Delta+1$ for a graphing $\fG=(V,\fB,E,\mu)$ that does not contain odd cycles and proved an upper bound of $\Delta+O(\sqrt{\Delta})$ colors for graphings in general.
In a related result, Bernshteyn \cite[Theorem 1.3]{BernshteynEarlyLLL} proved that $\Delta+o(\Delta)$ colors are enough (even for the so-called list-coloring version) provided that the graphing factors to the shift action $\Gamma\curvearrowright [0, 1]^\Gamma$ of a finitely generated group $\Gamma$.
Answering a question of Ab\' ert, the first author and Pikhurko \cite{grebik2020measurable} proved a measurable version of Vizing's theorem for graphings, that is, $\chi'_\mu(\fG)\le \Delta+1$ for any graphing $\fG=(V,\fB,\mu,E)$.
As we discuss in Section~\ref{sec:Vizing}, the technique developed in \cite{grebik2020measurable} was applied in the LOCAL model by Bernshteyn \cite{BernshteynVizing} and others \cite{Christiansen,BernshteynVizing2}.
Finally, the full measurable analogue of Vizing's theorem was derived by the first author \cite{GrebikVizing}.

\begin{theorem}[Measurable Vizing coloring, \cite{GrebikVizing}]\label{thm:MeasureVizing}
    Let $\fG=(V,E,\fB,\mu)$ be a measured graph.
    Then $\chi'_\mu(\fG)\le \Delta(\fG)+1$.        
\end{theorem}

Combining Theorem~\ref{thm:MeasureVizing} and Theorem~\ref{thm:DetMarks} gives for every $\Delta\ge 3$ an example of an acyclic $\Delta$-regular Borel graph $\fG$ such that $\chi'_\mu(\fG)\le \Delta+1<2\Delta-1=\chi'_\fB(\fG)$ for every Borel probability measure $\mu$.

A \emph{Borel version} of Vizing's theorem is possible under additional assumption on the structure of the graph.
Most notably, Bowen and Weilacher \cite{BowenWeilacher} showed that $\chi'_\fB(\fG)\le \Delta(\fG)+1$ under the assumption that $\fG$ is bipartite and has finite \emph{Borel asymptotic dimension}, notion that generalizes the classical notion of \emph{asymptotic dimension} introduced by Gromov \cite{Gromov} to the setting of measurable combinatorics by Conley, Jackson, Marks, Seward and Tucker-Drob \cite{AsymptoticDim}, and Bernshteyn and Dhawan showed that $\chi'_\fB(\fG)\le \Delta(\fG)+1$ under the assumption that $\fG$ has subexponential neighborhood growth.
Stronger results in the special case of free Borel $\mathbb{Z}^d$-actions, that is, Borel edge coloring with $2d$ colors which is the analogue of K\H{o}nig's line coloring theorem in this setting, were obtained independently around the same time in \cite{ArankaFeriLaci,spencer_personal,grebik_rozhon2021toasts_and_tails,weilacher2021borel}.

\begin{remark}[Baire measurable colorings]\label{rem:Baire}
A \emph{topological relaxation} of Borel colorings is formalized using the notion of \emph{Baire measurable functions} and \emph{meager sets}.
Similarly to the measurable case, we are interested in coloring a given Borel graph off of a topologically negligible set for any compatible Polish topology that generates the underlying standard Borel structure.
We refer the reader to the book of Kechris \cite{kechrisclassical} for basic notions and to the survey of Kechris and Marks \cite{kechris_marks2016descriptive_comb_survey} for results in that direction.
\end{remark}

\begin{remark}[Derandomization]\label{rem:HigherRandom}
As illustrated in Figure~\ref{fig:GlobalLOCAL} and by the distributed Brooks' theorem, in the regime $o(\log(n))$ randomness gives an advantage in the LOCAL model for solving LCL problems on graphs of degree bounded by $\Delta\ge 3$.
This behavior might change when we restrict to a smaller class of graphs.
For example, randomness does not help on graphs that look locally like the grid $\mathbb{Z}^d$, see \cite{brandt_grids}, and the celebrated conjecture of Chang and Pettie \cite{chang2019time} would imply the same for any graph class of subexponential growth.
\vspace{+0.2cm}

On the side of measurable combinatorics, Cs\'{o}ka, Grabowski, M\'{a}th\'{e}, Pikhurko and Tyros \cite{OlegLLL} found a Borel version of the Lov\'{a}sz Local Lemma for graphs of subexponential growth, which when combined with the correspondence of Bernshteyn shows that if an LCL problem has \emph{randomized} LOCAL complexity $O(\log(n))$ then it can be solved in a \emph{Borel} way on such graphs, see \cite[Theorem~2.15]{Bernshteyn2021LLL}.
This is a particular instance of a derandomization on graphs of subexponential growth, see also \cite{ConleyTamuz,thornton2021orienting,BernshteynSubexp}.

Naturally, it is tempting to ask about a higher analogue of derandomization.
That is, for what class of Borel graphs does the classes of LCL problems that admit \emph{measurable} and \emph{Borel} solution coincide?
While the answer is positive for Borel graphs that look locally like $\mathbb{Z}$ even if we allow LCL problems with inputs \cite{grebik_rozhon2021LCL_on_paths}, a counterexample for Borel graphs that look locally like $\mathbb{Z}^2$ was recently constructed by Berlow, Bernshteyn, Lyons and Weilacher \cite{berlow2025separating}.

A related question asks whether assuming that a given LCL problem can be solved measurably on \emph{all} measured graphs of degree at most $\Delta$ implies the existence of a Borel solution on subexponential growth Borel graphs of degree at most $\Delta$.
The answer to this question turned out to be negative as well, in a recent work in progress a counterexample was found by Bowen, the first author and Rozho\v{n}.
Namely, it can be shown that the LCL problem that is a union of the sinkless orientation problem and ``mark the line'' problem can be solved measurably on all graphs, but does not admit a Borel solution on graphs that can have growth arbitrarily close to a linear growth.
\end{remark}

\section{Summary of coloring results and related topics}\label{sec:Summary}

For the convenience of the reader, we summarize the results presented in previous sections in a table.
The second column, that is, greedy coloring, refers to both a vertex coloring with $\Delta+1$ colors and an edge coloring with $2\Delta-1$ colors.
Note that while the conjectured LOCAL complexity for Brooks coloring is $\Theta(\log(n))$ deterministic and $\Theta(\log(\log(n)))$ randomized, the randomized LOCAL complexity of Vizing coloring remains mysterious, see Problem~\ref{pr:SublogVizing}.

Recall that we treat $\Delta$ as a constant, in particular, $O(n)=O(m)$, where $n$ and $m$ denote the sizes of the vertex and edge set of a given finite graph $G$.
We denote as $\dlocal$ and $\rlocal$ the deterministic and randomized LOCAL model respectively.
\vspace{+0.5cm}

\begin{tabular}{ |p{4cm}||p{4cm}|p{3cm}|p{4.3cm}|  }
 \hline
 \multicolumn{4}{|c|}{Graphs of degree bounded by $\Delta$} \\
 \hline
 & Greedy coloring & Brooks coloring & Vizing coloring\\
 \hline\hline
 $\dlocal$ upper bound & $O(\log^* (n))$   & $O(\log^2(n))$  & $\operatorname{poly}(\log(\log(n)))\log^5(n)$  \\
 &   \cite{cole86,goldberg88,linial92LOCAL} & \cite{ghaffari2020Delta_coloring} & \cite{BernshteynVizing2} \\
 \hline
 $\dlocal$ lower bound & $\Omega(\log^* (n))$  & $\Omega(\log(n))$ & $\Omega(\log(n))$ \\
 &  \cite{linial92LOCAL} & \cite{brandt_etal2016LLL} &  \cite{brandt_etal2016LLL} \\
 \hline
 $\rlocal$ upper bound & $O(\log^* (n))$  & $O(\log^2(\log(n)))$ & $\operatorname{poly}(\log(\log(n)))\log^2(n)$ \\
 &  \cite{cole86,goldberg88,linial92LOCAL} & \cite{ghaffari2020Delta_coloring} & \cite{BernshteynVizing2}\\
 \hline
 $\rlocal$ lower bound & $\Omega(\log^* (n))$  & $\Omega(\log(\log(n)))$  & $\Omega(\log(\log(n)))$ \\
 & \cite{linial92LOCAL} & \cite{chang_kopelowitz_pettie2019exp_separation}  & \cite{ChangHLPU20}\\
 \hline\hline
 Borel & yes  & no  & no \\
 & \cite{KST} & \cite{DetMarks} & \cite{DetMarks}\\
 \hline
 measurable & yes  & yes  & yes \\
 & \cite{KST} & \cite{conley2016brooks} & \cite{GrebikVizing} \\
 \hline\hline
 augmenting & $\Theta(1)$   & $\Theta(\log(n))$ & $\Theta(\log(n))$ \\
 subgraphs & &  \cite{PanconesiSrinivasan} &  \cite{ChangHLPU20,Christiansen}\\
 \hline\hline
 deterministic & $O(n)$ & $O(n)$ & $O(n\log(n))$\\
 sequential algorithm& & \cite{BrooksAlg} & \cite{Gabow}\\
 \hline
 randomized & $O(n)$ & $O(n)$ & $O(n)$\\
  sequential algorithm& & \cite{BrooksAlg} & \cite{BernshteynVizing2}\\
 \hline
\end{tabular}

\vspace{+0.5cm}

\subsection{Decidability and complexity}\label{subsec:Complexity}
Complexity considerations are an important topic not only in classical graph theory but also in measurable combinatorics and distributed computing, where they give natural lower bounds on the existence of measurable solutions or efficient distributed algorithms.
\vspace{+0.2cm}

We have seen that the Brooks and Vizing colorings can be constructed by a polynomial time algorithm.
What if we want to find a ``true'' coloring that would match the chromatic number or index of a given graph?
A classical result of Karp \cite{karp1975computational} says that deciding the existence of a vertex coloring with $k$ colors of a given graph is $\operatorname{\bf{NP}}$-complete for every $k \geq 3$.
The situation with chromatic numbers is more complicated, even approximate computations are known to be hard in various senses.
For example, we have the following.

\begin{theorem}[\cite{zuckerman2006linear}]
	For every $\varepsilon>0$ it is $\operatorname{\bf{NP}}$-hard to decide the chromatic number up to a factor of $n^{1-\varepsilon}$ of a graph of size $n$.
\end{theorem}

Nevertheless, a combination of the remarkable theorems of Emden-Weinert, Hougardy and Kreuter \cite{Emden}, and Molloy and Reed \cite{Molloy} shows that there is a dichotomy for detecting large chromatic numbers via local obstacles, akin to Brooks' theorem.
Let $\Delta\in \mathbb{N}$ and set $k_\Delta$ to be the maximum integer such that $(k+1)(k+2)\le \Delta$, i.e., $k_\Delta\approx \sqrt{\Delta}-2$.

\begin{theorem}\label{thm:RemarkableResult}
    Let $\Delta$ be large.
    Then we have the following.
    \begin{itemize}
        \item \cite{Emden} It is $\operatorname{\bf{NP}}$-complete problem to check whether $\chi(G)\le c$ for graphs of degree bounded by $\Delta$, where $3\le c\le \Delta-k_\Delta-1$,
        \item \cite{Molloy} There is a linear time deterministic algorithm to check whether $\chi(G)\le c$ for graphs of degree bounded by $\Delta$, where $\Delta-k_\Delta\le c$.
        Moreover, producing vertex coloring with $c$ colors, if it exists, can be done in polynomial time.
    \end{itemize}
\end{theorem}

A recent work of Bamas and Esperet \cite{BamasEsperet} showed that the same dichotomy\footnote{The threshold in their result is off by $1$ compared to Theorem~\ref{thm:RemarkableResult}} holds in the LOCAL model as well.
In particular, they showed that in the regime when $c$ is above the threshold, there is a fast randomized distributed algorithm that produces vertex coloring with $c$ colors or finds a local obstruction if such coloring does not exists.
On the other hand, they show that if $c$ is below the threshold, then the LOCAL complexity of any distributed algorithm that produces vertex coloring with $c$ colors on graphs that have chromatic number equal to $c$ is $\Omega(n)$.
These results have a complete analogue in measurable combinatorics, see \cite[Section~3.A]{Bernshteyn2021LLL}.
\vspace{+0.2cm}

For edge colorings, the situation is simpler in the sense that $\chi'(G)\in\{\Delta(G),\Delta(G)+1\}$ by Vizing's theorem.
There are known sufficient conditions for $\chi'(G)=\Delta(G)$.
For example, by K\H{o}nig's theorem it is enough if $G$ is bipartite, see the book of Stiebitz, Scheide, Toft, and Favrholdt \cite{EdgeColoringBook} for more details in this direction.
Holyer \cite{holyer1981np}, however, showed that distinguishing between the two cases in general is complicated.

\begin{theorem}[\cite{holyer1981np}]
    For every $\Delta\ge 3$.
	It is $\operatorname{\bf{NP}}$-complete to decide if a graph $G$ of maximum degree at most $\Delta$ satisfies $\chi'(G)\le \Delta$.
\end{theorem}

In descriptive set theory, complexity of a collection of structures with a given property, for example, Borel graphs of a given Borel chromatic number, is measured using the complexity of the set of \emph{codes} for these structures.
These codes can be identified with real numbers, hence a collection of structures with a given property can be thought of as a subset of $\mathbb{R}$, and the complexity is measured using the \emph{projective hierarchy} of subsets of $\mathbb{R}$, see \cite[Section~37]{kechrisclassical} for the definition and basic properties of projective sets.
Roughly speaking, the complexity measures the number of alternating quantifiers over $\mathbb{R}$ that are needed for the definition of the set.
For instance, the standard definition of the set of Borel graphs of Borel chromatic number at most $k\in \mathbb{N}$, that is,
\begin{equation}\label{eq:Codes}
    \{\fG=(V,E,\fB):\exists \text{ Borel } c:V\to [k] \text{ such that } \forall \{x,y\}\in E \ c(x)\not=c(y) \},
\end{equation}
can be expressed by first assuming $(V,\fB)=[0,1]$ and then encoding the parameters $E$ and $c$ as real numbers.
In particular, the set from \eqref{eq:Codes} is defined using $\exists\forall$ quantification over $\mathbb{R}$.
Such sets are called \emph{$\mathbf{\Sigma}^1_2$ sets}.
The class of $\mathbf{\Sigma}^1_2$ sets gives a natural upper bound on the complexity of codes for the existence of graph colorings on Borel graphs.
Analogously to $\operatorname{\bf{NP}}$-completeness, if both quantifiers $\exists\forall$ in the definition of a set are necessary, then we say that it is \emph{$\mathbf{\Sigma}^1_2$-complete}.
Surprisingly, there is an analogue of the aforementioned result of Karp in measurable combinatorics.
In fact, the following result implies a strong failure of Brooks'-like theorems in the Borel context.

\begin{theorem}[\cite{todorvcevic2021complexity,brandt_chang_grebik_grunau_rozhon_vidnyaszkyhomomorphisms}]\label{thm:BrooksComplexity}
    Let $k\ge 3$.
    The set of Borel graphs of Borel chromatic number at most $k$ is $\mathbf{\Sigma}^1_2$-complete set.
    This is true even if we restrict to the class of $k$-regular acyclic Borel graphs.
 In contract, Borel graphs admitting a Borel $2$-coloring for a $\mathbf{\Pi}^1_1$ set. 
\end{theorem}

It is worth mentioning that in the past couple of years a number of further results have been shown concerning descriptive complexity of combinatorial notions, revealing interesting connections particularly to the complexity theory of constraint satisfaction problems (CSPs). 

Note that graph coloring problems are special cases of homomorphism problems between relational structures (also called CSPs), for example, a $k$-coloring of a graph is just a homomorphism to the complete graph $K_k$. In the finite world, this kind of homomorphism problems are well investigated: a major recent breakthrough is that deciding the existence of a homomorphism to any $H$ is either in $\text{P}$ or $\text{NP}$-complete, \cite{bulatov2017dichotomy,zhuk2020proof,brady2022notes}. In fact, a complete algebraic understanding of structures $H$ for which the homomorphism problem is easy/hard has been reached. 

Thornton \cite{thornton2022algebraic} has initiated a systematic study of the complexity of Borel homomorphism problems of the following sort: given a finite relational structure $H$, decide the descriptive of complexity of Borel structures (with the same signature) that admit a Borel homomorphism to $H$. Interestingly, if the $H$ homomorphism problem is known to be $\text{NP}$-complete, then the corresponding Borel collection is $\mathbf{\Sigma}^1_2$-complete. However, recent results of the authors \cite{grebikcomplexity} show a stark contrast with the finite world: if $H$ corresponds to solving systems of linear equations over a finite field--this problem is in $P$ in the finite case, by Gaussian elimination--the Borel problem is still $\mathbf{\Sigma}^1_2$-complete. 

A different, wide open question is the investigation of the complexity of \emph{hyperfinite} Borel graphs, see, e.g., \cite{jackson2002countable,kechris2024theory}. Note that establishing that hyperfinite Borel graphs form a $\mathbf{\Sigma}^1_2$ set would yield a negative answer to some of the most important conjectures concerning this concept. In this direction, the first author and Higgins \cite{grebik2024complexity} have shown that deciding whether a Borel graph has \emph{Borel asymptotic dimension $1$}, which is a strengthening of hyperfiniteness, is $\mathbf{\Sigma}^1_2$-complete. Further, surprisingly, it turns out \cite{frisch2024hyper}that $\mathbf{\Sigma}^1_2$-completeness of hyperfiniteness also follows from a negative answer to the \emph{increasing union conjecture} \cite{jackson2002countable}, that is, the statement that hyper-hyperfinite equivalence relations are hyperfinite.

\begin{remark}[Deciding the Global Picture]
    It is an exciting question to understand for what graph classes is the membership question for LCL problems in Figure~\ref{fig:GlobalLOCAL} itself decidable.
    That is, after we fix a graph class, is there an algorithm that would take an LCL problem as an input and output its LOCAL complexity?
    As it is possible to simulate the halting problem using LCL problems on $\mathbb{Z}^2$, the membership problem is known to be \emph{undecidable} on graphs that look locally like the grid $\mathbb{Z}^d$, whenever $d>1$, see \cite{naorstockmeyer} and \cite{brandt_grids}.
    
    The same questions can be asked in the context of measurable combinatorics.
    That is, after we fix a class of Borel graphs, is there an algorithm that would take an LCL problem as an input and output whether it can be solved in a Borel way on every Borel graph in the class?
    Our current knowledge here is identical with the one in the LOCAL model.
    For example, undecidability for LCL problems on Borel graphs that look locally like $\mathbb{Z}^d$, where $d\ge 1$, has been independently established by Gao, Jackson, Krohne and Seward \cite{GJKS}.
\end{remark}

\subsection{Formal connections}\label{subsec:Correspondence}
In this section we give a brief overview of the formal connections between distributed computing and measurable combinatorics that were first discovered in the seminal paper of Bernshteyn \cite{Bernshteyn2021LLL}.
Throughout this section, we work with a fixed class of finite graphs that \emph{approximate locally} a class of Borel graphs.
One should think of the class of finite graphs of degree bounded by $\Delta$ and the class of Borel graphs of degree bounded by $\Delta$, or the class of finite trees and the class of Borel acyclic graphs.
Under these mild assumptions, Bernshteyn \cite{Bernshteyn2021LLL} proved the following, see Figure~\ref{fig:GlobalLOCAL}.

\begin{theorem}\label{thm:Bernshteyn}
    Let $\Pi$ be a graph coloring problem, or an LCL problem in general.
    \begin{itemize}
        \item If the deterministic LOCAL complexity of $\Pi$ is $o(\log(n))$, then $\Pi$ admits a Borel solution.
        \item If the randomized LOCAL complexity of $\Pi$ is $o(\log(n))$, then $\Pi$ admits a measurable solution.
    \end{itemize}
\end{theorem}

Note that this result allows a black-box translations of upper bounds from distributed computing to measurable combinatorics and lower bounds in the opposite directions.
We refer the reader to \cite[Section~3]{Bernshteyn2021LLL} to see some of the striking applications of Theorem~\ref{thm:Bernshteyn}.
\vspace{+0.2cm}

A natural question arises: can the implications in Theorem~\ref{thm:Bernshteyn} be reversed?
That is, can the existence of a Borel or a measurable solution imply the existence of an efficient distributed algorithm?
It turns out that in the setting of deterministic distributed algorithms, the right analogue on the descriptive set theory sides are \emph{continuous solutions}.
Gao, Jackson, Krohne and Seward \cite{GJKS} combinatorially characterized what LCL problems can be solved in a continuous way on $\mathbb{Z}^d$.
Their result immediately implies that this class coincide with LCL problems of deterministic LOCAL complexity $o(^d\sqrt{n})$ on finite graphs that look locally like $\mathbb{Z}^d$, see \cite[Section~7]{grebik_rozhon2021toasts_and_tails} for a high-level overview of the argument.
This was generalized independently by Bernshteyn \cite{Bernshteyn2021local=cont} and Seward \cite{Seward_personal} to Cayley graphs of finitely generated groups, and by Brandt, Chang, Greb\'{i}k, Grunau, Rozho\v{n} and Vidny\'{a}nszky \cite{DeterministicTrees} to regular trees.

We refer the reader to \cite{grebik_rozhon2021LCL_on_paths,grebik_rozhon2021toasts_and_tails,brandt_chang_grebik_grunau_rozhon_vidnyaszky2021LCLs_on_trees_descriptive,HolroydSchrammWilson2017FinitaryColoring,BernshteynWeilacherLLL,GJKS2,lyons2024descriptive} for further information about the relationship between the complexity classes of LCL problems for particular graph classes and for connections to the theory of random processes, in particular, \emph{finitary factors of iid}.

\section{Marks' games}
\label{sec:Marks}
In this section we discuss the adaptation of seminal results of Marks, see Theorem \ref{t:marksmain} and \cite{DetMarks}, to the LOCAL model. This adaptation has been useful in several different ways. First, it resulted in new proofs of lower bounds for colorings and new results for general LCL problems on bounded degree trees in the LOCAL model. Second, a certain technical difficulty pointed towards introducing the main novelty to be discussed, namely the \emph{homomorphism graph/ID graph technique} in the LOCAL model. Third, a transfer of this technique to descriptive combinatorics found a number of new applications there.

As discussed in Section \ref{sec:measurable}, Marks has shown that the greedy upper bound for vertex and edge coloring of $\Delta$-regular graphs is sharp in the Borel context, even in the case of acyclic graphs. His key idea was to define two player games, where, by strategy stealing arguments, the existence of winning strategies for any of the players contradicts the existence of Borel colorings. In what follows, we will only describe these techniques in the case of vertex colorings, nevertheless, they can be used in a much wider context, to rule out for example matchings or homomorphisms. These generalizations are often non-trivial and require additional technical work, e.g., considering hypergraphs instead of graphs and defining more complicated games, see \cite{DetMarks,brandt_chang_grebik_grunau_rozhon_vidnyaszkyhomomorphisms,thornton2021orienting}.

Let us also mention that the first lower bounds for deterministic vertex colorings of trees of degree bounded by $\Delta$ were proved in \cite{brandt_etal2016LLL}, and later Brandt \cite{brandt19automatic} has formalized these ideas in his celebrated \emph{round elimination method}. Roughly speaking, what the round elimination method does is the following: assuming that a labeling problem $\Pi$ on regular trees is solvable by an $r$-round algorithm, it constructs a new problem $\Pi'$ solvable by an $r-1$-round algorithm. Now, if $\Pi$ happens to be a fixed point of the operation $'$, then it must be solved by a $0$-round algorithm showing that either $\Pi$ is trivial, or the original assumption was false.\footnote{See \url{https://github.com/olidennis/round-eliminator} for a fully automated computation of $\Pi'$.} 

This method has a striking similarity to Marks' technique in certain cases. For example, in the case of so called sinkless orientation problem, which is a fixed point of $'$, round elimination essentially produces strategies for the players in the games similar to what Marks considered. Hence, it would be very tempting to think that these methods are the same, say, for fixed points of the operation $'$. This is not the case in general: there are fixed point LCL problems on trees solvable in the Borel context but unsolvable in the LOCAL model. Nevertheless, understanding the exact nature of this connection is still an exciting open problem, see \cite{lyons2024descriptive}.

The adaptation of Marks method will be illustrated by proving the following.

\begin{theorem}
	\label{t:marksmainloc}
	There is no deterministic algorithm of complexity $o(\log (n))$ that produces a vertex coloring with $\Delta$ colors of acyclic graphs of degree $\leq \Delta$. 
\end{theorem} 

There are three main ideas in the proof, all of which we will discuss in this section. As mentioned above, the most important one comes from \cite{DetMarks}:
\begin{shaded}
\begin{center}
 1. Games and strategy stealing to rule out the colorings.
\end{center}
\end{shaded}
As we will see, executing this idea runs into certain problems with the uniqueness of the IDs in the LOCAL model and with yielding acyclic graphs in the descriptive case. Let us also mention that this issue is present in Brandt's construction, as well; he overcame it by using probabilistic arguments that do not work in the Borel context. 
\vspace{+0.2cm}

In \cite{brandt_chang_grebik_grunau_rozhon_vidnyaszky2021LCLs_on_trees_descriptive} the homomorphism graph/ID graph technique was invented, which works in both LOCAL and Borel realms. 
\begin{shaded}
\begin{center}
2. Homomorphism graphs/ID graphs to enforce injectivity. 	
\end{center}
\end{shaded}
Subsequently, it became apparent that homomorphism graphs serve as a general tool to transfer properties of large degree graphs to bounded degree graphs. This has been widely utilized in \cite{brandt_chang_grebik_grunau_rozhon_vidnyaszkyhomomorphisms} in the descriptive case, but not yet in the LOCAL model.
\begin{shaded}
\begin{center}
	3. Formulating transfer principles and using games as voting systems. 
\end{center}	
\end{shaded}

Before turning to the proof, let us recall that by Bernshteyn's correspondence theorem (Theorem \ref{thm:Bernshteyn}) inexistence of Borel solutions to a locally checkable problem automatically implies a corresponding result in the deterministic LOCAL model in $o(\log n)$-many rounds, so formally, for the below results, we don't have to spell out the finitary adaptations. Despite this, we would like to describe them, since it has been developed in parallel to Bernshteyn's theorem, moreover, we believe that there should be models stronger than deterministic LOCAL (see Problem \ref{p:variant}), for which these techniques still could be used.
\vspace{+0.2cm}

\noindent{\emph{Technical preliminaries.}} In the rest of this section, $T_\Delta$ will stand for the $\Delta$-regular infinite tree and $T^r_\Delta$ for an $r$-neighborhood $B_{T_\Delta}(v,r)$ of any (every) vertex $v$ in $T_\Delta$. As discussed in Section \ref{sec:local} a deterministic $r$-round LOCAL algorithm on a fixed graph is the same as a map from $r$-neighborhoods with unique IDs of vertices of the graph. Therefore, sometimes we will consider $T^r_\Delta$ as being rooted in $v$ (which does not change the isomorphism type of such graphs), and we will evaluate algorithms on such neighborhoods, yielding a color on the root. Also, we will see that it is possible to prove lower bounds already from the existence of algorithms that map labeled copies of $T^r_\Delta$ to different colors, hence we could restrict our attention to only such maps. Moreover, in the case of vertex colorings, it suffices to check the validity of the algorithm on any pair of adjacent vertices, thus we will typically construct labelings of $T^{r+1}_\Delta$ to derive a contradiction. 

It will be also convenient in the first part of our discussion to consider maps from labeled trees with labelings that are not necessarily injective. This motivates the following definitions (here ``G" comes from generalized).

\begin{definition}
	
	An \emph{$r$-round GLOCAL algorithm (on trees with degrees $\leq \Delta$)} is a partial function $\mathcal{A}_r$ from (isomorphism classes of) such labeled rooted trees of radius $\leq r$ to some set $\Sigma$.\footnote{Note that in the definition of GLOCAL, we do not require the algorithm to always have an output. However, in practice, when we say that such an algorithm exists on some graph class, it necessarily outputs colors on vertices of those graphs. Thus, alternatively, we could require that $\mathcal{A}_r$ is defined on every isomorphism class, but only correct on the class of graphs in question.}	
	
\end{definition}

If $G$ is a labeled tree with $\Delta(G) \leq \Delta$, we can \emph{run $\mathcal{A}_r$ on $G$}: at any vertex $v$ of $G$ we take the labeled isomorphism class of $B_G(v,r)$ and feed it into $\mathcal{A}_r$. This gives a partial vertex labeling of $G$.

\begin{definition}
	Let $\mathcal{F}$ be a collection of finite labeled trees of degrees $\leq \Delta$. An $r$-round GLOCAL  algorithm $\mathcal{A}_r$ \emph{vertex colors $\mathcal{F}$}, if for any $G \in \mathcal{F}$, if $\mathcal{A}_r$ is run on $G$ in the above sense, it yields a vertex coloring of $G$ (in particular, every vertex must get a color). 
	
\end{definition}

\subsection{Games} We describe the most straightforward finitary adaptation of the original proof of Marks. This, in itself is not suitable to exclude colorings in the LOCAL model, as the labels come from the set $\{0,1,2,3\}$, and hence results only in a less interesting statement, but still encompasses the main idea. For the sake of simplicity, we will consider labeled $T^{r+1}_3$'s with labels $\in \{0,1,2,3\}$, so that  \emph{neighboring vertices must have different labels}, i.e., the labeling is injective on the edges. Denote this collection by $EI^{r+1}_3$. This can be viewed as a simplified version of the LOCAL model, in which the IDs are not unique, and the algorithm does not know the size of the graph.
\begin{theorem}
	Let $r$ be arbitrary. There is no $r$-round GLOCAL algorithm to $3$-color EI$^{r+1}_3$.
\end{theorem}

\begin{proof}
	Towards contradiction, assume that such an $r$-round GLOCAL algorithm $\mathcal{A}_r$ exists. For every $v \in \{0,1,2,3\}$ and $i \in \{0,1,2\}$ define the game $\mathbb{G}(v,i)$ as follows.
	Two players, $I$ and $II$ label with numbers $ \in \{0,1,2,3\}$ the $r$-neighborhood of the rooted $3$-regular tree from the root, according to Figure \ref{fig:game}, in rounds, alternatingly.

	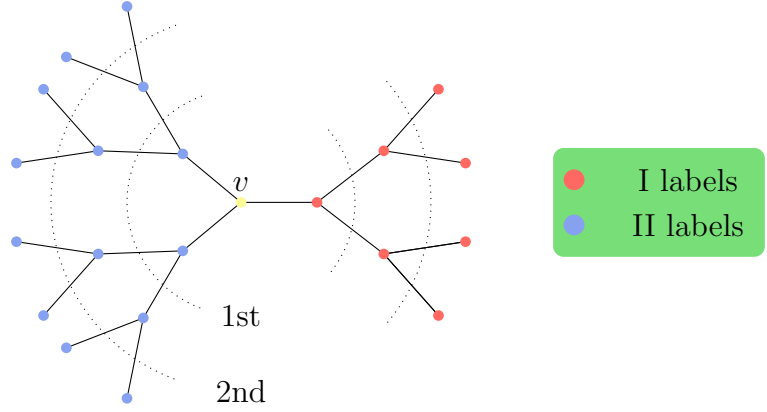
\begin{figure}

		\tikzstyle{mybox} = [draw=black, fill=blue!20, very thick,
		rectangle, rounded corners, inner sep=10pt]

		\tikzstyle{fancytitle1} =[fill=pastelred,rounded corners]
		\tikzstyle{fancytitle2} =[fill=pastelblue,rounded corners]
		\tikzstyle{legend} =[fill=pastelgreen,rounded corners]
		
		\begin{tikzpicture}
		\coordinate (0) at (0,0);
		
		\coordinate (a) at ($(0) +(0:1)$);
		\coordinate (b) at ($(0) +(140:1)$);
		\coordinate (c) at ($(0) +(220:1)$);
		
		\coordinate (ab) at ($(0) +(20:2)$);
		\coordinate (ac) at ($(0) +(-20:2)$);
		
		\coordinate (aba) at ($(0) +(30:3)$);
		\coordinate (abc) at ($(0) +(10:3)$);
		
		\coordinate (aca) at ($(0) +(-10:3)$);
		\coordinate (acb) at ($(0) +(-30:3)$);
		
		\coordinate (ba) at ($(0) +(130:2)$);
		\coordinate (bc) at ($(0) +(160:2)$);
		
		\coordinate (ca) at ($(0) +(200:2)$);
		\coordinate (cb) at ($(0) +(230:2)$);
		
		\coordinate (bab) at ($(0) +(140:3)$);
		\coordinate (bac) at ($(0) +(120:3)$);
		
		\coordinate (bcb) at ($(0) +(150:3)$);
		\coordinate (bca) at ($(0) +(170:3)$);
		
		\coordinate (cab) at ($(0) +(190:3)$);
		\coordinate (cac) at ($(0) +(210:3)$);
		
		\coordinate (cba) at ($(0) +(220:3)$);
		\coordinate (cbc) at ($(0) +(240:3)$);
		
		\coordinate (1st) at ($(0) +(270:1.5)$);
		\coordinate (2nd) at ($(0) +(270:2.5)$);
		\coordinate (edgelabel) at (0.5,0);
		
		\draw (0) -- (a);
		\draw (0) -- (b);
		\draw (0) -- (c);
		
		\draw (a) -- (ab);
		\draw (a) -- (ac);
		
		\draw (b) -- (ba);
		\draw (b) -- (bc);
		
		\draw (c) -- (ca);
		\draw (c) -- (cb);
		
		\draw (ab) -- (abc);
		\draw (ab) -- (aba);
		
		\draw (ac) -- (acb);
		\draw (ac) -- (aca);
		
		\draw (ac) -- (acb);
		\draw (ac) -- (aca);
		
		\draw (ba) -- (bac);
		\draw (ba) -- (bab);
		
		\draw (bc) -- (bca);
		\draw (bc) -- (bcb);
		
		\draw (ca) -- (cab);
		\draw (ca) -- (cac);
		
		\draw (cb) -- (cba);
		\draw (cb) -- (cbc);
		
		%\draw[blue, dotted]
		%let \p1 =  ($(3)-(center)$),
		%\n0 = {veclen(\x1,\y1)}
		%in (center) circle(\n0);
		
		%\draw[dotted] (0,0) arc [start angle=0, delta angle=30, radius=1.5cm];
		
		%\filldraw[draw=black]
		%let \p1 = ($(3) - (center)$),
		%\p2 = ($(4) - (center)$),
		%\n0 = {veclen(\x1,\y1)},            % Radius
		%\n1 = {atan(\y1/\x1)+180*(\x1<0)},  % initial angle
		%\n2 = {atan(\y2/\x2)+180*(\x2<0)}   % Final angle
		%in
		%(1) -- (2) --  (3) arc(\n1:\n2:\n0)  -- (5)  -- cycle;
		\node[above] at (0) {$v$};
	
		\node[] at (1st) {1st};
		\node[] at (2nd) {2nd};
		
		\fill[color=pastelyellow] (0) circle(2pt);
		
		\foreach \dot in {b,c,ba,bc,ca,cb,cab,cac,cba,cbc,bab,bac,bca,bcb} {
			\fill[color=pastelblue] (\dot) circle(2pt);
			%	\node[above] at (\dot) {$\sigma$};
		}
		
		\foreach \dot in {a,ab,ac,aba,abc,aca,acb} {
			\fill[color=pastelred] (\dot) circle(2pt);
			%	\node[above] at (\dot) {$\sigma$};
		}
		
		\draw[dotted] (40:2.5) arc (40:-40:2.5) ;
		
		\draw[dotted] (110:2.5) arc (110:250:2.5) ;
		
		\draw[dotted] (40:1.5) arc (40:-40:1.5) ;
		
		\draw[dotted] (110:1.5) arc (110:250:1.5) ;
		
		\node [legend] (ll) at (5.5,0) {
			\begin{tikzpicture}[strong/.style={line width=0.5mm},weak/.style={line width=0.1mm}]
			\node[fancytitle1] at (0.5,1.2) {};
			\node[fancytitle2] at (0.5,0.6) {};
			
			\node[] (v) at (2,1.2)  {I labels};
			\node[] (j) at (2,0.6)  {II labels};
			
			\end{tikzpicture}
		};
		
	\end{tikzpicture}
	\centering
	
	\caption{The game $\mathbb{G}(v,i)$}
	\label{fig:game}
\end{figure}
	
	The root is initially labeled by $v$. In the $n$th round $I$ and $II$ label all the vertices having distance $n$ from the root. In the first round $I$ labels one of the neighbors of the root, then $II$ labels the other two neighbors. Then $I$ continues with labelling the unlabelled neighbors of her first move (i.e., her ``side"), then $II$ with his, etc. The only rule which have to be maintained is that neighboring vertices must get different labels along the game. After $r$-rounds they label an $r$-neighborhood of the root, thus the following winning condition is well defined:
	\[I \text{ wins in $\mathbb{G}(v,i)$ iff $\mathcal{A}_r$ evaluated on the labeled $r$-neighborhood}\] \[\text{obtained by the end of the game gives a color different from $i$}.\] 
	
	The basic strategy stealing observation is the following.
	\begin{lemma}
		\label{l:firstplayerloc0}
		For every $v$ there is an $i$ so that $I$ has no winning strategy in $\mathbb{G}(v,i)$.
	\end{lemma}
	\begin{proof}
		Playing these winning strategies, $(\sigma_i)_{i \in \{0,1,2\}}$, against each other (in the obvious manner, by ``rotating them", see Figure \ref{f:strat}) would yield in a partial labeling of $T^{r}_3$ on which $\mathcal{A}$ does not output any color.
	\end{proof}
	\begin{figure}

		\tikzset{every picture/.style={line width=0.75pt}} %set default line width to 0.75pt        
		
		\begin{tikzpicture}[x=0.75pt,y=0.75pt,yscale=-1,xscale=1]
		%uncomment if require: \path (0,300); %set diagram left start at 0, and has height of 300
		
		%Straight Lines [id:da6843690707728033] 
		\draw    (202.23,151) -- (257.25,150.73) ;
		%Curve Lines [id:da6873210282688915] 
		\draw [color={rgb, 255:red, 208; green, 2; blue, 27 }  ,draw opacity=1 ]   (274.34,199.7) .. controls (226.64,184.73) and (237.13,112.62) .. (271.16,101.7) ;
		%Straight Lines [id:da7565761177610845] 
		\draw    (202.37,150.81) -- (169.42,106.75) ;
		%Curve Lines [id:da4222764434591063] 
		\draw [color={rgb, 255:red, 208; green, 2; blue, 27 }  ,draw opacity=1 ]   (198.61,63.88) .. controls (214.96,111.13) and (150.75,145.59) .. (121.74,124.73) ;
		%Straight Lines [id:da7644412697946201] 
		\draw    (202.37,150.81) -- (169.17,193.67) ;
		%Curve Lines [id:da497199770318425] 
		\draw [color={rgb, 255:red, 208; green, 2; blue, 27 }  ,draw opacity=1 ]   (117.89,170.39) .. controls (161.52,145.96) and (206.78,203.07) .. (191.36,235.31) ;
		%Straight Lines [id:da10834221722207249] 
		\draw    (398.23,149) -- (453.25,148.73) ;
		%Curve Lines [id:da5854353816997578] 
		\draw [color={rgb, 255:red, 74; green, 144; blue, 226 }  ,draw opacity=1 ]   (501.42,211.75) .. controls (460.42,195.75) and (465.39,102.66) .. (499.42,91.75) ;
		%Straight Lines [id:da8918817539824246] 
		\draw    (453.25,148.73) -- (489.17,181.67) ;
		%Straight Lines [id:da8823284344063993] 
		\draw    (453.25,148.73) -- (488.17,112.67) ;
		%Curve Lines [id:da5426984864831557] 
		\draw [color={rgb, 255:red, 74; green, 144; blue, 226 }  ,draw opacity=1 ]   (350.17,89.83) .. controls (390.42,112.75) and (385.07,198.73) .. (351.42,210.75) ;
		%Straight Lines [id:da2133198307312778] 
		\draw    (398.23,149) -- (360.67,113.22) ;
		%Straight Lines [id:da5522635127938791] 
		\draw    (398.23,149) -- (360.17,183.67) ;
		
		% Text Node
		\draw (381,124) node [anchor=north west][inner sep=0.75pt]   [align=left] {};
		% Text Node
		
		% Text Node
		\draw (414.22,93.77) node [anchor=north west][inner sep=0.75pt]  [rotate=-5.78] [align=left] { };
		% Text Node
		\draw (276,89.4) node [anchor=north west][inner sep=0.75pt]    {$\sigma _{0}$};
		% Text Node
		\draw (179,238.4) node [anchor=north west][inner sep=0.75pt]    {$\sigma _{1}$};
		% Text Node
		\draw (196,41.4) node [anchor=north west][inner sep=0.75pt]    {$\sigma _{2}$};
		% Text Node
		\draw (204,126.23) node [anchor=north west][inner sep=0.75pt]    {$v$};
		% Text Node
		
		% Text Node
		\draw (466,84.4) node [anchor=north west][inner sep=0.75pt]    {$\tau$};
		% Text Node
		\draw (367,83.33) node [anchor=north west][inner sep=0.75pt]   [align=left] {$\sigma$};
		% Text Node
		\draw (396,128.73) node [anchor=north west][inner sep=0.75pt]    {$v$};
		% Text Node
		\draw (445,127.73) node [anchor=north west][inner sep=0.75pt]    {$w$};
		% Text Node
		
		\end{tikzpicture}
		\caption{Strategy stealing }
		\label{f:strat}
	\end{figure}
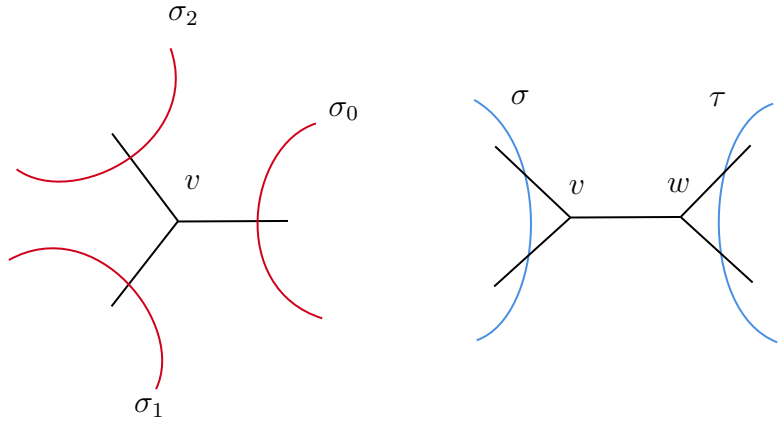
	Now, by the pigeonhole principle, there are distinct $v,w \in \{0,1,2,3\}$ and an $i$ so that $II$ has winning strategies in $\mathbb{G}(v,i)$ and $\mathbb{G}(w,i)$, denote these by $\sigma$ and $\tau$. Play $\sigma$ and $\tau$ against each other so that for $\sigma$ we pretend that $I$'s first move was $w$, and for $\tau$ we pretend that it was $v$ in the first round. Then, in the second round we use the response from $\tau$ as $I$'s play for $\sigma$, and similarly for $\tau$ we use $\sigma$'s response, etc (see again Figure \ref{f:strat}).
	
	This way we get a partial labeling of $T^{r+1}_3$, extend it in an arbitrary manner to an element of EI$^{r+1}_3$. Now, evaluating $\mathcal{A}_r$ on the central nodes labeled by $v$ and $w$, we must obtain the color $i$ at both vertices, since $\sigma$ and $\tau$ are winning strategies, contradicting that $\mathcal{A}_r$ correctly solves the $3$-coloring problem.

\end{proof}

\subsection{Homomorphism graphs} Of course, the key question now is how to use these ideas to obtain results about the ordinary LOCAL model. The reader can convince themselves that adding rules to the game in the most na\"ive way (e.g., requiring one of the players to use injective IDs) ruins the symmetry on which the strategy stealing arguments rely. 

Observe that the above restrictions on the labeling can be stated equivalently as requiring the labeling to give a homomorphism from $T^{r}_3$ to $K_4$, the complete graph on $4$ vertices, and this, restricted injectivity can be maintained throughout our games. One might replace $K_4$ with any other graph, motivating the following definition.

\begin{definition}
	Let $\HOM(T^r_\Delta,\mathcal{H})$ be the set of labelings of $T^r_{\Delta}$ that are homomorphisms to $\mathcal{H}$.
\end{definition}

Now, of course, requiring homomorphisms alone cannot guarantee injectivity. However, adding edge labeling helps: instead of the $\Delta$-regular tree consider it equipped with a proper edge $\Delta$-coloring and assume that $\mathcal{H}$ is equipped with a $\Delta$ edge labeling (which is typically not a proper coloring). To emphasize that we consider edge colored version of $T_\Delta$, we denote it by $T'_{\Delta}$. The next simple observation is the key.

\begin{proposition}
	\label{pr:injectivity}
	Assume that $\varphi:T'_\Delta \to \mathcal{H}$ an edge label preserving homomorphism and $\mathcal{H}$ is acyclic. Then $\varphi$ is injective. Similarly, for any $k>0$ if the girth of $\mathcal{H}$ is $>2k+2$ then $\varphi$ is injective in every $k+1$-neighborhood.
\end{proposition}
\begin{proof}
	To see the first statement, note that since its range is acyclic, if $\varphi$ was not injective then there would be some $(u,v)$ and $(v,w)$ distinct edges in $T'_\Delta$ so that $\varphi(u)=\varphi(w)$. But as $(u,v)$ and $(v,w)$ are adjacent, their color is different, so the color of their image must also be different, contradicting $\varphi(u)=\varphi(w)$. The proof of the second statement is identical.	
\end{proof}
	
	Thus, it is useful to consider the next definition. 

\begin{definition}
	Assume that the graph $\mathcal{H}$ is equipped with a $\Delta$ edge labeling. Let $\Homed(T^{r}_\Delta,\mathcal{H})$ be the set of labelings of $B_{T'_\Delta}(v,r)$ (for any $v$ in $T'_\Delta$) that are homomorphisms to $\mathcal{H}$ and preserve edge labels.
\end{definition}

	A way to view the graph $\mathcal{H}$ is as a ``set of rules", or a ``game plan'', for the players in the game, as we will soon see.
\subsection{Transfer principles}
Now, we can summarize the intuition which guides us. We take a suitable graph $\mathcal{H}$, and consider all the possible homomorphism labelings of the tree $T_\Delta$ to $\mathcal{H}$. We would like to prove that there is no ``local" coloring rule to $\Delta$-color the resulting graphs. Assume that there is such a rule. For each $v \in V(\mathcal{H})$ we can label the tree in a lot of ways so that the root is labeled by $v$. Of course, all these labeled rooted trees get a color, according to the rule. We use the game to decide the most popular color $i$, that is, when $II$ has a winning strategy in $\mathbb{G}(v,i)$. It turns out, that this transfers the original coloring to a coloring of $\mathcal{H}$.

Let us state the transfer principle for the distributed case.

\begin{theorem}[Transfer principle $1$, LOCAL]
	\label{t:tr1local}
	Let $\mathcal{H}$ be a finite graph. TFAE:
	\begin{enumerate}
		\item $\mathcal{H}$ admits a $\Delta$-coloring.
		\item There exists a $0$-round GLOCAL algorithm to $\Delta$-color $\HOM(T^{r}_\Delta,\mathcal{H})$, for all $r$.
		\item There exist an $r$ and an $r$-round GLOCAL algorithm to $\Delta$-color $\HOM(T^{r+1}_\Delta,\mathcal{H})$. 
		
	\end{enumerate}
\end{theorem}

\begin{proof}
	To see that (1) implies (2), assume that $\mathcal{H}$ admits a $\Delta$-coloring. Given a vertex in some element of $\HOM(T^r_\Delta,\mathcal{H})$, it has a label from $\mathcal{H}$. Color the vertex with the color of its label.
	
	Clearly, (2) $\implies$ (3), so it suffices to show that (3) implies (1), so assume that $\HOM(T^{r+1}_\Delta,\mathcal{H})$ admits an $r$-round GLOCAL algorithm $\mathcal{A}$. For each color $i$ and vertex $v \in V(\mathcal{H})$ consider the game $\mathbb{G}(v,i)$ defined as follows: two players build a neighborhood of a vertex of some element in $\HOM(T^r_\Delta,\mathcal{H})$, according to Figure \ref{fig:game}, where the vertex is labeled by $v$ and the two players have to maintain that the labeling is a homomorphism to $\mathcal{H}$ at every step. \[I \text{ wins iff $\mathcal{A}$ evaluated on the labeled $r$-neighborhood}\] \[\text{obtained by the end of the game gives a color different from $i$}.\] 
	
	\begin{lemma}
		\label{l:firstplayerloc}
		For every $v$ there is an $i$ so that $I$ has no winning strategy in $G(v,i)$.
	\end{lemma}
	\begin{proof}
		Playing these winning strategies against each other would yield in a labeling of $T^r_\Delta$ on which $\mathcal{A}$ does not output any color.
	\end{proof}
	
	The next lemma finishes the proof of the transfer principle. 
	\begin{lemma}
		\label{l:secondplayerloc}
		For a $v \in V(\mathcal{H})$ let
		\[c(v)=\min \{i: \text{$\mathbb{G}(v,i)$ is won by $II$}\}.\] Then $c$ is $\Delta$-coloring of $\mathcal{H}$. 
	\end{lemma}
	\begin{proof}
		Assume that this is not the case, i.e., there exists some $(v,w) \in \mathcal{H}$ so that $c(v)=c(w)$. Then $II$ has winning strategies $\sigma$ and $\tau$ in the games $\mathbb{G}(v,i)$ and $\mathbb{G}(w,i)$. Now play $\sigma$ and $\tau$ against each other so that for $\sigma$ we pretend that $I$'s first move was $w$ and for $\tau$ we pretend that it was $v$ in the first round. Then in the second round we use the response from $\tau$ as $I$'s play for $\sigma$ and similarly for $\tau$ we use $\sigma$'s response, etc.
		
		This way we get a partial labeling of $T^{r+1}_\Delta$, extend this in an arbitrary manner to obtain an element of $\HOM(T^{r+1}_\Delta,\mathcal{H})$. Now, evaluating $\mathcal{A}$ on the central nodes labeled by $v$ and $w$, we must obtain the color $i$ at both vertices, since $\sigma$ and $\tau$ are winning strategies, contradicting that $\mathcal{A}$ solves the $\Delta$-coloring problem. 
	\end{proof}

\end{proof}
This transfer principle alone is insufficient to get results for the LOCAL model, since the labelings typically are not injective. To remedy this, we have to consider edge labeled versions of the homomorphism graphs. For technical reasons (namely, in order for the players not to get stuck in the game), we assume that the edge labeling of $\mathcal{H}$ is \emph{nice}, that is, every vertex is adjacent to edges with all possible colors.

\begin{definition}
	\label{d:weakcoloring} Assume that $\mathcal{G}$ is a graph equipped with an edge $\Delta$-labeling $e$. An \emph{edge labeled $\Delta$-coloring} of $\mathcal{G}$ is a map $c:V(\mathcal{G}) \to \Delta$ so that there are no adjacent vertices $v,w$ with $c(v)=c(w)=e(v,w)$. If such a coloring exists, we will say that the \emph{edge labeled chromatic number of $\mathcal{G}$ is $\leq{\Delta}$}, and denote this fact by $\chi^{el}(\mathcal{G}) \leq \Delta$. 
\end{definition}

\begin{theorem}[Transfer principle $2$, LOCAL]
	\label{t:tr2local}
	Let $\mathcal{H}$ be a finite graph together with a nice edge $\Delta$-labeling. TFAE:
	\begin{enumerate}
		\item $\chi^{el}(\mathcal{H}) \leq \Delta$.
		\item there is a $0$-round GLOCAL algorithm producing a $\Delta$-edge labeled coloring of $\Homed(T^{r}_\Delta,\mathcal{H})$, for every $r$.
		\item  there is an $r$ and an $r$-round GLOCAL algorithm  producing a $\Delta$-edge labeled coloring of $\Homed(T^{r+1}_\Delta,\mathcal{H})$. 
		
	\end{enumerate}
\end{theorem}

\begin{proof}
	The proof is a straightforward modification of the proof of Theorem \ref{t:tr1local}, in particular, the implications (1) $\implies$ (2) $\implies$ (3) follow immediately.
	
	To see that (3) $\implies$ (1), for any $v \in \mathcal{H}$ and $i \leq \Delta$ define the game $\mathbb{G}(v,i)$ as in the proof of Theorem \ref{t:tr1local}, with the additional assumption that the edge connecting the vertex labeled by $v$ and $I$s first move has to be colored $i$. 
	
	It is easy to check that the arguments in Lemma \ref{l:firstplayerloc} and Lemma \ref{l:secondplayerloc} still work, and $c$ is an edge labeled $\Delta$-coloring of $\mathcal{H}$. 
\end{proof}

\subsubsection{Back to LOCAL}

Now we are ready to go back to the ordinary LOCAL model, and deduce Theorem \ref{t:marksmainloc}. Let us make a simple observation first.

\begin{claim}
	\label{cl:obvious}
	Assume that $\mathcal{F}$ is a collection of labeled trees of degree $\leq \Delta$ and $f: \N \to \N$ such that for each $G \in \mathcal{F}$ with $|V(G)|=n$
	\begin{enumerate}
		\item the set of labels has size $ \leq n$,
		\item if $n$ is large enough, the labeling is injective on every $f(n)+1$-sized neighborhood of $G$.
	\end{enumerate}
	Assume moreover that there is no $g \in o(f)$ so that there is a $g(n)$-round GLOCAL algorithm to $\Delta$-color every $n$-sized $G \in \mathcal{F}$.	Then there is no $o(f)$-round deterministic distributed algorithm to $\Delta$-color every $n$-sized $G \in \mathcal{F}$.    
\end{claim}

\begin{proof}
	If there was such a LOCAL algorithm, by the local injectivity of the labeling, we could interpret the labels as IDs and run the algorithms on the elements of $\mathcal{F}$. 
\end{proof}

The following Lemma guarantees the existence of graphs that satisfy simultaneously the negation of (1) in Theorem \ref{t:tr2local} and Proposition \ref{pr:injectivity}. The proof is based on a probabilistic method. 

\begin{lemma}
	\label{l:existenceof} Let $f \in o(\log n)$. There exists a sequence of nicely $\Delta$ edge labeled graphs $(\mathcal{H}_n)$ with the following properties:
	\begin{enumerate}
		\item $|V(\mathcal{H}_n)|\leq n$.
		\item $\chi^{el}(\mathcal{H}_n)>\Delta$.
		\item the girth of $\mathcal{H}_n$ is $ > 2f(n)+2$.
	\end{enumerate}	
\end{lemma}
\begin{proof}[Proof sketch.]
	Such a sequence can be constructed probabilistically, using the \emph{configuration model} for regular random graphs, see \cite{wormald1999models}.
    Namely, on the vertex set $\{1,2,\dots, n\}$ we independently choose $\Delta$ many $d$-regular random graphs and let $\mathcal{H}_n$ to be the union of those, for a large constant $d$.
    Note that this gives a canonical edge labeling of $\mathcal{H}_n$ with $\Delta$ colors with the property that every vertex is adjacent to exactly $d$ edges of each of the colors.
    
    By the results of Bollob\'as \cite{bollobas} on the independence ratio of random regular graphs and concentration inequalities of McDiarmid \cite{mcdiarmid2002concentration} it can be shown that $\chi^{el}(\mathcal{H}_n)>\Delta$ asymptotically almost surely as $n\to \infty$.
	Finally, seeing the sample $\mathcal{H}_n$ as a random $d\Delta$-regular graph, one can use \cite[Corollary~1]{mckayshortcycles} to show that (3) is satisfied with positive probability.	
\end{proof}

Now we can prove the main theorem on colorings.

\begin{proof}[Proof of Theorem \ref{t:marksmainloc}]
	Towards a contradiction, assume that there exists a LOCAL $\Delta$-coloring algorithm for acyclic graphs of degree $\leq \Delta$ with running time $g \in o(\log n)$. We can assume that $g$ is monotone. Let $(\mathcal{H}_n)_{n \in \N}$ be the sequence guaranteed by Lemma \ref{l:existenceof} for the function $f(n)=g(\Delta n)$.
	
	Consider the family of graphs $\mathcal{F}=\bigcup_{n} \Homed(T^{r_n}_\Delta,\mathcal{H}_n)$, where $r_n$ is chosen minimal so that $|V(T^{r_n}_\Delta)| \geq n $. 
	
	\begin{claim}
		For every graph in $\Homed(T^{r_n}_\Delta,\mathcal{H}_n)$ the labeling is injective in every $f(n)+1$-sized neighborhood.
	\end{claim}
	\begin{proof}
		Indeed, observe that, as the homomorphism defined by the labels has to preserve edge colors and the girth of $\mathcal{H}_n$ is $> 2f(n)+2$, it must be injective on every $f(n)+1$ sized neighborhood by Proposition \ref{pr:injectivity}.
	\end{proof}
	
	By the minimality of $r_n$, we have that $|V(T^{r_n}_\Delta)| \leq \Delta n$. Thus, the labeling of every graph in $\Homed(T^{r_n}_\Delta,\mathcal{H}_n)$ is injective in every neighborhood of size \[f(n)+1 =g(\Delta n)+1 \geq g(|V(T^{r_n}_\Delta)|)+1,\] showing that the conditions of Claim \ref{cl:obvious} are satisfied. Hence, on the one hand, there is an $o(g)$-time GLOCAL algorithm for the $\Delta$-coloring of $\mathcal{F}$. 
	
	On the other hand, $r_n>c \log n$, for some $c>0$ and, as $g \in o(\log n)$, for a large enough $n$ we have $g(n)<c\log n<r_n$. Therefore, by assumption, there is a $g(n)<r_n$ round GLOCAL algorithm to $\Delta$-color $\Homed(T^{r_n}_\Delta,\mathcal{H}_n)$, which is in particular an $r_n-1$-round GLOCAL algorithm. By Theorem \ref{t:tr2local} this implies $\chi^{el}(\mathcal{H}_n) \leq\Delta$, contradicting the choice of $\mathcal{H}_n$.

\end{proof}

\subsubsection{A localized transfer} As we will see in the descriptive case in the next section (Theorem \ref{t:transfer}), even descriptive combinatorial properties can be transferred through homomorphism graphs. In order to mimic this in the distributed version, we have to consider ``localized" versions of the graph $\mathcal{H}$ instead, i.e., instead of a fixed enumeration, we consider all the possible ID distributions on $\mathcal{H}$.

\begin{definition}
	Let $\HOM(T^r_\Delta,\mathcal{H}^{ID})$ be the set of labelings of $T^r_\Delta$ that are the pullbacks of IDs (i.e., injective labelings with naturals $<|V(\mathcal{H})|$) on $\mathcal{H}$ through a homomorphism.
	
	Similarly, if $\mathcal{H}$ is equipped with an edge $\Delta$-labeling, let $\Homed(T^r_\Delta,\mathcal{H}^{ID})$ be the set of labelings of $T^r_\Delta$ that are the pullbacks of IDs on $\mathcal{H}$ through an edge label preserving homomorphism.
\end{definition}

Now we have the following.

\begin{theorem}[Transfer principle $3$, LOCAL]
	\label{t:tr3local}

	Let $\mathcal{H}$ be a finite graph.
	
	\begin{enumerate}
		\item If $\HOM(T^{r+1}_\Delta,\mathcal{H}^{ID})$ has an $r$-round $\Delta$-coloring GLOCAL algorithm, then $\mathcal{H}$ has an $r$-round $\Delta$-coloring in the LOCAL model.
		\item Assume that $\mathcal{H}$ is equipped with a nice edge labeling. If $\Homed(T^{r+1}_\Delta,\mathcal{H}^{ID})$ has an $r$-round GLOCAL algorithm for edge labeled $\Delta$-coloring, then $\mathcal{H}$ has an $r$-round edge labeled $\Delta$-coloring in the LOCAL model.

	\end{enumerate}
	
\end{theorem}

\begin{proof}
	The proofs are the same as the proofs of the (3) $\implies$ (1) implications of Theorems \ref{t:tr1local} and \ref{t:tr2local}. The only thing to observe is that the coloring defined as in Lemma \ref{l:secondplayerloc} in fact only depends on the IDs in an $r$-round neighborhood of $v$ in $\mathcal{H}$, hence it is doable in the LOCAL model.   
\end{proof}

The downside here that there is no hope for an equivalence in general, as the existence of a LOCAL algorithm on $\mathcal{H}$ can happen for trivial reasons: for example if $\mathcal{H}$ is star-like, with constant round LOCAL algorithm one can access global information. However, this information is not present on the homomorphism graph, because of the degree bound. 
\begin{remark}
We conjecture that using this transfer principle one can give a proof of Theorem \ref{t:marksmainloc} without having to go through Lemma \ref{l:existenceof}. 
\end{remark}

\subsection{The descriptive case}

To close our discussion of Marks' method, let us mention that statements completely parallel to the ones described in the preceding section have been shown in the Borel context in \cite{DetMarks,brandt_chang_grebik_grunau_rozhon_vidnyaszkyhomomorphisms}. 

Marks initially proved the following. Consider the group \[\Gamma_\Delta=\langle \alpha_1,
\dots,\alpha_\Delta:\alpha^2_1=\dots=\alpha^2_\Delta=1\rangle,\]
that is, $\Gamma_\Delta$ is the free product of $\Delta$-many involutions.
Observe that $\Cay(\Gamma_\Delta,S)$ is a $\Delta$-regular tree, where $S=\{\alpha_1,\dots,\alpha_\Delta\}$.

\begin{theorem}
	\label{t:marksmain}
	Let $\Delta >2$. Then there exists a $\Delta$-regular acyclic Borel graph with $\chi_B(G)=\Delta+1$. In fact, this holds for the Schreier graph of the free part of the action of $\Gamma_{\Delta}$ on $2^{\Gamma_\Delta}$.
\end{theorem}
In the proof of this theorem, the players play on the infinite tree, games analogous to the ones described above. Interestingly, the existence of winning strategies is not at all clear: every acyclic $\Delta$-regular Borel graph has a $\mu$-measurable $\Delta$-coloring by the results of Conley-Marks-Tucker-Drob \cite{conley2016brooks} for any Borel measure $\mu$, described above, showing that winning strategies might not exist even for Lebesgue-measurable sets. However, since the winning conditions in the games are Borel, to guarantee their existence, one can use Martin's Borel Determinacy theorem, one of the cornerstones of descriptive set theory. 

The problem of non-injectivity of IDs in the Borel context corresponds to the players having to end up in the free part of the above action.  Marks remedied  this by arguing that there exists a Borel $\Delta$-color assignment $c'$ of the Schreier graph on the non-free part, with the property that for no $x$ and $i$ we have $c'(x)=c'(\gamma_i \cdot x)=i$. That is, $\chi^{el}_B$ (defined analogously to Definition \ref{d:weakcoloring}) of the not-free part of the Schreier graphs is $\leq \Delta$. So, if there exists a Borel $\Delta$-coloring of the free part then $\chi^{el}_B(Sch(\Gamma,2^{\Gamma_\Delta})) \leq \Delta$. Now, we can derive a contradiction similarly to the proof of Theorem \ref{t:tr2local}. 

The homomorphism graph can be defined in the Borel context as follows: observe that if $\mathcal{H}$ is a Borel graph, $V(\mathcal{H})^{\Gamma_{\Delta}}$ is a Borel space on which $\Gamma_\Delta$ acts by left-shift.

\begin{definition}
	Let $\mathcal{H}$ be a Borel graph. Let $\HOM(\Gamma_\Delta,\mathcal{H})$ be the restriction of $\schreier{\Gamma_\Delta}{S}{V(\mathcal{H})^{\Gamma_\Delta}}$ to the set
	\[\{h \in V(\mathcal{H})^{\Gamma_\Delta}: \text{$h$ is a graph homomorphism from $\Cay(\Gamma_\Delta,S)$ to $V(\mathcal{H}$)}\}.\]
\end{definition}

Similarly to Lemma \ref{l:secondplayerloc}, one can define a coloring on $\mathcal{H}$ by using a ``voting system" based on the information about the winning strategies in the games analogous to the ones discussed above. Since this requires to know which player has a winning strategy, the resulting coloring will be slightly worse than Borel, namely, in the class $\Game\mathbf{\Delta}^1_1$. Moreover, assuming extra set theoretic axioms and considering a slightly larger class, we can even obtain a full equivalence (below, $\chi_{pr}$ stands for projective chromatic number):
\begin{theorem}[\cite{brandt_chang_grebik_grunau_rozhon_vidnyaszkyhomomorphisms}, transfer principle $1$, Borel]
	\label{pr:chromatic}
 Let $\mathcal{H}$ be a Borel graph.
	\begin{itemize}
		\item $\HOM(\Gamma_\Delta,\mathcal{H})$ admits a Borel homomorphism into $\mathcal{H}$. Thus, for any $n$, if $\chi_B(\mathcal{H}) \leq n$ then $\chi_B(\HOM(\Gamma_\Delta,\mathcal{H})) \leq n$.
		\item If $\chi_{\Game\mathbf{\Delta}^1_1} (\mathcal{H})>\Delta$ then $\chi_B(\HOM(\Gamma_\Delta,\mathcal{H})) >\Delta$.
	\end{itemize}
	Assuming suitable large cardinal hypotheses we also have
	\[\chi_{pr} (\mathcal{H})>\Delta \iff \chi_{pr}(\HOM(\Gamma_\Delta,\mathcal{H})) >\Delta.\]
	
\end{theorem}

 Finally, observe that $\Cay(\Gamma_\Delta,S)$ is equipped with a proper edge coloring. Hence, the edge labeling trick is sufficient to produce acyclic graphs here as well. 
 
 \begin{definition}	
 	Assume that the graph $\mathcal{H}$ is equipped with a Borel edge $\Delta$-labeling. Let $\Homed(\Gamma_\Delta,\mathcal{H})$ be the restriction of $\HOM(\Gamma_\Delta,\mathcal{H})$ to the set
 	\[\{h \in V(\HOM(\Gamma_\Delta,\mathcal{H})): \text{$h$ preserves the edge labels}\}.\]
 \end{definition}

 In case $\mathcal{H}$ is acyclic, an argument similar to the proof of Proposition \ref{pr:injectivity} yields that each $h \in \HOM(\Gamma_\Delta,\mathcal{H})$ is injective. Using this, one can show the following.

\begin{theorem}[Transfer principle $2$, Borel]
	\label{t:transfer}
	Assume that $\mathcal{H}$ is a locally countable Borel graph and $e$ is a Borel edge $\Delta$-labeling of $\mathcal{H}$. 	
	\begin{itemize}
		
		\item $\Homed(\Gamma_\Delta,\mathcal{H})$ admits a label preserving Borel homomorphism into $\mathcal{H}$. Thus, if $\chi^{el}_B(\mathcal{H}) \leq \Delta$ then $\chi_B(\Homed(\Gamma_\Delta,\mathcal{H})) \leq \Delta$.
        \item If $\mathcal{H}$ is acyclic then so is $\Homed(\Gamma_\Delta,\mathcal{H})$.
		\item If $\mathcal{H}$ is acyclic and hyperfinite then so is $\Homed(\Gamma_\Delta,\mathcal{H})$.
		\item If $\chi^{el}_{\Game\mathbf{\Delta}^1_1} (\mathcal{H})>\Delta$ then $\chi_B(\Homed(\Gamma_\Delta,\mathcal{H})) >\Delta$.
	\end{itemize}
	Assuming suitable large cardinal hypotheses we also have
	\[\chi^{el}_{pr} (\mathcal{H})>\Delta \iff \chi^{el}_{pr}(\HOM(\Gamma_\Delta,\mathcal{H})) >\Delta.\]

\end{theorem}

Since a suitable edge labeled version of $\mathcal{H}=\mathbb{G}_0$ (see \cite{KST,brandt_chang_grebik_grunau_rozhon_vidnyaszkyhomomorphisms}) is hyperfinite and $\chi^{el}_{\Game\mathbf{\Delta}^1_1} (\mathcal{H})>\Delta$, Theorem \ref{t:transfer} yields Theorem \ref{t:marksmain} and can be also used to obtain the results in \cite{conleyhyp}.

\section{Versions of Vizing's theorem}\label{sec:Vizing}

In this section we describe the ideas behind the proof of the measurable version of Vizing's theorem and their applications to the LOCAL model of distributed computing.
\vspace{+0.2cm}

The strategy to construct edge colorings both in measurable and distributed context is to gradually improve a partial coloring using augmenting subgraphs.
To get a feeling about this type of construction in measurable combinatorics, we start by revisiting the result of Lyons and Nazarov \cite{lyons2011perfect}, where they produced a \emph{perfect matching on the $\Delta$-regular tree as a factor of iid}.
Their overall approach, that originated in the work of Nguyen and Onak \cite{nguyen_onak2008matching} that was later formalized by Elek and Lippner \cite{eleklippner}, turned out to be extremely important not only for edge colorings, but also for applications that include for example the measurable circle squaring, or paradoxical decompositions, see \cite{OlegExpasion,measurablesquare,FolnerTilings}.

The argument that Lyons and Nazarov employed to show that the final matching, after the augmenting procedure that might take an infinite number of steps, is well-defined almost surely is conceptually different than the one for the measurable Vizing's theorem.
Namely, their argument uses the fact that the underlying measured graph, iid on $\Delta$-regular tree, \emph{expands in measure} which is a form of \emph{global} expansion.
Since not every acyclic graphing expands in measure, and interestingly Kun \cite{KunPerfectMatching} found examples of $\Delta$-regular acyclic graphings that do not admit measurable perfect matching, we rely on a different tool in the proof of measurable Vizing's theorem: a form of \emph{local expansion} introduced in \cite{grebik2020measurable} that was later generalized in the notion of a \emph{multi-step Vizing chain} by Bernshteyn \cite{BernshteynVizing}.
This type of local expansion have proved to be particularly well-suited for the construction of distributed algorithms \cite{BernshteynVizing,Christiansen,BernshteynVizing2}.
\vspace{+0.2cm}

After discussing the important ideas behind the result of Lyons and Nazarov \cite{lyons2011perfect} in the next subsection, we formally introduce the concept of multi-step Vizing chains.
In the last part, we summarize the applications that include measurable Vizing's theorem and the construction of distributed algorithms of LOCAL complexity $\operatorname{poly}(\log(n))$.

\subsection{Perfect matching on regular trees}\label{subsec:LyonsNazarov}
Recall that a \emph{matching} in a graph $G=(V,E)$ is a collection of edges $M\subseteq E$ such that for every $v\in V$ there is at most one $e\in M$ such that $v\in e$.
We denote as $U_M$ the set of \emph{unmatched vertices of $M$}, that is, $v\in U_M$ if and only if $v\not\in e$ for every $e\in M$.
Vertices that are not unmatched are \emph{covered by $M$}.
We say that $M$ is a \emph{perfect matching} if $U_M=\emptyset$.

A classical result of K\H{o}nig states that a bipartite $\Delta$-regular graph $G$ satisfies $\chi'(G)=\Delta$.
In particular, $G$ admits a perfect matching.
Analogues of K\H{o}nig's theorem in measurable combinatorics are simply false.
First, Laczkovich \cite{laczk} gave an example of a $2$-regular measurably bipartite graphing that does not admit a measurable perfect matching.
Then Marks \cite{DetMarks} showed that there are acyclic $\Delta$-regular Borel bipartite Borel graphs that do not admit Borel perfect matching for every $\Delta\ge 3$.
Finally, Kun \cite{KunPerfectMatching} strengthened Marks' result by finding an acyclic $\Delta$-regular measurably bipartite graphing that does not admit measurable perfect matching for every $\Delta\ge 3$.

On the positive side, Lyons and Nazarov \cite{lyons2011perfect} showed that there is a factor of iid perfect matching on the $\Delta$-regular tree, which implies, by the result of Hatami, Lov\'asz and Szegedy \cite{hatamilovaszszegedy}, that \emph{every} acyclic $\Delta$-regular graphing admits a measurable matching such that the measure of unmatched vertices is arbitrarily small.
\vspace{+0.2cm}

In our exposition of the arguments from \cite{lyons2011perfect} we follow the approach of Elek and Lippner \cite{eleklippner}.
Basically all measurable constructions that are based on augmenting subgraph technique use some version of Theorem~\ref{thm:ElekLippner} in the background.

\begin{definition}
    Let $M$ be a matching in $G=(V,E)$.
    We say that
    $$P=(v_1,v_2,\dots,v_k)\subseteq V$$
    is an \emph{($M$-)augmenting path} if $v_1,v_k\in U_M$ and $\{v_{2i},v_{2i+1}\}\in M$ for every $i\in \{1,\dots,k/2-1\}$.

    \emph{Flipping the edges} on an augmenting path $P=(v_1,v_2,\dots,v_k)$ modifies the matching $M$ to a matching $M'$ that agrees with $M$ outside of $P$ and such that $\{v_{2i-1},v_{2i}\}\in M'$ for every $i\in \{1,\dots, k/2\}$.
\end{definition}

Note that the length of every augmenting path $P$ is even and that all vertices of $P$ are covered after flipping the edges.

We state a version of the result of Elek and Lippner \cite{eleklippner}, which is suitable in our context.

\begin{theorem}[Elek--Lippner \cite{eleklippner}]\label{thm:ElekLippner}
    Let $\fG=(V,E,\fB,\mu)$ be a graphing of degree bounded by $\Delta\in \mathbb{N}$, $M$ be a measurable matching in $\fG$ and $k\in \mathbb{N}$.
    Then there is a measurable matching $M'\subseteq E$ such that
    \begin{equation}\label{eq:BoundMeasure}
        \mu(\{v\in V:\exists w\in N_\fG(v) \ \{v,w\}\in M'\setminus M\})\le k\mu(U_M)
    \end{equation}
    and every $M'$-augmenting path has length at least $k$.
\end{theorem}
\begin{proof}[Proof sketch]
    Let $\fG^{(2k+1)}$ denote the $2k+1$st power of $\fG$, that is, $x,y\in V$ form an edge in $\fG^{(2k+1)}$ if their graph distance in $\fG$ is at most $2k+1$.
    It is easy to check that $\fG^{(2k+1)}$ is a graphing of bounded maximum degree.
    By a straightforward application of Theorem~\ref{thm:BorelGreedy}, there is a sequence $(B_i)_{i\in \mathbb{N}}$ of Borel $\fG^{(2k+1)}$-independent sets such that $\{i\in \mathbb{N}:v\in B_i\}$ is infinite for every $v\in V$.

    Set $M_0=M$ and inductively define a sequence of measurable matchings $(M_i)_{i\in \mathbb{N}}$ as follows.
    Suppose that $M_i$ has been defined.
    If possible pick an $M_i$-augmenting path $P_v$ in $\mathcal{G}$ of length at most $k$ that starts at $v$ for every $v\in B_i$.
    This assignment can be done in Borel way.
    Note that if $v\not=w\in B_i$, then $P_v$ and $P_w$ are edge disjoint.
    In particular, we may define the measurable matching $M_{i+1}$ by flipping the edges on all of these $M_i$-augmenting paths simultaneously.

    We claim that $M'=\lim_{i\to\infty} M_i$ is a well-defined measurable matching that satisfies the claim of the theorem.
    Indeed, the limit is well-defined as every edge is flipped only finitely many times.
    This follows from the combination of the facts that every vertex lies on finitely many paths of length at most $k$ and that once a vertex is covered by $M_i$, then it is covered by $M_j$ for every $i\le j\in \mathbb{N}$.
    Consequently, if $P$ is a finite $M'$-augmenting path starting at $v\in V$, then there is an $i_0\in \mathbb{N}$ such that $P$ is an $M_i$-augmenting path for every $i\ge i_0$.
    In particular, the length of $P$ has to be at least $k$, as otherwise it would be flipped for any $i\ge i_0$ such that $v\in B_i$.

    Finally, \eqref{eq:BoundMeasure} follows from the mass transport principle as every vertex from $U_M$ is responsible for flipping edges at $k$-many other vertices. 
\end{proof}

Let $\fG=(V,E,\fB,\mu)$ be an acyclic $\Delta$-regular graphing and consider the following process.
Start with an empty matching $M_0$ and define inductively $M_{k+1}$ from $M_{k}$ using Theorem~\ref{thm:ElekLippner} so that every $M_{k+1}$-augmenting path has length at least $2k$ (as our main application, Theorem~\ref{thm:Lyons2011}, is for bipartite graphs, we use $2k$ instead of just $k$).
This produces a sequence of measurable matchings $(M_k)_{k\in \mathbb{N}}$.

We would like to claim now that that the sequence $(M_k)_{k\in \mathbb{N}}$ admits almost surely a pointwise limit $M=\lim_{k\to \infty} M_k$ that is the desired measurable perfect matching.
Recall that the pointwise limit is defined by viewing each $M_k$ as its characteristic function.
However, such a limit need not exist as it might be the case that a positive fraction of edges is flipped in the sequence $(M_k)_{k\in \mathbb{N}}$ infinitely often.

Observe that a sufficient condition for the pointwise limit $M=\lim_{k\to \infty} M_k$ to exist almost surely is if
\begin{equation}\label{eq:BC}
    \sum_{k=0}^\infty 2k\mu(U_{M_k})<\infty.
\end{equation}
This follows from a combination of Theorem~\ref{thm:ElekLippner} and Borel--Cantelli lemma.
Hence, it is enough to find a condition that guarantees that $\mu(U_{M_k})\in o(1/k^2)$.
This is the place where expansion in measure kicks in.
To keep things simple, we introduce the concept of measure expansion for measurably bipartite graphs only.

\begin{definition}
    Let $\eta>0$ and $\fG=(V,E,\fB,\mu)$ be an acyclic $\Delta$-regular graphing that admits a measurable bipartition $V=A\sqcup B$.
    We say that $\fG$ is an \emph{$\eta$-expander} if $\mu(N_\fG(X))\ge (1+\eta)\mu(X)$ for every $X\subseteq V$ such that $X\subseteq A$ or $X\subseteq B$, and $\mu(X)\le 1/4$.
\end{definition}

Now we are ready to state a weak version of the result of Lyons and Nazarov \cite{lyons2011perfect}.

\begin{theorem}[Lyons--Nazarov \cite{lyons2011perfect}]\label{thm:Lyons2011}
    Let $\eta>0$ and $\fG=(V,E,\fB,\mu)$ be an acyclic $\Delta$-regular measurably bipartite graphing that is an $\eta$-expander.
    Then $\fG$ admits a measurable perfect matching.
\end{theorem}
\begin{proof}[Proof sketch]
    It is enough to verify that the sequence of measurable matching $(M_k)_{k\in \mathbb{N}}$ defined after Theorem~\ref{thm:ElekLippner} satisfies $\mu(U_{M_k})\in o(1/k^2)$.
    
    Observe that as $\fG$ is regular and satisfies the mass transport principle, we have that $\mu(A)=\mu(B)=1/2$.
    In fact, we have $\mu(U_{M_k}\cap A)=\mu(U_{M_k}\cap B)$ for every $k\in \mathbb{N}$.
    Set $X^k_0=U_{M_k}\cap A$ and define inductively a sequence of measurable sets $(X^k_\ell)_{\ell\in \mathbb{N}}$ as follows.
    If $\ell$ is even, then set $X^k_{\ell+1}=N_\fG(X^k_\ell)\subseteq B$, and if $\ell$ is odd, then set
    $$X^k_{\ell+1}=\{v\in V:\exists w\in X^k_\ell \ \{v,w\}\in M_k\}\subseteq A.$$
    Let $\ell_0\in \mathbb{N}$ be the first (odd) number such that either $\mu(X^k_{\ell_0})> 1/4$ or $\mu(X^k_{\ell_0}\cap U_{M_k})>0$.
    This is well-defined as $\fG$ is an $\eta$-expander.

    Observe that for every $0\le j\le (\ell_0+1)/2$, we have that
    $$(1+\eta)^{j+1}\mu(X^k_0)\le \mu(X^k_{2j+1})=\mu(X^k_{2j+2})$$
    as $\fG$ is a graphing and an $\eta$-expander.
    We define analogously a sequence $(Y^k_\ell)_{\ell\in \mathbb{N}}$ and $\ell'_0\in \mathbb{N}$ for $Y^k_0=U_{M_k}\cap B$.
    In particular,
    $$\mu(U_{M_k})= \mu(X^k_0)+\mu(Y^k_0)\le 2(1+\eta)^{-(\max\{\ell_0,\ell'_0\}+1)/2}.$$
    We finish the proof by showing that $k\le \max\{\ell_0,\ell'_0\}+1$
    
    First observe that if $\mu(X^k_{\ell_0}\cap U_{M_k})>0$, then there are $M_k$-augmenting paths of length $\ell_0+1$.
    By symmetry the same holds for $Y^k_{\ell'_0}$.
    In particular, if $\max\{\ell_0,\ell'_0\}+1<2k$ it must be the case that 
    $$\mu(X^k_{\ell_0}),\mu(Y^k_{\ell'_0})> 1/4.$$
    But then we conclude that $\mu(X^k_{\ell_0+1}\cap Y^k_{\ell_0'})>0$ as $X^k_{\ell_0+1},Y^k_{\ell'_0}\subseteq A$, $\mu(A)=1/2$ and $\mu(X^k_{\ell_0+1})=\mu(X^k_{\ell_0})>1/4$.
    Consequently, this guarantees the existence of $M_k$-augmenting paths of length $\ell_0+1+\ell'_0+1$.
    Indeed, consider a concatenation of any paths from $X^k_0$ to $X^k_{\ell_0+1}\cap Y^k_{\ell_0'}$ and from $X^k_{\ell_0+1}\cap Y^k_{\ell_0'}$ to $Y^k_0$.
    Altogether, this shows $k\le \max\{\ell_0,\ell'_0\}+1$ and the proof is finished.
\end{proof}

\begin{remark}[Approximate K\H{o}nig's line coloring theorem]
    It was observed, for example, by Lyons \cite{LyonsTrees} that the iterative application of these arguments yield a factor of iid edge coloring with $\Delta+1$ colors of the $\Delta$-regular tree such that one of the color classes can have arbitrarily small measure.
    In fact this holds for all bipartite measured graphs of degree bounded by $\Delta$, see \cite{TothShcreier,grebikApprox,GrebikVizing}.
\end{remark}

\subsection{Multi-step Vizing chains}\label{subsec:Multi}
Our goal is to define the notion of a multi-step Vizing chain and discuss the connections to local expansion.
We start by recalling some notation from \cite{grebik2020measurable}.
\vspace{+0.2cm}

Let $G=(V,E)$ be a finite (or countable) connected graph and set $\Delta=\Delta(G)\in \mathbb{N}$.
For the rest of this section we fix a partial edge coloring $d:E\rightharpoonup [\Delta+1]$.
We write $m_d(v)$ for the set of colors missing at $v\in V$, that is, $\alpha\not\in m_d(v)$ if there is an $e\in \dom(d)$ such that $v\in e$ and $d(e)=\alpha$.
Observe that $|m_d(v)|>0$ for every $v\in V$.

A \emph{chain} is a sequence $P=(e_0,e_1,\dots )\subseteq E$ of edges such that for every $i\in \mathbb{N}$ with $e_i,e_{i+1}\in P$ we have $e_i\cap e_{i+1}\not=\emptyset$.
Let $l(P)=|P|$ denote the the number of edges in $P$, we allow $l(P)$ to be infinite.
We say that $P$ is \emph{edge injective} if every edge appears at most once in $P$.
If $P$ is a finite chain with the last edge $e$ and $Q$ is a chain with the first edge $f$ and $e\cap f\not=\emptyset$, then we write $P^\frown Q$ for the chain that is the \emph{concatenation} of $P$ and~$Q$.
If a chain $R$ is of the form $P^\frown Q$, then we say that $P$ is a \emph{prefix} of $R$, denoted by $P\sqsubseteq R$.

We say that a chain $P=(e_i)_{i<l(P)}$ is \emph{$d$-shiftable} if $l(P)\ge 1$, $P$ is edge injective,  $e_0\in U_d$ and $e_j\in \dom(d)$ for every $1\le j<l(P)$.
A $d$-shiftable chain $P$ is called \emph{$d$-proper-shiftable} if $d_P:E\rightharpoonup [\Delta+1]$, the \emph{$P$-shift of $d$}, defined as 
        \begin{itemize}
   	    \item $\dom(d_P)=\dom(d)\cup\{e_0\}\setminus \{e_{l(P)-1}\}$ where we put $\{e_{l(P)-1}\}=\emptyset$ if $l(P)=\infty$,
            \item $d_P(e_i)=d(e_{i+1})$ for every $i+1<l(P)$,
            \item $d_P(f)=d(f)$ for every $f\in \dom(d)\setminus P$;
        \end{itemize}
is a partial edge coloring.
Finally, we say that a $d$-proper-shiftable chain is \emph{$d$-augmenting} if either $l(P)=\infty$ or $P$ is finite with $m_{d_P}(x)\cap m_{d_P}(y)\not=\emptyset$ where $x\not=y$ are the vertices of the last edge $e_{l(P)-1}$ of~$P$.
\vspace{+0.2cm}

A \emph{fan around $x\in V$ starting at $e\in E$} is a finite edge injective chain $F=(e_0,e_1,\dots,e_k)$ such that $x\in e_j$ for every $0\le j\le k$ and $e_0=e$.
If we don't want to specify $x\in V$ and $e\in E$, we simply say that $F$ is a \emph{fan}.

The \emph{$\alpha/\beta$-path starting at $x\in V$}, denoted by $P_d(x,\alpha/\beta)$, is the unique maximal chain $P=(e_i)_{i<l(P)}$ such that $d(e_{i})=\alpha$ (resp.\ $d(e_{i})=\beta$) for every $i<l(P)$ that is even (resp.\ odd).
A \emph{truncated $\alpha/\beta$-path} is a prefix of any $\alpha/\beta$-path $P_d(x,\alpha/\beta)$.
If we dont want to specify the colors in the path, we simply say that $P=P_d(x,\alpha/\beta)$ is an \emph{alternating path}.

\begin{definition}[Vizing Chain]
    Let $e\in U_d$ and $x\in e$.
    A \emph{Vizing chain} $W(x,e)$ is a $d$-proper-shiftable chain of the form $W(x,e)={F_e}^\frown P_{\alpha_e,\beta_e}$, where $F$ is a fan at $x$ starting at $e$ and $P$ is an $\alpha_e/\beta_e$-path for some $\alpha_e\not=\beta_e\in [\Delta+1]$, see Figure~\ref{fig:Vizing}.
\end{definition}

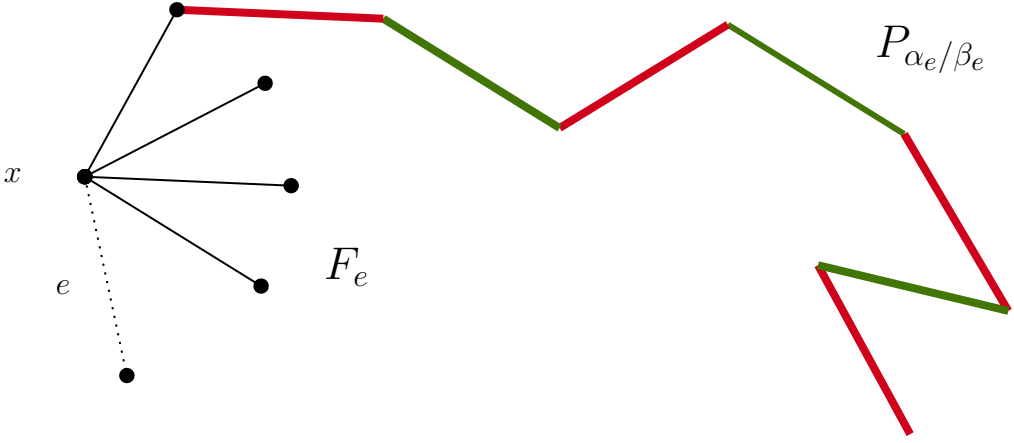
\begin{figure}
    \centering

\tikzset{every picture/.style={line width=0.75pt}} %set default line width to 0.75pt        

\begin{tikzpicture}[x=0.75pt,y=0.75pt,yscale=-1,xscale=1]
%uncomment if require: \path (0,291); %set diagram left start at 0, and has height of 291

%Straight Lines [id:da6945171873939207] 
\draw    (118,132) -- (206,187) ;
\draw [shift={(206,187)}, rotate = 32.01] [color={rgb, 255:red, 0; green, 0; blue, 0 }  ][fill={rgb, 255:red, 0; green, 0; blue, 0 }  ][line width=0.75]      (0, 0) circle [x radius= 3.35, y radius= 3.35]   ;
\draw [shift={(118,132)}, rotate = 32.01] [color={rgb, 255:red, 0; green, 0; blue, 0 }  ][fill={rgb, 255:red, 0; green, 0; blue, 0 }  ][line width=0.75]      (0, 0) circle [x radius= 3.35, y radius= 3.35]   ;
%Straight Lines [id:da7975143809853416] 
\draw  [dash pattern={on 0.84pt off 2.51pt}]  (118,132) -- (139,232) ;
\draw [shift={(139,232)}, rotate = 78.14] [color={rgb, 255:red, 0; green, 0; blue, 0 }  ][fill={rgb, 255:red, 0; green, 0; blue, 0 }  ][line width=0.75]      (0, 0) circle [x radius= 3.35, y radius= 3.35]   ;
\draw [shift={(118,132)}, rotate = 78.14] [color={rgb, 255:red, 0; green, 0; blue, 0 }  ][fill={rgb, 255:red, 0; green, 0; blue, 0 }  ][line width=0.75]      (0, 0) circle [x radius= 3.35, y radius= 3.35]   ;
%Straight Lines [id:da3944170298879295] 
\draw [color={rgb, 255:red, 65; green, 117; blue, 5 }  ,draw opacity=1 ][line width=2.25]    (439,55.5) -- (527,110.5) ;
%Straight Lines [id:da8492123866301569] 
\draw [color={rgb, 255:red, 208; green, 2; blue, 27 }  ,draw opacity=1 ][line width=3]    (355,107.5) -- (439,55.5) ;
%Straight Lines [id:da17120022357844067] 
\draw [color={rgb, 255:red, 65; green, 117; blue, 5 }  ,draw opacity=1 ][line width=3]    (267,52.5) -- (355,107.5) ;
%Straight Lines [id:da01925797266730922] 
\draw [color={rgb, 255:red, 208; green, 2; blue, 27 }  ,draw opacity=1 ][line width=3]    (164,48) -- (267,52.5) ;
%Straight Lines [id:da8243788466157467] 
\draw    (118,132) -- (164,48) ;
\draw [shift={(164,48)}, rotate = 298.71] [color={rgb, 255:red, 0; green, 0; blue, 0 }  ][fill={rgb, 255:red, 0; green, 0; blue, 0 }  ][line width=0.75]      (0, 0) circle [x radius= 3.35, y radius= 3.35]   ;
\draw [shift={(118,132)}, rotate = 298.71] [color={rgb, 255:red, 0; green, 0; blue, 0 }  ][fill={rgb, 255:red, 0; green, 0; blue, 0 }  ][line width=0.75]      (0, 0) circle [x radius= 3.35, y radius= 3.35]   ;
%Straight Lines [id:da04452985792151942] 
\draw    (118,132) -- (208,85) ;
\draw [shift={(208,85)}, rotate = 332.43] [color={rgb, 255:red, 0; green, 0; blue, 0 }  ][fill={rgb, 255:red, 0; green, 0; blue, 0 }  ][line width=0.75]      (0, 0) circle [x radius= 3.35, y radius= 3.35]   ;
\draw [shift={(118,132)}, rotate = 332.43] [color={rgb, 255:red, 0; green, 0; blue, 0 }  ][fill={rgb, 255:red, 0; green, 0; blue, 0 }  ][line width=0.75]      (0, 0) circle [x radius= 3.35, y radius= 3.35]   ;
%Straight Lines [id:da034181702211812715] 
\draw    (118,132) -- (221,136.5) ;
\draw [shift={(221,136.5)}, rotate = 2.5] [color={rgb, 255:red, 0; green, 0; blue, 0 }  ][fill={rgb, 255:red, 0; green, 0; blue, 0 }  ][line width=0.75]      (0, 0) circle [x radius= 3.35, y radius= 3.35]   ;
\draw [shift={(118,132)}, rotate = 2.5] [color={rgb, 255:red, 0; green, 0; blue, 0 }  ][fill={rgb, 255:red, 0; green, 0; blue, 0 }  ][line width=0.75]      (0, 0) circle [x radius= 3.35, y radius= 3.35]   ;
%Straight Lines [id:da33527780649668726] 
\draw [color={rgb, 255:red, 208; green, 2; blue, 27 }  ,draw opacity=1 ][line width=3]    (484,176.5) -- (530,261.5) ;
%Straight Lines [id:da5197675982027272] 
\draw [color={rgb, 255:red, 208; green, 2; blue, 27 }  ,draw opacity=1 ][line width=3]    (527,110.5) -- (579,199.5) ;
%Straight Lines [id:da06699854868152966] 
\draw [color={rgb, 255:red, 65; green, 117; blue, 5 }  ,draw opacity=1 ][line width=3]    (484,176.5) -- (579,199.5) ;

% Text Node
\draw (236,166.4) node [anchor=north west][inner sep=0.75pt]  [font=\Large]  {$F_{e}$};
% Text Node
\draw (511,54.4) node [anchor=north west][inner sep=0.75pt]  [font=\Large]  {$P_{\alpha _{e} /\beta _{e}}{}$};
% Text Node
\draw (76,126.4) node [anchor=north west][inner sep=0.75pt]    {$x$};
% Text Node
\draw (102,182.4) node [anchor=north west][inner sep=0.75pt]    {$e$};

\end{tikzpicture}

    \caption{An example of a Vizing chain $W(x,e)={F_e}^\frown P_{\alpha_e,\beta_e}$.}
    \label{fig:Vizing}
\end{figure}

Recall that the standard proof of Vizing's theorem can be viewed as a sequential algorithm that at each step finds a Vizing chain $W(x,e)$, for a given uncolored edge $e$, that is $d$-augmenting and improves the current coloring $d$ by passing to $d_{W(x,e)}$ and coloring the last edge of $W(x,e)$.
In particular, we get that a $d$-augmenting Vizing chain exists for every uncolored edge.
\vspace{+0.2cm}

Next, we present a general definition of a multi-step Vizing chain and discuss some of the properties related to local expansion.
The high-level idea is to turn the algorithm for finding a Vizing chain into a branching process by repeatedly truncating the alternating path and growing a new Vizing chain after shifting the colors.

\begin{definition}[Multi-step Vizing chain, \cite{BernshteynVizing,Christiansen}]
    Let $e\in U_d$, $x\in e$ and $i\in\mathbb{N}$.
    An \emph{$i$-step Vizing chain} $V(x,e)$ is a $d$-proper-shiftable chain of the form
    $$V(x,e)={F_1}^\frown {P_1}^\frown\dots^\frown {F_i}^\frown {P_i},$$
    where $F_1$ is a fan at $x$ starting at $e$ and for every $1\le j\le i$ we have that $F_j$ is a fan and $P_j$ is a truncated alternating path, see Figure~\ref{fig:VizingMultiChain}.
\end{definition}

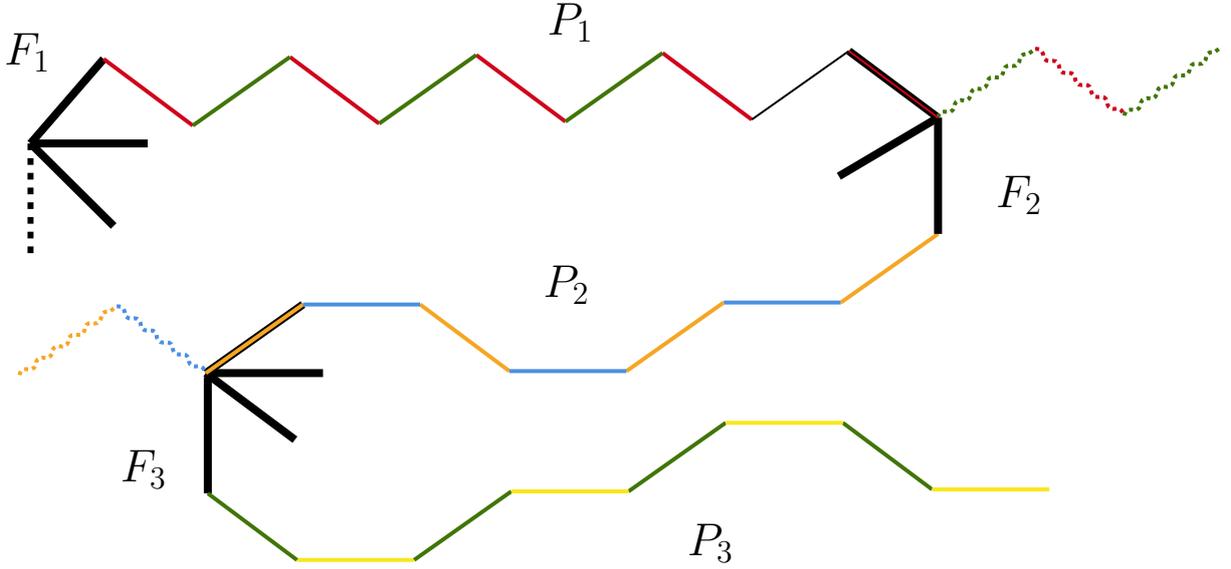
\begin{figure}
    \centering

\tikzset{every picture/.style={line width=0.75pt}} %set default line width to 0.75pt        

\begin{tikzpicture}[x=0.75pt,y=0.75pt,yscale=-1,xscale=1]
%uncomment if require: \path (0,332); %set diagram left start at 0, and has height of 332

%Straight Lines [id:da9129487189237562] 
\draw [color={rgb, 255:red, 248; green, 231; blue, 28 }  ,draw opacity=1 ][line width=1.5]    (173.5,299) -- (232,299) ;
%Straight Lines [id:da7773221057554789] 
\draw [color={rgb, 255:red, 74; green, 144; blue, 226 }  ,draw opacity=1 ][line width=1.5]    (386.5,169.5) -- (445,169.5) ;
%Straight Lines [id:da9728903187115054] 
\draw [color={rgb, 255:red, 208; green, 2; blue, 27 }  ,draw opacity=1 ][line width=1.5]    (356,44) -- (400.5,77.5) ;
%Straight Lines [id:da9725752350444119] 
\draw [line width=3]    (40.5,89.5) -- (77,47) ;
%Straight Lines [id:da3279265706608929] 
\draw [line width=3]    (40.5,89.5) -- (99,89.5) ;
%Straight Lines [id:da8403237527240675] 
\draw [line width=3]    (40.5,89.5) -- (82,131) ;
%Straight Lines [id:da3611804320390526] 
\draw [line width=2.25]  [dash pattern={on 2.53pt off 3.02pt}]  (40.5,89.5) -- (40.5,148) ;
%Straight Lines [id:da5934790554731899] 
\draw [color={rgb, 255:red, 208; green, 2; blue, 27 }  ,draw opacity=1 ][line width=1.5]    (263,45) -- (307.5,78.5) ;
%Straight Lines [id:da44522664047409966] 
\draw [color={rgb, 255:red, 208; green, 2; blue, 27 }  ,draw opacity=1 ][line width=1.5]    (170,46) -- (214.5,79.5) ;
%Straight Lines [id:da019771091801909835] 
\draw [color={rgb, 255:red, 208; green, 2; blue, 27 }  ,draw opacity=1 ][line width=1.5]    (77,47) -- (121.5,80.5) ;
%Straight Lines [id:da7737788792291802] 
\draw    (400.5,77.5) -- (449,43) ;
%Straight Lines [id:da5121673074684605] 
\draw [color={rgb, 255:red, 65; green, 117; blue, 5 }  ,draw opacity=1 ][line width=1.5]    (121.5,80.5) -- (170,46) ;
%Straight Lines [id:da0019068804174739995] 
\draw [line width=3]    (449,43) -- (493.5,76.5) ;
%Straight Lines [id:da4163768333054545] 
\draw [color={rgb, 255:red, 245; green, 166; blue, 35 }  ,draw opacity=1 ][line width=1.5]    (445,169.5) -- (493.5,135) ;
%Straight Lines [id:da9270308254918604] 
\draw [line width=3]    (493.5,135) -- (493.5,76.5) ;
%Straight Lines [id:da10814299957255846] 
\draw [line width=3]    (444,106) -- (493.5,76.5) ;
%Straight Lines [id:da6246276467342591] 
\draw [line width=3]    (128,205) -- (176.5,170.5) ;
%Straight Lines [id:da49653758720322827] 
\draw [line width=3]    (128,205) -- (186.5,205) ;
%Straight Lines [id:da7504711063829081] 
\draw [line width=3]    (128,205) -- (172.5,238.5) ;
%Straight Lines [id:da03989046428911669] 
\draw [line width=3]    (129,207) -- (129,265.5) ;
%Straight Lines [id:da20537208924013095] 
\draw [color={rgb, 255:red, 65; green, 117; blue, 5 }  ,draw opacity=1 ][line width=1.5]    (129,265.5) -- (173.5,299) ;
%Straight Lines [id:da36482584782638217] 
\draw [color={rgb, 255:red, 65; green, 117; blue, 5 }  ,draw opacity=1 ][line width=1.5]    (232,299) -- (280.5,264.5) ;
%Straight Lines [id:da21129300226267955] 
\draw [color={rgb, 255:red, 245; green, 166; blue, 35 }  ,draw opacity=1 ][line width=1.5]    (235,170.5) -- (279.5,204) ;
%Straight Lines [id:da31340798140854487] 
\draw [color={rgb, 255:red, 65; green, 117; blue, 5 }  ,draw opacity=1 ][line width=1.5]  [dash pattern={on 1.5pt off 1.5pt}]  (493.5,76.5) .. controls (493.89,74.17) and (495.24,73.21) .. (497.57,73.6) .. controls (499.9,73.99) and (501.26,73.03) .. (501.65,70.7) .. controls (502.04,68.38) and (503.4,67.42) .. (505.72,67.81) .. controls (508.05,68.2) and (509.41,67.24) .. (509.8,64.91) .. controls (510.19,62.58) and (511.54,61.62) .. (513.87,62.01) .. controls (516.2,62.4) and (517.56,61.44) .. (517.95,59.11) .. controls (518.34,56.78) and (519.69,55.82) .. (522.02,56.21) .. controls (524.35,56.6) and (525.7,55.64) .. (526.09,53.31) .. controls (526.48,50.98) and (527.84,50.02) .. (530.17,50.42) .. controls (532.5,50.81) and (533.85,49.85) .. (534.24,47.52) .. controls (534.63,45.19) and (535.99,44.23) .. (538.32,44.62) -- (542,42) -- (542,42) ;
%Straight Lines [id:da22463970877329853] 
\draw [color={rgb, 255:red, 208; green, 2; blue, 27 }  ,draw opacity=1 ][line width=1.5]  [dash pattern={on 1.5pt off 1.5pt}]  (542,42) .. controls (544.33,41.67) and (545.66,42.68) .. (545.99,45.01) .. controls (546.32,47.34) and (547.66,48.34) .. (549.99,48.01) .. controls (552.32,47.68) and (553.65,48.69) .. (553.98,51.02) .. controls (554.31,53.35) and (555.65,54.36) .. (557.98,54.03) .. controls (560.31,53.7) and (561.64,54.71) .. (561.97,57.04) .. controls (562.3,59.37) and (563.64,60.37) .. (565.97,60.04) .. controls (568.3,59.71) and (569.63,60.72) .. (569.96,63.05) .. controls (570.29,65.38) and (571.63,66.39) .. (573.96,66.06) .. controls (576.29,65.73) and (577.62,66.73) .. (577.95,69.06) .. controls (578.28,71.39) and (579.62,72.4) .. (581.95,72.07) .. controls (584.28,71.74) and (585.61,72.75) .. (585.94,75.08) -- (586.5,75.5) -- (586.5,75.5) ;
%Straight Lines [id:da07294062303394933] 
\draw [color={rgb, 255:red, 74; green, 144; blue, 226 }  ,draw opacity=1 ][line width=1.5]  [dash pattern={on 1.5pt off 1.5pt}]  (83.5,171.5) .. controls (85.83,171.17) and (87.16,172.18) .. (87.49,174.51) .. controls (87.82,176.84) and (89.16,177.84) .. (91.49,177.51) .. controls (93.82,177.18) and (95.15,178.19) .. (95.48,180.52) .. controls (95.81,182.85) and (97.15,183.86) .. (99.48,183.53) .. controls (101.81,183.2) and (103.14,184.21) .. (103.47,186.54) .. controls (103.8,188.87) and (105.14,189.87) .. (107.47,189.54) .. controls (109.8,189.21) and (111.13,190.22) .. (111.46,192.55) .. controls (111.79,194.88) and (113.13,195.89) .. (115.46,195.56) .. controls (117.79,195.23) and (119.12,196.23) .. (119.45,198.56) .. controls (119.78,200.89) and (121.12,201.9) .. (123.45,201.57) .. controls (125.78,201.24) and (127.11,202.25) .. (127.44,204.58) -- (128,205) -- (128,205) ;
%Straight Lines [id:da23369091418595622] 
\draw [color={rgb, 255:red, 245; green, 166; blue, 35 }  ,draw opacity=1 ][line width=1.5]  [dash pattern={on 1.5pt off 1.5pt}]  (35,206) .. controls (35.39,203.67) and (36.74,202.71) .. (39.07,203.1) .. controls (41.4,203.49) and (42.76,202.53) .. (43.15,200.2) .. controls (43.54,197.88) and (44.9,196.92) .. (47.22,197.31) .. controls (49.55,197.7) and (50.91,196.74) .. (51.3,194.41) .. controls (51.69,192.08) and (53.04,191.12) .. (55.37,191.51) .. controls (57.7,191.9) and (59.06,190.94) .. (59.45,188.61) .. controls (59.84,186.28) and (61.19,185.32) .. (63.52,185.71) .. controls (65.85,186.1) and (67.2,185.14) .. (67.59,182.81) .. controls (67.98,180.48) and (69.34,179.52) .. (71.67,179.92) .. controls (74,180.31) and (75.35,179.35) .. (75.74,177.02) .. controls (76.13,174.69) and (77.49,173.73) .. (79.82,174.12) -- (83.5,171.5) -- (83.5,171.5) ;
%Straight Lines [id:da978244684260321] 
\draw [color={rgb, 255:red, 208; green, 2; blue, 27 }  ,draw opacity=1 ]   (449,43) -- (493.5,76.5) ;
%Straight Lines [id:da9234209306058037] 
\draw [color={rgb, 255:red, 65; green, 117; blue, 5 }  ,draw opacity=1 ][line width=1.5]    (214.5,79.5) -- (263,45) ;
%Straight Lines [id:da6540569294857708] 
\draw [color={rgb, 255:red, 65; green, 117; blue, 5 }  ,draw opacity=1 ][line width=1.5]    (307.5,78.5) -- (356,44) ;
%Straight Lines [id:da96844562970687] 
\draw [color={rgb, 255:red, 65; green, 117; blue, 5 }  ,draw opacity=1 ][line width=1.5]  [dash pattern={on 1.5pt off 1.5pt}]  (586.5,75.5) .. controls (586.89,73.17) and (588.24,72.21) .. (590.57,72.6) .. controls (592.9,72.99) and (594.26,72.03) .. (594.65,69.7) .. controls (595.04,67.38) and (596.4,66.42) .. (598.72,66.81) .. controls (601.05,67.2) and (602.41,66.24) .. (602.8,63.91) .. controls (603.19,61.58) and (604.54,60.62) .. (606.87,61.01) .. controls (609.2,61.4) and (610.56,60.44) .. (610.95,58.11) .. controls (611.34,55.78) and (612.69,54.82) .. (615.02,55.21) .. controls (617.35,55.6) and (618.7,54.64) .. (619.09,52.31) .. controls (619.48,49.98) and (620.84,49.02) .. (623.17,49.42) .. controls (625.5,49.81) and (626.85,48.85) .. (627.24,46.52) .. controls (627.63,44.19) and (628.99,43.23) .. (631.32,43.62) -- (635,41) -- (635,41) ;
%Straight Lines [id:da21365089856394626] 
\draw [color={rgb, 255:red, 245; green, 166; blue, 35 }  ,draw opacity=1 ][line width=1.5]    (338,204) -- (386.5,169.5) ;
%Straight Lines [id:da8819423345765927] 
\draw [color={rgb, 255:red, 245; green, 166; blue, 35 }  ,draw opacity=1 ][line width=1.5]    (128,205) -- (176.5,170.5) ;
%Straight Lines [id:da03475662008370528] 
\draw [color={rgb, 255:red, 74; green, 144; blue, 226 }  ,draw opacity=1 ][line width=1.5]    (279.5,204) -- (338,204) ;
%Straight Lines [id:da4882253657722049] 
\draw [color={rgb, 255:red, 74; green, 144; blue, 226 }  ,draw opacity=1 ][line width=1.5]    (176.5,170.5) -- (235,170.5) ;
%Straight Lines [id:da06956785780838892] 
\draw [color={rgb, 255:red, 65; green, 117; blue, 5 }  ,draw opacity=1 ][line width=1.5]    (446,230) -- (490.5,263.5) ;
%Straight Lines [id:da20604075914020403] 
\draw [color={rgb, 255:red, 65; green, 117; blue, 5 }  ,draw opacity=1 ][line width=1.5]    (339,264.5) -- (387.5,230) ;
%Straight Lines [id:da07316140246713765] 
\draw [color={rgb, 255:red, 248; green, 231; blue, 28 }  ,draw opacity=1 ][line width=1.5]    (280.5,264.5) -- (339,264.5) ;
%Straight Lines [id:da631226351542101] 
\draw [color={rgb, 255:red, 248; green, 231; blue, 28 }  ,draw opacity=1 ][line width=1.5]    (387.5,230) -- (446,230) ;
%Straight Lines [id:da6068541213617966] 
\draw [color={rgb, 255:red, 248; green, 231; blue, 28 }  ,draw opacity=1 ][line width=1.5]    (490.5,263.5) -- (549,263.5) ;

% Text Node
\draw (26,32.4) node [anchor=north west][inner sep=0.75pt]  [font=\Large]  {$F_{1}$};
% Text Node
\draw (521,104.4) node [anchor=north west][inner sep=0.75pt]  [font=\Large]  {$F_{2}$};
% Text Node
\draw (84,242.4) node [anchor=north west][inner sep=0.75pt]  [font=\Large]  {$F_{3}$};
% Text Node
\draw (297,17.4) node [anchor=north west][inner sep=0.75pt]  [font=\Large]  {$P_{1}$};
% Text Node
\draw (295,149.4) node [anchor=north west][inner sep=0.75pt]  [font=\Large]  {$P_{2}$};
% Text Node
\draw (367,279.4) node [anchor=north west][inner sep=0.75pt]  [font=\Large]  {$P_{3}$};

\end{tikzpicture}

    \caption{An example of a $3$-step Vizing chain.
    Fans are depicted by black thick edges, while dotted curly edges indicate how the alternating paths continue after they get truncated.    
    }
    \label{fig:VizingMultiChain}
\end{figure}

The concept was first introduced for $i=2$ by the first author and Pikhurko in \cite{grebik2020measurable} and for general $i\in \mathbb{N}$ by Bernshteyn \cite{BernshteynVizing}. 
Note that while the high-level idea described before the definition provides some intuition on how to build multi-step Vizing chains, it is the $d$-proper shiftable requirement on the chains that makes the analysis complicated.
A successful approach to remedy this issue is to count the number of \emph{non-overlapping} multi-step Vizing chains, these might be defined for example by demanding that $P_j\cap P_{j'}=\emptyset$ for every $1\le j<j'\le i$.
A precise analysis of the number of $i$-step Vizing chains for various values of $i$ was done in \cite{grebik2020measurable,BernshteynVizing,Christiansen,BernshteynVizing2}.
In particular, we refer the reader to \cite[Section~2 and 5.1]{BernshteynVizing2} for the description and analysis of the randomized Multi-Step Vizing Algorithm that finds a short augmenting subgraph with high-probability.
\vspace{+0.2cm}

In the remainder of this section, we discuss the connections to local expansion to complement the discussion in Section~\ref{subsec:LyonsNazarov}.

Consider the following problem.
Let $G$ be a finite graph and suppose that for every $e\in U_d$ there is no augmenting subgraph of size at most $\ell\gg \Delta$.
In particular, every ($1$-step) Vizing chain $W(x,e)=F_1^\frown P_1$ assigned to $e\in U_d$ and $x\in e$ has size at least $\ell$.
What can we say about the fraction of uncolored edges $|U_d|/m$?

The \emph{key property} of multi-step Vizing chains, which is essential in all the arguments in \cite{grebik2020measurable,BernshteynVizing,Christiansen,BernshteynVizing2}, is that every edge $f\in \dom(d)$ belongs to at most $\operatorname{poly}(\Delta)^i$-many distinct $i$-step Vizing chains.
Indeed, detecting an uncolored edge $e\in U_d$ such that $f\in V(x,e)$ for some $i$-step Vizing chain basically boils down to following iteratively at most $i$-many alternating paths starting at $f$.
Importantly, this bound does not depend on $\ell$.
On the other hand, we expect, at least when $i$ is a small constant, that the number of $i$-step Vizing chains that we can grow from a given $e\in U_d$ should be of the order $(\ell/i)^i$.
The heuristic here is that we run the ``branching chain'' process $i$-times and in each step $1\le j\le i$ we can truncate the currently last alternating path, that must have length at least $\ell/i$, at almost every of the first $\ell/i$-many edges to create $P_j$.

To derive an estimate on $|U_d|/m$, we formalize these observations in a double counting argument assuming for simplicity that $i=1$.
Define an auxiliary bipartite graph $H_d=U_d\sqcup \dom(d)$, where $(e,f)$ forms an edge if there is $x\in e$ such that $f\in W(x,e)$.
By our assumption we have that $\deg_{H_d}(e)\ge  \ell-1$ for every $e\in U_d$, and by the discussion in the previous paragraph, we have that $\deg_{H_d}(f)\le \operatorname{poly}(\Delta)$ for every $f\in \dom(d)$.
A double counting argument in $H_d$ then yields that
\begin{equation}\label{eq:DoubelCount}
    (\ell-1)|U_d|\le \operatorname{poly}(\Delta) |\dom(d)|\le \operatorname{poly}(\Delta) m.
\end{equation}
In particular, the fraction of uncolored edges is $O(1/\ell)$ if we view $\Delta$ as a constant.

We remark that an analogous computation for the perfect matching problem, assuming that the size of every augmenting path is large, does not give any estimate on the size of unmatched edges as the degrees in both parts of the analogously defined auxiliary bipartite graph might be comparable.

\subsection{Applications}\label{subsec:ApplicationsVizing}

We start by recalling the remarkable result of Christiansen Theorem~\ref{thm:ClassicChristiansen}.

\begin{theorem}[Christiansen \cite{Christiansen}]\label{thm:CHristiansenTwo}
    Let $G$ be a graph, $d:E\rightharpoonup [\Delta(G)+1]$ be a partial edge coloring and $e\in U_d$ be an uncolored edge.
    Then there exists an augmenting subgraph $H$ for $e$ such that $|E(H)|\in O(\Delta^7 \log(n))$.    
\end{theorem}

The strategy to derive Theorem~\ref{thm:CHristiansenTwo} is to fix $i,\ell\in \mathbb{N}$ and count the number of vertices reachable by $i$-step Vizing chains with the property that each of its truncated alternating paths have size at most $\ell$.
This strategy originated in \cite{BernshteynVizing}, where Bernshteyn described a randomized algorithm for growing multi-step Vizing chains and showed that \emph{there are many} uncolored edges that admit an augmenting subgraph of size $O(\log^2(n))$.
We remark that for his analysis, he needed both parameters to satisfy $i,\ell\in \Theta(\log(n))$.

The main novelty in the approach of Christiansen is a highly non-trivial argument that even with $\ell\in \operatorname{poly}(\Delta)$ the number of reachable points from every uncolored edge grows exponentially in $i$ unless one of these chains is $d$-augmenting.
This implies, in particular, that after $\Theta(\log(n))$-steps we either exhaust the whole graph or find a $d$-augmenting chain.
\vspace{+0.2cm}

In order to derive a fast distributed algorithm for edge colorings based on augmenting subgraphs of size $\Omega(\log(n))$, one needs to show not only that such subgraphs exists, but also that they can be assigned to a large portion of the uncolored edges in an edge disjoint way.
Results along this line are essential in all the constructions from \cite{BernshteynVizing,Christiansen,BernshteynVizing2}.

We formulate the currently best version of such result that is from \cite{BernshteynVizing2} and that combines the idea of multi-step Vizing chains with the \emph{entropy compression method}.
We remark that the existence of the randomized distributed algorithm from Theorem~\ref{thm:DistributedVizing} is a direct consequence of this result, see \cite{BernshteynVizing2} for more details.

\begin{theorem}[Many disjoint augmenting subgraphs, Theorem~2.4 in \cite{BernshteynVizing2}]
    Let $d:E\rightharpoonup [\Delta+1]$ be a partial edge coloring.
    There exists a randomized distributed algorithm of LOCAL complexity $\operatorname{poly}(\Delta)\log(n)$ that outputs a set $W\subseteq U_d$ with expected size $\mathbb{E}(|W|)\ge |U_d|/\operatorname{poly}(\Delta)$ and an assignment of augmenting subgraphs $H_e$ for $e\in W$ such that 
    \begin{itemize}
        \item the graphs $H_e$, for $e\in W$, are vertex disjoint,
        \item $|E(H_e)|\le \operatorname{poly}(\Delta)\log(n)$.
    \end{itemize}
\end{theorem}

Next, we turn our attention to applications in measurable combinatorics.
Combining the ideas and concepts introduced in the previous parts of this section allows us to sketch a formal proof of the measurable Vizing's theorem for graphings.

\begin{theorem}[Vizing's theorem for graphings, \cite{grebik2020measurable}]\label{thm:VizingGraphing}
    Let $\fG=(V,E,\fB,\mu)$ be a graphing of degree bounded by $\Delta\in \mathbb{N}$.
    Then $\chi'(\fG)\le \Delta+1$.
\end{theorem}
\begin{proof}[Proof sketch]
    As a first step we pass to a Borel probability measure $\nu$ on $E$ that satisfies the mass transport principle on the Borel line graph $L(\fG)$ and has the property that $\mu(\{v:\exists e\in A \ v\in e\})=0$ whenever $\nu(A)=0$ for every measurable set $A\subseteq E$, see \cite[Section~3]{grebik2020measurable}.
    
    Following the same line of reasoning as in Theorem~\ref{thm:ElekLippner}, see \cite[Section~4]{grebik2020measurable}, we construct a sequence of partial measurable edge colorings $(d_n)_{n\in \mathbb{N}}$ with the property that
    \begin{itemize}
        \item [(a)] every $d_n$-augmenting $2$-step Vizing chain has length at least $\ell_n=2^{n}$,
        \item [(b)] $U_{d_n}\supseteq U_{d_{n+1}}$,
        \item [(c)] $\nu(\{e\in E:d_{n+1}(e)\not=d_n(e)\})\le 8\ell_n\nu(U_{d_n})$
    \end{itemize}
    for every $n\in \mathbb{N}$.

    Define the auxiliary bipartite graphing $\mathcal{H}_{d_n}=U_{d_n}\sqcup \dom(d_n)$ by setting $(e,f)$ to be an edge if there is a $2$-step Vizing chain $V(x,e)$ of length at most $\ell_n$ such that $f\in V(x,e)$.
    The computation in \cite[Section~3]{grebik2020measurable} confirms the intuition that $\deg_{\mathcal{H}_{d_n}}(e)\in \Omega(\ell_n^2)$ for every $e\in U_{d_n}$ while the bound $\deg_{\mathcal{H}_{d_n}}(f)\in \operatorname{poly}(\Delta)$ holds for every $f\in \dom(d_n)$ by the argument from Section~\ref{subsec:Multi}.
    As $\nu$ satisfies the mass transport principle, we conclude that 
    $$\nu(U_{d_n})\in O(1/\ell_n^2).$$
    Combined with the property (c) above, we see that the condition analogous to \eqref{eq:BC} is satisfied as
    $$ \sum_{n=0}^\infty 8 \ \ell_n\mu(U_{d_n})=\sum_{n=0}^\infty 2^{n+3}/2^{2n}<\infty.$$
    In particular, $d=\lim_{n\to \infty} d_n$ is a well-defined measurable edge coloring of $\fG$ by the Borel--Cantelli lemma.
\end{proof}

We finish this section with a discussion about the additional ingredients that are needed in the proof of the full measurable Vizing's theorem for general measured graphs, Theorem~\ref{thm:MeasureVizing}.
\vspace{+0.2cm}

It is a standard fact in measurable combinatorics, see for example \cite[Chapter~9]{topics}, that if $\fG=(V,E,\fB,\mu)$ is a measured graph, then we may assume that $\mu$ satisfies a \emph{weighted version of the mass transport principle}.
That is, every pair of points $(v,w)\in V\times V$ is assigned a value $\rho_\mu(v,w)$ that measures the ratio of the infinitesimal mass at $w$ compared to $v$ and satisfies the weighted version of the equality \eqref{eq:MassTransport}.
The function $\rho_\mu$ is known as the \emph{Radon--Nikodym cocycle} of $\fG$ (and $\mu$).
Mimicking the argument from Theorem~\ref{thm:VizingGraphing} and replacing the length of the $2$-step Vizing chains by their relative weight given by $\rho_\mu$, we run into a trouble when estimating the size of $\nu(U_{d_n})$ for every $n\in \mathbb{N}$.
This is because short chains might have big weight which works against us in the double counting argument on $\mathcal{H}_{d_n}$.

The way how to remedy this issue in \cite{GrebikVizing} is to ``smooth out'' the measure $\mu$ so that the value of the Radon--Nikodym cocycle on every pair of adjacent vertices $\{v,w\}\in E$ satisfies
\begin{equation}\label{eq:Vizing}
    1/(2\Delta)\le\rho(v,w)\le 2\Delta.
\end{equation}
This is roughly done as follows.
First, we use Theorem~\ref{thm:BorelGreedy} to assign in a Borel way to $\mu$-almost every $x\in V$ a finite measure $\kappa_x$ supported on the (countable) connectivity component of $x$ with the property that if $\dist_\fG(x,y)=k$, then $\kappa_x(y)=\rho_\mu(x,y)/(4\Delta)^k$.
Then we define a new measure $\nu$ to be the normalization of the convolution of the collection $\{\kappa_x\}_{x\in V}$ and $\mu$.
That is, $\nu(A)$ is equal, up to re-normalization, to
$$\mu\star\{\kappa_x\}_{x\in V}(A)=\int_{x\in V} \sum_{y\in A}\kappa_x(y) \ d\mu(x) $$
for every $A\in \fB$.
A direct computation shows that $\nu$ and $\mu$ are equivalent and $\rho_\nu$ satisfies \eqref{eq:Vizing}.

The estimate in \eqref{eq:Vizing} allows to compare the size and weight of any, for technical reasons, $3$-step Vizing chain in a way that works in our favor for the double counting argument.
We refer the reader to \cite{GrebikVizing} for more details about the arguments.

\section{Brooks' theorem for subexponential growth graphs}\label{sec:Subexp}

In this section we describe a deterministic distributed algorithm of LOCAL complexity $O(\log^*(n))$ that produces Brooks coloring on a class of finite graphs of a fixed subexponential growth that satisfy the assumptions from Brooks' theorem.
This result was likely known to experts in the field but we were unable to find a reference, or a written proof.

By Bernshteyn's correspondence \cite{Bernshteyn2021LLL}, Theorem~\ref{thm:Bernshteyn}, this implies that Brooks' theorem holds in the Borel context for Borel graphs of subexponential growth.
It was communicated to us by Bernshteyn that this result also follows from the proofs in \cite{conley2016brooks} as the graph induced on the set of the vertices where the described construction fails has to be of exponential growth.
This is in stark contrast to Theorem \ref{t:marksmain}; see also Remark~\ref{rem:HigherRandom}, which discusses a derandomization perspective for subexponential growth graphs.
\vspace{+0.2cm}

Recall that $f:\mathbb{N}\to\mathbb{N}$ is \emph{subexponential} if $f \in o((1+\varepsilon)^n)$ for every $\epsilon>0$.

\begin{definition}[Graphs of subexponential growth]
    A function $f:\mathbb{N}\to\mathbb{N}$ \emph{bounds the growth} of a graph $G$ if for every $v \in V(G)$ we have $|B_G(x,r)|<f(r)$.
    We denote by $\mathbb{F}_f$ the class of graphs of growth bounded by $f$.

    We say that class of graphs ${\bf F}$, or a graph $G$, is of a \emph{subexponential growth}, if there is subexponential $f:\mathbb{N}\to\mathbb{N}$ such that ${\bf F}\subseteq \mathbb{F}_f$, or $G\in \mathbb{F}_f$.
\end{definition}

We prove the following.

\begin{theorem}\label{thm:FastSubexpVizing}
    Let $\Delta\ge 3$ and ${\bf F}\subseteq \mathbb{F}_f$ be a class of graphs of a subexponential growth and degree bounded by $\Delta$ that do not contain the complete graph on $\Delta+1$ vertices.
    Then there is a deterministic distributed algorithm for Brooks coloring on ${\bf F}$ of LOCAL complexity $O(\log^*(n))$.
\end{theorem}

As an immediate corollary, we get a Borel version of Brooks theorem.
\begin{remark}
Note that the below theorem also follows from \cite[Theorem 2.15]{Bernshteyn2021LLL} combined with the distributed upper bound for Brooks' coloring from \cite{ghaffari_grunau_rozhon2020improved_network_decomposition}: indeed, by Bernshteyn's theorem, $O(\log n)$ randomized complexity algorithms directly yield Borel solutions on subexponential growth graphs. 
\end{remark}
\begin{theorem}[Borel Brooks coloring for graphs of subexponential growth, \cite{conley2016brooks,Bernshteyn2021LLL}]
	\label{t:subexp}
	Let $\mathcal{G}=(V,E,\fB)$ be a Borel graph of subexponential growth that does not contain a complete graph on $\Delta(\fG)+1$ vertices and satisfies $\Delta(\fG)\ge 3$.
    Then $\chi_\fB(\mathcal{G}) \leq \Delta(\fG)$. 
\end{theorem}
\begin{proof}
    Let $f$ be the subexponential function that bounds the growth of $\fG$.
    Define ${\bf F}$ to be the class of finite subgraphs of $\fG$ and observe that we may apply Theorem~\ref{thm:FastSubexpVizing} for ${\bf F}$.
    The rest follows from Theorem~\ref{thm:Bernshteyn}.
\end{proof}

Our approach to prove Theorem~\ref{thm:FastSubexpVizing} is identical to the one in \cite{BernshteynSubexp}, where Bernshteyn and Dhawan proved an analogous theorem for Vizing coloring.
Nevertheless, we provide the details for completeness.
The only difference is that in the place where they use the result of Christiansen, Theorem~\ref{thm:ClassicChristiansen}, to bound the size of an augmenting subgraph for edge colorings, we use the result of Panconesi and Srinivasan \cite{PanconesiSrinivasan} that we recall next.

    \begin{theorem}[Small augmenting subgraphs for vertex colorings, \cite{PanconesiSrinivasan}]
    	\label{t:pancsri}
     Let $G$ be a finite graph that does not contain a complete graph on $\Delta+1$ vertices, where $\Delta\ge \Delta(G)$, $v\in V(G)$ and $c$ be a partial vertex coloring with $\Delta$ colors such that $U_c=\{v\}$.
     Then there is an augmenting subgraph for $v$ that is a path of length $O(\log n)$. 
    \end{theorem}

\begin{proof}[Proof of Theorem~\ref{thm:FastSubexpVizing}]
    Let $f$ be the subexponential function that bounds the growth of the graphs from ${\bf F}$.
    It is clearly enough to prove the claim for the inclusion maximal class ${\bf F}$.

    The main observation, see \cite[Lemma~3.3]{BernshteynSubexp}, that crucially uses the subexponential bound on the growth is that we can always find an augmenting subgraph of constant size.
    
\begin{proposition}
	\label{l:augmenting} There exists an $R \in \N$ so that for every $G \in {\bf F}$, every partial vertex coloring $c$ of $G$ with $\Delta$ colors and every vertex $v \in U_c$ there is an augmenting subgraph $H$ for $v$ contained in $B_G(v,R)$. 
\end{proposition}
\begin{proof}
	 Let $C$ be a constant so that Theorem \ref{t:pancsri} holds with a path of size $C\log n$ for each graph of size $n$. Find some $\varepsilon>0$ with $C \log(1+\varepsilon)<1$ and an $R$ large enough so that $f(R)<(1+\varepsilon)^R$.

	Now let $c$ be an arbitrary partial vertex coloring of an element $G \in {\bf F}$ with $\Delta$ colors and $v \in U_c$.
    Consider the graph $G'$ which is the connected component of $v$ in the graph $G$ restricted to $B_G(v,R) \cap (\dom(c) \cup \{v\})$.
    Let $H$ be an augmenting subgraph in $G'$ for $c$ returned by Theorem \ref{t:pancsri}.
    Then we have
     \[|V(H)|<C\log |V(G')|\leq C\log |B_G(v,R)| \leq C\log f(R)<C\log \left((1+\varepsilon)^R\right) < R.\]
	Recall that an augmenting subgraph is connected by the definition, so, if the size of $H$ is $\leq R-1$, it has to be contained in the ball $B_G(v,R-1)$. But, then $H$ must be an augmenting subgraph for $c$ in the entire graph $G$. 
\end{proof}

	Now, to prove Theorem~\ref{thm:FastSubexpVizing}, take any graph $G \in {\bf F}$.
    Recall that $G^{(k)}$ stands for the $k$th power of $G$, that is, the graph where two vertices are connected if their graph distance in $G$ is at most $k$.
    Clearly, the graph $G^{(2R+2)}$ has degrees bounded by $\Delta^{2R+2}$, so, in $O(\log^*n)$ rounds we can compute a proper vertex coloring of $G^{(2R+2)}$ with $k=\Delta^{2R+2}+1$ colors by Theorem \ref{thm:DistributedGreedy}.
    Enumerate the color classes as $(C_i)_{i\leq k}$. Then, in $k$-many steps inductively, define partial vertex colorings $(c_i)_{i \leq k}$ as follows: if $c_{j-1}$ is defined on $\bigcup_{i<j} C_i$, for every $v \in C_j$ by Proposition \ref{l:augmenting} we can find an augmenting subgraph in $B_G(v,R)$ for $c_{j-1}$ and $v$. Recolor the vertices using this graph to obtain $c_j$.
    Observe that, as the distance between any two elements of $C_j$ is $>2R+2$, $c_j$ will still be a partial vertex coloring. This algorithm terminates after $k$-many steps, yielding a vertex coloring with $\Delta$ colors.
    Finally, observe that as $k$ and $R$ are constants, the overall complexity of the algorithm remains $O(\log^*n)$, as desired. 
\end{proof}

\section{Open problems}\label{sec:Problems}

We finish the survey by listing some exciting open problems that are directly connected to the results discussed in this survey.
\vspace{+0.2cm}

\noindent
{\bf Edge colorings.}
We start with a couple of questions about Vizing colorings.
The main problem here concerns the existence of a fast distributed algorithm for Vizing's theorem, see also \cite[Question~1.9]{BernshteynVizing2} and \cite[Question~1.2]{BernshteynVizing}.

\begin{problem}\label{pr:SublogVizing}
    Is the randomized LOCAL complexity of Vizing coloring $o(\log(n))$?
\end{problem}

In \cite{Christiansen}, Christiansen proved a \emph{local version of Vizing's theorem}, that is, every finite graph $G$ admits an edge coloring $c$ with $\Delta(G)+1$ colors with the property that
$$c(e)\le \max\{\deg_G(v),\deg_G(w)\}+1$$
for every edge $e=\{v,w\}\in E(G)$.
The approach from \cite{Christiansen} does not seem to automatically generalize to the measurable setting.

\begin{problem}
    Is there a measurable version of local Vizing's theorem?
\end{problem}

In Remark~\ref{rem:Baire}, we briefly mentioned the topic of \emph{Baire measurable colorings}.
Qian and Weilacher \cite{QianWeilacher} showed that $\Delta+2$ colors are enough for a Baire measurable edge coloring of a graph of degree bounded by $\Delta$.

\begin{problem}
    Is there a Baire measurable version of Vizing's theorem?
\end{problem}

In Section~\ref{subsec:LyonsNazarov}, we have seen that there is a factor of iid perfect matching on the $\Delta$-regular tree $T_\Delta$.
The following question is still open.
Note that Kun's result \cite{KunPerfectMatching} implies that K\H{o}nig's line coloring theorem fails for measurably bipartite acyclic $\Delta$-regular measured graphs in general, for every $\Delta\ge 3$.

\begin{problem}
    Is there a factor of iid edge coloring with $\Delta$ colors on the $\Delta$-regular tree $T_\Delta$?
    In other words, does measurable K\H{o}nig's line coloring theorem hold for the iid graph on $T_\Delta$?
\end{problem}

An intriguing question related to the results in Section~\ref{subsec:LyonsNazarov} asks about an optimal algorithm for perfect matching on the $\Delta$-regular tree $T_\Delta$.
This can be formalized using he notion of \emph{finitary factors of iid}, which provides a strictly finer framework to measure complexity of LCLs than the one given by the LOCAL model, see \cite{HolroydSchrammWilson2017FinitaryColoring,grebik_rozhon2021toasts_and_tails,grebik_rozhon2021LCL_on_paths} for the exact definition and related results.
On a high-level, finitary factors of iid are randomized deterministic algorithms that do not know the size of the graph, hence, the locality into which every vertex needs to look is a random variable that is almost surely finite.
The complexity is measured by the tail decay of this random variable.

A recent result of Berlow, Bernshteyn, Lyons and Weilacher gives the first example of an LCL problem (on $\mathbb{Z}^2$) that admits a factor of iid solution but does not admit a finitary factor of iid solution.
It would be nice to show that the perfect matching problem on $T_\Delta$ has the same property.
We remark that by \cite{brandt_chang_grebik_grunau_rozhon_vidnyaszky2021LCLs_on_trees_descriptive}, there is no finitary factor of iid perfect matching on $T_\Delta$ of subexponential tail decay.

\begin{problem}
    Is there a finitary factor of iid perfect matching on the $\Delta$-regular tree $T_\Delta$?
\end{problem}

\noindent
{\bf  LCL problems.}
The class of measured graphs that satisfy the mass transport principle is much more studied than the class of general measured graphs.
It is, however, not known if the mass transport principle gives any advantage for solving LCL problems.
To the following problem the discussion at the end of Section~\ref{subsec:ApplicationsVizing} about Radon--Nikodym cocycles might be relevant.

\begin{problem}[Question~6.10 in \cite{berlow2025separating}]\label{pr:PMPvsQUASIPMP}
    Is there an LCL problem that could be solved on measured graphs that satisfy the mass transport principles but not on general measured graphs?
    The same question is open on any subclass of measured graphs that look locally like a fixed graph, e.g., grid, regular tree etc.
\end{problem}

A possible candidate to answer Problem~\ref{pr:PMPvsQUASIPMP} is the \emph{unfriendly coloring problem}.
In its most basic form it asks for a coloring of vertices of a graph $G$ with $2$ colors so that every vertex receives a color that is different from the color of at least half of its neighbors.
Conley and Tamuz \cite{ConleyTamuz} showed that a measured graph $\fG$ admits a measurable unfriendly coloring provided that the Radon--Nikodym cocycle satisfies $1-1/\Delta\le \rho(x,y)\le 1+1/\Delta$ for every edge $\{x,y\}\in \fG$.
Note that this is much stronger than \eqref{eq:Vizing}.

\begin{problem}
    Does every measured graph of bounded degree admits a measurable unfriendly coloring?
    In fact, it is also open whether every Borel graph of bounded degree admits a Borel unfriendly coloring. 
\end{problem}

In Remark~\ref{rem:HigherRandom}, we discussed versions of derandomization in the context of measurable combinatorics.
While the recent counterexamples discussed there answered in negative most of the basic questions, the following two problems are still open.

\begin{problem}
    Let ${\bf F}$ be a class of graphs of bounded degree and linear growth.
    Is it true that an LCL problem on ${\bf F}$ can be solved measurably on measured graphs that are locally in ${\bf F}$ if and only if it can be solved in a Borel way?
\end{problem}

    The following problem is a special case of the celebrated conjecture of Chang and Pettie \cite{chang2019time}.

\begin{problem}
    Let ${\bf F}$ be a class of graphs of bounded degree and of subexponential growth.
    Is the LOCAL complexity of the distributed Lov\'{a}sz Local Lemma on ${\bf F}$ equal to $\Theta(\log^*(n))$?
    That is, is there a speed-up result from randomized LOCAL complexity $O(\log(n))$ to deterministic LOCAL complexity $O(\log^*(n))$ for ${\bf F}$?
\end{problem}

\noindent
{\bf Complexity.}
In Section~\ref{subsec:Complexity}, we mentioned in Theorem~\ref{thm:BrooksComplexity} that Brooks' theorem strongly fails in the Borel context.
In particular, the set of $\Delta$-regular acylic Borel graphs that admit a Borel vertex coloring with $\Delta$ colors is as complicated as it gets.
Related is the question about the complexity of edge colorings.
Thornton \cite{thornton2022algebraic} showed that distinguishing between Borel chromatic index $\Delta$ and $\Delta+1$ for Borel graphs of degree bounded by $\Delta$ is $\mathbf{\Sigma}^1_2$-complete.
As the graphs in his construction are not acyclic and Marks' result shows that Borel chromatic index of $\Delta$-regular acyclic graphs can be anything between $\Delta$ and $2\Delta-1$, it is natural to ask the following.

\begin{problem}
    Is it $\mathbf{\Sigma}^{1}_2$-complete to distinguish between any $k$ and $k+1$ Borel chromatic index on $\Delta$-regular acyclic graphs, where $\Delta\le k\le 2\Delta-2$?
\end{problem}

As mentioned above, the projective complexity in the Borel context reflects to some extent the finite experience: deciding Borel $2$-colorability of a Borel graph is strictly simpler than to decide its $3$-colorability (see \cite{benen,thornton2022algebraic}). More precisely, the codes of the collections of Borel graphs that admit a Borel $2$-coloring form a $\mathbf{\Pi}^1_1$ set, while for $k$-coloring, where $k\ge 3$, the corresponding set is $\mathbf{\Sigma}^1_2$-complete, see also Section \ref{subsec:Complexity}.

\begin{problem}
    Let $\Delta\ge 3$.
    Do $\Delta$-regular acyclic Borel hyperfinite Borel graphs that have Borel chromatic number at most $\Delta$ form a $\Sigma^1_2$-complete?    
\end{problem}

To our knowledge, surprisingly, there are no results about complexity in the measurable context.

\begin{problem}
    Let $k\ge 3$ and fix a Borel measure $\mu$ on some standard Borel space $X$. What is the complexity of the codes\footnote{In some fixed encoding, where all the codes form a Borel set. This is possible if one discards a $\mu$ measure zero set.} of locally countable Borel graphs $\mc{G}$ with $\chi_\mu(\mc{G}) \leq k$?
    
	Is there a precise way in which deciding the measurable $2$-colorability of an acyclic bounded degree graphing is easier than to decide its $3$-colorability?
 \end{problem}

    The straightforward upper bound for both $2$-and $3$-colorings in this situation is $\mathbf{\Sigma}^1_1$, so it would be extremely interesting if the latter was $\mathbf{\Sigma}^1_1$-hard and the former was $\mathbf{\Delta}^1_1$.
    \vspace{+0.2cm}

    A different complexity related question arises from the fact that all the examples considered in complexity considerations in \cite{todorvcevic2021complexity,brandt_chang_grebik_grunau_rozhon_vidnyaszkyhomomorphisms} follow the same pattern: constructing a family of Borel graphs parametrized by the reals, so that each of them admits a Borel $3$-coloring, but altogether it is hard to decide whether such a coloring exists.
    This suggests that there could be a positive result, once one excludes this type of examples. The following question formalizes the most optimistic scenario.

\begin{problem}
	Assume that $\mc{G}$ is a $\Delta$-regular acyclic Borel graph so that there is no smooth Borel superequivalence relation\footnote{See \cite{kechris2024theory} for the definition.} $E$ of $E_\mc{G}$ so that $\mc{G}$ restricted to each $E$ class admits a Borel $\Delta$-coloring, for $\Delta \geq 3$. Does Marks' example, i.e., the Schreier graph of the free part of the left-shift action of $\Gamma_\Delta$ on $2^{\Gamma_\Delta}$, Borel homomorph to $\mc{G}$? 
\end{problem}

    Furthermore, the produced examples use non-compactness of the underlying space in an essential way. Thus, another positive result could be possible when one restricts to nicer Borel graphs, such as continuous graphs on compact spaces.

\begin{problem}
	Assume that $\mc{G}$ is a compact subshift of the free part of $2^{\Gamma_\Delta}$ that does not admit a Borel $\Delta$-coloring, for $\Delta \geq 3$. Does Marks' example Borel homomorph to $\mc{G}$?
    What is the projective complexity of such subshifts?
   
\end{problem}

\noindent
{\bf Generalizations.}
One of the most exciting problems is to transfer the full strength of Martin's theorem to the finite context. Indeed, excluding deterministic algorithms of complexity  $o(\log n)$ for $\Delta$ vertex colorings on trees of degree $\leq \Delta$ only requires the determinacy of clopen games, the proof of which is straightforward, see \cite[Section 20.B]{kechrisclassical}. In contrast, full Borel Determinacy is one of the most involved and exciting arguments of classical descriptive set theory. 

A key feature that distinguishes the Borel case from the LOCAL model is that ``one can see in an arbitrary distance". It would be nice to have a version of the LOCAL model, which incorporates this, but still non-trivial. Let us pose the following -admittedly vague- problem.

\begin{problem}
	\label{p:variant}
	Is there a stronger variant of the LOCAL model (e.g., for which all the LCL problems that are known to admit a Borel solution are efficiently solvable, see \cite{brandt_chang_grebik_grunau_rozhon_vidnyaszky2021LCLs_on_trees_descriptive, grebik_rozhon2021LCL_on_paths}) so that $3$-coloring of acyclic graphs of degree at most $3$ is still not possible, because of Marks' method?
 
\end{problem}

More generally, instead of colorings, it is reasonable to consider homomorphism problems. The following problems are variants of \cite[Problem 8.12]{kechris_marks2016descriptive_comb_survey}.

\begin{problem}
	Characterize the following families of finite graphs:
	\begin{itemize}
		\item Graphs $H$ so that there is a deterministic distributed algorithm of complexity $O(\log^* n)$ on the class of trees of degree $\leq \Delta$ that produces a homomorphism to $H$.
		\item Graphs $H$ so that every $\Delta$-regular acyclic Borel graph admits a Borel homomorphism to $H$.
		\item Graphs $H$ so that there is a factor of iid homomorphism to $H$ from $T_\Delta$.
		\item  Graphs $H$ so that every $\Delta$-regular acyclic Borel graph equipped with a measure $\mu$ admits a Borel homomorphism to $H$ modulo a $\mu$ null set. 
	\end{itemize}
 Is membership in any of these families algorithmically undecidable?

\end{problem}

Let us mention that undecidability phenomenons are known to be present in the case of grid graphs, see \cite[Section 4.3]{GJKS}. However, it is also conceivable that one can have a full understanding of the solvability of such homomorphism problems in the special case of trees, as it is the case in the Baire measurable context, see \cite{brandt_chang_grebik_grunau_rozhon_vidnyaszky2021LCLs_on_trees_descriptive}. 

It seems that even to prove a negative result in these cases, one would need to significantly extend the collection of examples $H$ to which there are definable homomorphisms of the above sort, see \cite{csoka2024fiid}, as, so far all the examples about which we are aware of are very close to $K_{\Delta+1}$.

\bibliographystyle{alpha}
	\bibliography{ref}

\newcommand{\etalchar}[1]{$^{#1}$}
\begin{thebibliography}{EWHK98}

\bibitem[AL07]{aldouslyons2007}
David Aldous and Russell Lyons.
\newblock Processes on unimodular random networks.
\newblock {\em Electron. J. Probab.}, 12:1454--1508, 2007.

\bibitem[BBLW25]{berlow2025separating}
Katalin Berlow, Anton Bernshteyn, Clark Lyons, and Felix Weilacher.
\newblock Separating complexity classes of lcl problems on grids.
\newblock {\em arXiv:2501.17445}, 2025.

\bibitem[BBOS20]{balliu2020almost_global_problems}
Alkida Balliu, Sebastian Brandt, Dennis Olivetti, and Jukka Suomela.
\newblock Almost global problems in the local model.
\newblock {\em Distributed Computing}, pages 1--23, 2020.

\bibitem[BCC{\etalchar{+}}24]{Faster}
Sayan Bhattacharya, Din Carmon, Mart\'{i}n Costa, Shay Solomon, and Tianyi
  Zhang.
\newblock Faster $({\Delta}+1)$-edge coloring: Breaking the $m\sqrt{n}$ time
  barrier.
\newblock {\em arXiv:1907.03201}, 2024.

\bibitem[BCG{\etalchar{+}}21]{brandt_chang_grebik_grunau_rozhon_vidnyaszky2021LCLs_on_trees_descriptive}
Sebastian Brandt, Yi-Jun Chang, Jan Greb\'{i}k, Christoph Grunau, Vaclav
  Rozho{\v n}, and Zolt\'{a}n Vidny\'{a}nszky.
\newblock Local problems on trees from the perspectives of distributed
  algorithms, finitary factors, and descriptive combinatorics.
\newblock {\em arXiv:2106.02066}, 2021.
\newblock A short version of the paper was presented at the \emph{13th
  Innovations in Theoretical Computer Science conference (ITCS 2022)}.

\bibitem[BCG{\etalchar{+}}22]{DeterministicTrees}
Sebastian Brandt, Yi-Jun Chang, Jan Greb\'{i}k, Christoph Grunau, V\'{a}clav
  Rozho\v{n}, and Zolt\'{a}n Vidny\'{a}nszky.
\newblock Deterministic distributed algorithms and descriptive combinatorics on
  ${\Delta}$-regular trees.
\newblock {\em arXiv:2204.09329}, 2022.

\bibitem[BCG{\etalchar{+}}24]{brandt_chang_grebik_grunau_rozhon_vidnyaszkyhomomorphisms}
Sebastian Brandt, Yi-Jun Chang, Jan Grebík, Christoph Grunau, Václav Rozhoň,
  and Zoltán Vidnyánszky.
\newblock On homomorphism graphs.
\newblock {\em Forum of Mathematics, Pi}, 12:e10, 2024.

\bibitem[BD23]{BernshteynVizing2}
Anton Bernshteyn and Abhishek Dhawan.
\newblock Fast algorithms for {V}izing's theorem on bounded degree graphs.
\newblock {\em arXiv:2303.05408}, 2023.

\bibitem[BD24]{BernshteynVizing3}
Anton Bernshteyn and Abhishek Dhawan.
\newblock A linear-time algorithm for $(1+\epsilon)${$\Delta$}-edge-coloring.
\newblock {\em arXiv:2407.04887}, 2024.

\bibitem[BD25]{BernshteynSubexp}
Anton Bernshteyn and Abhishek Dhawan.
\newblock {B}orel {V}izing's theorem for graphs of subexponential growth.
\newblock {\em Proc. Amer. Math. Soc.}, 153:7--14, 2025.

\bibitem[BE08]{BarenboimElkinPaper}
Leonid Barenboim and Michael Elkin.
\newblock Sublogarithmic distributed mis algorithm for sparse graphs using
  nash-williams decomposition.
\newblock In {\em Proceedings of the Twenty-Seventh ACM Symposium on Principles
  of Distributed Computing}, PODC '08, page 25–34, New York, NY, USA, 2008.
  Association for Computing Machinery.

\bibitem[BE13]{barenboimelkin_book}
Leonid Barenboim and Michael Elkin.
\newblock {\em Distributed Graph Coloring: Fundamentals and Recent
  Developments}.
\newblock Morgan \& Claypool Publishers, 2013.

\bibitem[BE19]{BamasEsperet}
\'{E}tienne Bamas and Louis Esperet.
\newblock {Distributed Coloring of Graphs with an Optimal Number of Colors}.
\newblock In {\em 36th International Symposium on Theoretical Aspects of
  Computer Science (STACS 2019)}, volume 126 of {\em Leibniz International
  Proceedings in Informatics (LIPIcs)}, pages 10:1--10:15, Dagstuhl, Germany,
  2019. Schloss Dagstuhl -- Leibniz-Zentrum f{\"u}r Informatik.

\bibitem[Ber19]{BernshteynEarlyLLL}
Anton Bernshteyn.
\newblock Measurable versions of the {L}ovász {L}ocal {L}emma and measurable
  graph colorings.
\newblock {\em Advances in Mathematics}, 353:153--223, 2019.

\bibitem[Ber22a]{BernshteynSurvey}
Anton Bernshteyn.
\newblock Descriptive combinatorics and distributed algorithms.
\newblock {\em Notices of the American Mathematical Society}, 2022.

\bibitem[Ber22b]{BernshteynVizing}
Anton Bernshteyn.
\newblock A fast distributed algorithm for {$(\Delta+1)$}-edge-coloring.
\newblock {\em Journal of Combinatorial Theory, Series B}, 152:319--352, 2022.

\bibitem[Ber23a]{Bernshteyn2021LLL}
Anton Bernshteyn.
\newblock Distributed algorithms, the {L}ovász {L}ocal {L}emma, and
  descriptive combinatorics.
\newblock {\em Invent. math.}, 233:495–542, 2023.

\bibitem[Ber23b]{Bernshteyn2021local=cont}
Anton Bernshteyn.
\newblock Probabilistic constructions in continuous combinatorics and a bridge
  to distributed algorithms.
\newblock {\em Advances in Mathematics}, 415:108895, 2023.

\bibitem[BFH{\etalchar{+}}16]{brandt_etal2016LLL}
Sebastian Brandt, Orr Fischer, Juho Hirvonen, Barbara Keller, Tuomo
  Lempi{\"a}inen, Joel Rybicki, Jukka Suomela, and Jara Uitto.
\newblock A lower bound for the distributed {L}ov{\'a}sz local lemma.
\newblock In {\em Proc.\ 48th ACM Symp.\ on Theory of Computing (STOC)}, pages
  479--488, 2016.

\bibitem[BHK{\etalchar{+}}17]{brandt_grids}
Sebastian Brandt, Juho Hirvonen, Janne~H. Korhonen, Tuomo Lempi\"{a}inen,
  Patric~R.J. \"{O}sterg\r{a}rd, Christopher Purcell, Joel Rybicki, Jukka
  Suomela, and Przemys\l{}aw Uzna\'{n}ski.
\newblock {L}{C}{L} problems on grids.
\newblock In {\em Proceedings of the ACM Symposium on Principles of Distributed
  Computing {(PODC)}}, page 101–110, New York, NY, USA, 2017. Association for
  Computing Machinery.

\bibitem[BHK{\etalchar{+}}18]{balliu2018new_classes-loglog*-log*}
Alkida Balliu, Juho Hirvonen, Janne~H Korhonen, Tuomo Lempi{\"a}inen, Dennis
  Olivetti, and Jukka Suomela.
\newblock New classes of distributed time complexity.
\newblock In {\em Proceedings of the 50th Annual ACM SIGACT Symposium on Theory
  of Computing}, pages 1307--1318, 2018.

\bibitem[BHT24]{ArankaFeriLaci}
Ferenc Bencs, Aranka Hrušková, and László~Márton Tóth.
\newblock Factor-of-{IID} schreier decorations of lattices in euclidean spaces.
\newblock {\em Discrete Mathematics}, 347(9):114056, 2024.

\bibitem[Bol81]{bollobas}
B\'{e}la Bollob\'{a}s.
\newblock The independence ratio of regular graphs.
\newblock {\em Proc. Amer. Math. Soc.}, 83(2):433--436, 1981.

\bibitem[Bra19]{brandt19automatic}
Sebastian Brandt.
\newblock An automatic speedup theorem for distributed problems.
\newblock In {\em Proceedings of the 2019 {ACM} Symposium on Principles of
  Distributed Computing, {PODC} 2019, Toronto, ON, Canada, July 29 - August 2,
  2019}, pages 379--388, 2019.

\bibitem[Bra22]{brady2022notes}
Zarathustra Brady.
\newblock Notes on {CSP}s and polymorphisms.
\newblock {\em arXiv:2210.07383}, 2022.

\bibitem[Bro41]{Brooks}
R.~Leonard Brooks.
\newblock On colouring the nodes of a network.
\newblock {\em Mathematical Proceedings of the Cambridge Philosophical
  Society}, 37(2):194--197, 1941.

\bibitem[Bul17]{bulatov2017dichotomy}
Andrei~A. Bulatov.
\newblock A dichotomy theorem for nonuniform {CSP}s.
\newblock In {\em IEEE 58th Annual Symposium on Foundations of Computer Science
  (FOCS)}, pages 319--330, 2017.

\bibitem[BW23a]{BernshteynWeilacherLLL}
Anton Bernshteyn and Felix Weilacher.
\newblock Borel versions of the {L}ocal {L}emma and {LOCAL} algorithms for
  graphs of finite asymptotic separation index.
\newblock {\em arXiv:2308.14941}, 2023.

\bibitem[BW23b]{BowenWeilacher}
Matt Bowen and Felix Weilacher.
\newblock Definable {K}\"{o}nig theorems.
\newblock {\em Proc. Amer. Math. Soc.}, 151:4991--4996, 2023.

\bibitem[CGA{\etalchar{+}}22]{OlegLLL}
Endre Cs{\'o}ka, {\L}ukas Grabowski, M{\'a}th{\'e} Andr{\'as}, Oleg Pikhurko,
  and Konstantin Tyros.
\newblock Moser-tardos algorithm with small number of random bits.
\newblock {\em arXiv:2203.05888}, 2022.

\bibitem[CGM{\etalchar{+}}17]{csokagrabowski}
Endre Cs{\'o}ka, {\L}ukasz Grabowski, Andr{\'a}s M{\'a}th{\'e}, Oleg Pikhurko,
  and Konstantinos Tyros.
\newblock {B}orel version of the {L}ocal {L}emma.
\newblock {\em arXiv:1605.04877}, 2017.

\bibitem[Cha20]{chang2020n1k_speedups}
Yi-Jun Chang.
\newblock {The Complexity Landscape of Distributed Locally Checkable Problems
  on Trees}.
\newblock In Hagit Attiya, editor, {\em 34th International Symposium on
  Distributed Computing (DISC 2020)}, volume 179 of {\em Leibniz International
  Proceedings in Informatics (LIPIcs)}, pages 18:1--18:17, Dagstuhl, Germany,
  2020. Schloss Dagstuhl--Leibniz-Zentrum f{\"u}r Informatik.

\bibitem[CHL{\etalchar{+}}19]{ChangHLPU20}
Yi-Jun Chang, Qizheng He, Wenzheng Li, Seth Pettie, and Jara Uitto.
\newblock Distributed edge coloring and a special case of the constructive
  lov\'{a}sz local lemma.
\newblock {\em ACM Trans. Algorithms}, 16(1), November 2019.

\bibitem[Chr23]{Christiansen}
Aleksander Bj\o{}rn~Grodt Christiansen.
\newblock The power of multi-step vizing chains.
\newblock In {\em Proceedings of the 55th Annual ACM Symposium on Theory of
  Computing}, STOC 2023, page 1013–1026, New York, NY, USA, 2023. Association
  for Computing Machinery.

\bibitem[CJK{\etalchar{+}}18]{FolnerTilings}
Clinton~T. Conley, Steve~C. Jackson, David Kerr, Andrew~S. Marks, Brandon
  Seward, and Robin~D. Tucker-Drob.
\newblock Følner tilings for actions of amenable groups.
\newblock {\em Math. Ann.}, 371:663--683, 2018.

\bibitem[CJM{\etalchar{+}}20]{conleyhyp}
Clinton~T. Conley, Steve Jackson, Andrew~S. Marks, Brandon Seward, and Robin
  Tucker-Drob.
\newblock Hyperfiniteness and {B}orel combinatorics.
\newblock {\em J. Eur. Math. Soc. (JEMS)}, 22(3):877--892, 2020.

\bibitem[CJM{\etalchar{+}}23]{AsymptoticDim}
Clinton~T. Conley, Steve Jackson, Andrew~S. Marks, Brandon Seward, and Robin
  Tucker-Drob.
\newblock Borel asymptotic dimension and hyperfinite equivalence relations.
\newblock {\em Duke Math. J.}, 172(16):3175--3226, 2023.

\bibitem[CKP19]{chang_kopelowitz_pettie2019exp_separation}
Yi-Jun Chang, Tsvi Kopelowitz, and Seth Pettie.
\newblock An exponential separation between randomized and deterministic
  complexity in the local model.
\newblock {\em SIAM Journal on Computing}, 48(1):122--143, 2019.

\bibitem[CLP16]{csokalippnerpikhurko}
Endre Cs{\'o}ka, Gabor Lippner, and Oleg Pikhurko.
\newblock Invariant gaussian processes and independent sets on regular graphs
  of large girth.
\newblock {\em Forum of Math., Sigma}, 4, 2016.

\bibitem[CM17]{millerreducibility}
Clinton~T. Conley and Benjamin~D. Miller.
\newblock Measure reducibility of countable {B}orel equivalence relations.
\newblock {\em Ann. of Math. (2)}, 185(2):347--402, 2017.

\bibitem[CMSV21]{benen}
Rapha\"el Carroy, Benjamin~D. Miller, David Schrittesser, and Zolt\'an
  Vidny\'anszky.
\newblock Minimal definable graphs of definable chromatic number at least
  three.
\newblock {\em Forum Math. Sigma}, 9:e7, 2021.

\bibitem[CMTD16]{conley2016brooks}
Clinton~T Conley, Andrew~S Marks, and Robin~D Tucker-Drob.
\newblock Brooks’theorem for measurable colorings.
\newblock In {\em Forum of Mathematics, Sigma}, volume~4. Cambridge University
  Press, 2016.

\bibitem[CP19]{chang2019time}
Yi-Jun Chang and Seth Pettie.
\newblock A time hierarchy theorem for the local model.
\newblock {\em SIAM Journal on Computing}, 48(1):33--69, 2019.

\bibitem[CPS17]{chung2017LLL}
Kai-Min Chung, Seth Pettie, and Hsin-Hao Su.
\newblock Distributed algorithms for the lov{\'a}sz local lemma and graph
  coloring.
\newblock {\em Distributed Computing}, 30(4):261--280, 2017.

\bibitem[CT21]{ConleyTamuz}
Clinton~T. Conley and Omer Tamuz.
\newblock Unfriendly colorings of graphs with finite average degree.
\newblock {\em Proc. Lond. Math. Soc.}, 3, 2021.

\bibitem[CU22]{spencer_personal}
Nishant Chandgotia and Spencer Unger.
\newblock Borel factors and embeddings of systems in subshifts.
\newblock {\em arxiv:2203.09359}, 2022.

\bibitem[CV86]{cole86}
Richard Cole and Uzi Vishkin.
\newblock Deterministic coin tossing with applications to optimal parallel list
  ranking.
\newblock {\em Information and Control}, 70(1):32--53, 1986.

\bibitem[CV24]{csoka2024fiid}
Endre Cs\'{o}ka and Zoltan Vidnyanszky.
\newblock {FIID} homomorphisms and entropy inequalities.
\newblock {\em arXiv:2407.10006}, 2024.

\bibitem[DF92]{doughertyforeman}
Randall Dougherty and Matthew Foreman.
\newblock Banach-{T}arski paradox using pieces with the property of {B}aire.
\newblock {\em Proc. Nat. Acad. Sci. U.S.A.}, 89(22):10726--10728, 1992.

\bibitem[EL10]{eleklippner}
G\'{a}bor Elek and G\'{a}bor Lippner.
\newblock {B}orel oracles. {A}n analytical approach to constant-time
  algorithms.
\newblock {\em Proc. Amer. Math. Soc.}, 138(8):2939--2947, 2010.

\bibitem[Ele18]{elek2018qualitative}
G{\'a}bor Elek.
\newblock Qualitative graph limit theory. cantor dynamical systems and
  constant-time distributed algorithms.
\newblock {\em arXiv:1812.07511}, 2018.

\bibitem[Erd59]{ErdosProbMethod}
Paul Erd\H{o}s.
\newblock Graph theory and probability.
\newblock {\em Canadian Journal of Mathematics}, 11:34--38, 1959.

\bibitem[EWHK98]{Emden}
Thomas Emden-Weinert, Stefan Hougardy, and Bernd Kreuter.
\newblock Uniquely colourable graphs and the hardness of colouring graphs of
  large girth.
\newblock {\em Combinatorics, Probability and Computing}, 7(4):375--386, 1998.

\bibitem[FHM23]{FastDistributedBrooksTheorem}
Manuela Fischer, Magnús~M. Halldórsson, and Yannic Maus.
\newblock {\em Fast Distributed Brooks' Theorem}, pages 2567--2588.
\newblock Association for Computing Machinery, 2023.

\bibitem[FSV24]{frisch2024hyper}
Joshua Frisch, Forte Shinko, and Zoltan Vidnyanszky.
\newblock Hyper-hyperfiniteness and complexity.
\newblock {\em arXiv:2409.16445}, 2024.

\bibitem[Gab00]{gaboriau}
Damien Gaboriau.
\newblock Co\^{u}t des relations d'\'{e}quivalence et des groupes.
\newblock {\em Invent. Math.}, 139(1):41--98, 2000.

\bibitem[GGR21]{ghaffari_grunau_rozhon2020improved_network_decomposition}
Mohsen Ghaffari, Christoph Grunau, and V\'{a}clav Rozho\v{n}.
\newblock Improved deterministic network decomposition.
\newblock In {\em Proc. of the 32. ACM-SIAM Symp. on Discrete Algorithms
  (SODA)}, page 2904–2923, USA, 2021. Society for Industrial and Applied
  Mathematics.

\bibitem[GH24]{grebik2024complexity}
Jan Greb{\'\i}k and Cecelia Higgins.
\newblock Complexity of finite borel asymptotic dimension.
\newblock {\em arXiv:2411.08797}, 2024.

\bibitem[Gha19]{ghaffari2019distributed}
Mohsen Ghaffari.
\newblock Distributed maximal independent set using small messages.
\newblock In {\em Proceedings of the Thirtieth Annual ACM-SIAM Symposium on
  Discrete Algorithms}, pages 805--820. SIAM, 2019.

\bibitem[GHK18]{ghaffari2018derandomizing}
Mohsen Ghaffari, David~G Harris, and Fabian Kuhn.
\newblock On derandomizing local distributed algorithms.
\newblock In {\em 2018 IEEE 59th Annual Symposium on Foundations of Computer
  Science (FOCS)}, pages 662--673. IEEE, 2018.

\bibitem[GHKM18]{ghaffari2020Delta_coloring}
Mohsen Ghaffari, Juho Hirvonen, Fabian Kuhn, and Yannic Maus.
\newblock Improved distributed delta-coloring.
\newblock In {\em Proceedings of the 2018 ACM Symposium on Principles of
  Distributed Computing}, pages 427--436, 2018.

\bibitem[GJKS18]{GJKS}
Su~Gao, Steve Jackson, Edward Krohne, and Brandon Seward.
\newblock Continuous combinatorics of abelian group actions., 2018.

\bibitem[GJKS24]{GJKS2}
S.~Gao, S.~Jackson, E.~Krohne, and B.~Seward.
\newblock Borel combinatorics of countable group actions.
\newblock {\em arXiv:2401.13866}, 2024.

\bibitem[GKM17]{ghaffari2017complexity}
Mohsen Ghaffari, Fabian Kuhn, and Yannic Maus.
\newblock On the complexity of local distributed graph problems.
\newblock In {\em Proceedings of the 49th Annual ACM SIGACT Symposium on Theory
  of Computing}, pages 784--797, 2017.

\bibitem[GMP17]{measurablesquare}
{\L}ukasz Grabowski, Andr\'{a}s M\'{a}th\'{e}, and Oleg Pikhurko.
\newblock {M}easurable circle squaring.
\newblock {\em Ann. of Math. (2)}, 185(2):671--710, 2017.

\bibitem[GMP22]{OlegExpasion}
Lukasz Grabowski, András Máthé, and Oleg Pikhurko.
\newblock Measurable equidecompositions for group actions with an expansion
  property.
\newblock {\em J. Eur. Math. Soc.}, 24:4277--4326, 2022.

\bibitem[GNK{\etalchar{+}}85]{Gabow}
H.~N. Gabow, T.~Nishizeki, O.~Kariv, D.~Leven, and O.~Terada.
\newblock Algorithms for edge-coloring graphs.
\newblock {\em Tel Aviv University}, 1985.

\bibitem[GP20]{grebik2020measurable}
Jan Greb{\'\i}k and Oleg Pikhurko.
\newblock Measurable versions of vizing's theorem.
\newblock {\em Advances in Mathematics}, 374:107378, 2020.

\bibitem[GPS88]{goldberg88}
Andrew~V. Goldberg, Sergey~A. Plotkin, and Gregory~E. Shannon.
\newblock Parallel symmetry-breaking in sparse graphs.
\newblock {\em SIAM Journal on Discrete Mathematics}, 1(4):434--446, 1988.

\bibitem[GR21]{grebik_rozhon2021LCL_on_paths}
Jan Greb{\'\i}k and V{\'a}clav Rozho{\v{n}}.
\newblock Classification of local problems on paths from the perspective of
  descriptive combinatorics.
\newblock In {\em Extended Abstracts EuroComb 2021: European Conference on
  Combinatorics, Graph Theory and Applications}, pages 553--559. Springer,
  2021.

\bibitem[GR23]{grebik_rozhon2021toasts_and_tails}
Jan Grebík and Václav Rozhoň.
\newblock Local problems on grids from the perspective of distributed
  algorithms, finitary factors, and descriptive combinatorics.
\newblock {\em Advances in Mathematics}, 431:109241, 2023.

\bibitem[GRB22]{brandt_grunau_rozhon2021classification_of_lcls_trees_and_grids}
Christoph Grunau, V\'{a}clav Rozho\v{n}, and Sebastian Brandt.
\newblock The landscape of distributed complexities on trees and beyond.
\newblock In {\em Proceedings of the 2022 ACM Symposium on Principles of
  Distributed Computing}, PODC'22, page 37–47, New York, NY, USA, 2022.
  Association for Computing Machinery.

\bibitem[Gre22]{grebikApprox}
Jan Greb{\'\i}k.
\newblock Approximate {S}chreier decorations and approximate {K}\"{o}nig’s
  line coloring {T}heorem.
\newblock {\em Annales Henri Lebesgue}, 5:303--315, 2022.

\bibitem[Gre24]{GrebikVizing}
Jan Greb{\'\i}k.
\newblock Measurable {V}izing's theorem.
\newblock {\em accepted to Forum Math Sigma}, 2024.

\bibitem[Gro93]{Gromov}
M.~Gromov.
\newblock {\em Asymptotic invariants of infinite groups, Geometric group
  theory, Vol. 2.}, volume 182 of {\em London Math. Soc. Lecture Note Ser.}
\newblock Cambridge University Press, 1993.

\bibitem[Gup67]{Gupta}
R.~P. Gupta.
\newblock Studies in the theory of graphs.
\newblock {\em Ph.D. thesis, Tata Institute of Fundamental Research, Bombay},
  1967.

\bibitem[GV25]{grebikcomplexity}
Jan Greb{\'i}k and Zolt{\'a}n Vidny{\'a}nszky.
\newblock Complexity of linear equations and infinite gadgets.
\newblock {\em arXiv:2501.06114}, 2025.

\bibitem[HLS14]{hatamilovaszszegedy}
Hamed Hatami, L\'{a}szl\'{o} Lov\'{a}sz, and Bal\'{a}zs Szegedy.
\newblock Limits of locally-globally convergent graph sequences.
\newblock {\em Geom. Funct. Anal.}, 24(1):269--296, 2014.

\bibitem[Hol81]{holyer1981np}
Ian Holyer.
\newblock The {NP}-completeness of edge-coloring.
\newblock {\em SIAM Journal on computing}, 10(4):718--720, 1981.

\bibitem[HS21]{hirvonen2021distributed}
Juho Hirvonen and Jukka Suomela.
\newblock Distributed algorithms 2020.
\newblock {\em Finland: Aalto University, Dec}, 11:221, 2021.

\bibitem[HSW17]{HolroydSchrammWilson2017FinitaryColoring}
Alexander~E. Holroyd, Oded Schramm, and David~B. Wilson.
\newblock Finitary coloring.
\newblock {\em Ann. Probab.}, 45(5):2867--2898, 2017.

\bibitem[JKL02]{jackson2002countable}
Steve Jackson, Alexander~S. Kechris, and Alain Louveau.
\newblock Countable {B}orel equivalence relations.
\newblock {\em J. Math. Logic}, 2(01):1--80, 2002.

\bibitem[Kar75]{karp1975computational}
Richard~M Karp.
\newblock On the computational complexity of combinatorial problems.
\newblock {\em Networks}, 5(1):45--68, 1975.

\bibitem[Kec95]{kechrisclassical}
Alexander~S. Kechris.
\newblock {\em Classical descriptive set theory}, volume 156 of {\em Graduate
  Texts in Mathematics}.
\newblock Springer-Verlag, New York, 1995.

\bibitem[Kec24]{kechris2024theory}
Alexander~S Kechris.
\newblock {\em The theory of countable Borel equivalence relations}.
\newblock Cambridge University Press, 2024.

\bibitem[KM04]{topics}
Alexander~S. Kechris and Benjamin~D. Miller.
\newblock {\em Topics in orbit equivalence}, volume 1852 of {\em Lecture Notes
  in Mathematics}.
\newblock Springer-Verlag, Berlin, 2004.

\bibitem[KM20]{kechris_marks2016descriptive_comb_survey}
Alexander~S. Kechris and Andrew~S. Marks.
\newblock Descriptive graph combinatorics.
\newblock {\em
  \url{http://www.math.caltech.edu/~kechris/papers/combinatorics20book.pdf}},
  2020.

\bibitem[KST99]{KST}
Alexander~S. Kechris, S{\l}awomir Solecki, and Stevo Todor{\v{c}}evi{\'c}.
\newblock Borel chromatic numbers.
\newblock {\em Adv. Math.}, 141(1):1--44, 1999.

\bibitem[Kun21]{KunPerfectMatching}
G\'{a}bor Kun.
\newblock The measurable {H}all theorem fails for treeings.
\newblock {\em arXiv:2106.02013}, 2021.

\bibitem[Lac90]{laczk}
Mikl{\'os} Laczkovich.
\newblock Equidecomposability and discrepancy; a solution of {T}arski's
  circle-squaring problem.
\newblock {\em J. Reine Angew. Math.}, 404:77--117, 1990.

\bibitem[Lin92]{linial92LOCAL}
Nati Linial.
\newblock Locality in distributed graph algorithms.
\newblock {\em SIAM Journal on Computing}, 21(1):193--201, 1992.

\bibitem[LN11]{lyons2011perfect}
Russell Lyons and Fedor Nazarov.
\newblock Perfect matchings as {IID} factors on non-amenable groups.
\newblock {\em European Journal of Combinatorics}, 32(7):1115--1125, 2011.

\bibitem[Lyo17]{LyonsTrees}
Russell Lyons.
\newblock Factors of {IID} on trees.
\newblock {\em Combin. Probab. Comput.}, 26(2):285--300, 2017.

\bibitem[Lyo24]{lyons2024descriptive}
Clark~Richard Lyons.
\newblock {\em Descriptive Combinatorics on Trees, Grids, and Non-Amenable
  Graphs}.
\newblock PhD thesis, UCLA, 2024.

\bibitem[Mar75]{martin}
Donald~A. Martin.
\newblock Borel determinacy.
\newblock {\em Ann. of Math. (2)}, 102(2):363--371, 1975.

\bibitem[Mar16]{DetMarks}
Andrew~S. Marks.
\newblock A determinacy approach to {B}orel combinatorics.
\newblock {\em J. Amer. Math. Soc.}, 29(2):579--600, 2016.

\bibitem[Mar22]{marks2022measurable}
Andrew~S Marks.
\newblock Measurable graph combinatorics.
\newblock In {\em Proc. Int. Cong. Math}, volume~3, pages 1488--1502, 2022.

\bibitem[McD02]{mcdiarmid2002concentration}
Colin McDiarmid.
\newblock Concentration for independent permutations.
\newblock {\em Combinatorics Probability and Computing}, 11(2):163--178, 2002.

\bibitem[MR14]{Molloy}
Michael Molloy and Bruce Reed.
\newblock Colouring graphs when the number of colours is almost the maximum
  degree.
\newblock {\em Journal of Combinatorial Theory, Series B}, 109:134--195, 2014.

\bibitem[MU16]{marks2016baire}
Andrew Marks and Spencer Unger.
\newblock Baire measurable paradoxical decompositions via matchings.
\newblock {\em Advances in Mathematics}, 289:397--410, 2016.

\bibitem[MU17]{marksunger}
Andrew~S. Marks and Spencer~T. Unger.
\newblock Borel circle squaring.
\newblock {\em Ann. of Math. (2)}, 186(2):581--605, 2017.

\bibitem[MWW04]{mckayshortcycles}
Brendan~D. McKay, Nicholas~C. Wormald, and Beata Wysocka.
\newblock Short cycles in random regular graphs.
\newblock {\em Electron. J. Combin.}, 11(1):Research Paper 66, 12, 2004.

\bibitem[NO08]{nguyen_onak2008matching}
Huy~N. Nguyen and Krzysztof Onak.
\newblock Constant-time approximation algorithms via local improvements.
\newblock In {\em 2008 49th Annual IEEE Symposium on Foundations of Computer
  Science}, pages 327--336, 2008.

\bibitem[NS95]{naorstockmeyer}
Moni Naor and Larry Stockmeyer.
\newblock What can be computed locally?
\newblock {\em SIAM Journal on Computing}, 24(6):1259--1277, 1995.

\bibitem[Pik21]{pikhurko2021descriptive_comb_survey}
Oleg Pikhurko.
\newblock Borel combinatorics of locally finite graphs.
\newblock {\em Surveys in Combinatorics 2021, the 28th British Combinatorial
  Conference}, pages 267--319, 2021.

\bibitem[PS92]{panconesi-srinivasan}
Alessandro Panconesi and Aravind Srinivasan.
\newblock Improved distributed algorithms for coloring and network
  decomposition problems.
\newblock In {\em Proc.\ 24th ACM Symp.\ on Theory of Computing (STOC)}, pages
  581--592, 1992.

\bibitem[PS95]{PanconesiSrinivasan}
A.~Panconesi and A.~Srinivasan.
\newblock The local nature of {$\Delta$}-coloring and its algorithmic
  applications.
\newblock {\em Combinatorica}, 15:255--280, 1995.

\bibitem[QW22]{QianWeilacher}
Long Qian and Felix Weilacher.
\newblock Descriptive combinatorics, computable combinatorics, and asi
  algorithms.
\newblock {\em arXiv:2206.08426}, 2022.

\bibitem[RG20]{RozhonG19}
Václav Rozho\v{n} and Mohsen Ghaffari.
\newblock Polylogarithmic-time deterministic network decomposition and
  distributed derandomization.
\newblock In {\em Proc.~Symposium on Theory of Computation (STOC)}, 2020.

\bibitem[Roz24]{rozhovn2024invitation}
V{\'a}clav Rozho{\v{n}}.
\newblock Invitation to local algorithms.
\newblock {\em arXiv:2406.19430}, 2024.

\bibitem[Sew]{Seward_personal}
Brandon Seward.
\newblock Personal communication.

\bibitem[Sin21]{Fast}
Corwin Sinnamon.
\newblock Fast and simple edge-coloring algorithms.
\newblock {\em arXiv:1907.03201}, 2021.

\bibitem[Sku06]{BrooksAlg}
San Skulrattanakulchai.
\newblock {$\Delta$}-list vertex coloring in linear time.
\newblock {\em Information Processing Letters}, 98(3):101--106, 2006.

\bibitem[SSTF12]{EdgeColoringBook}
M.~Stiebitz, D.~Scheide, B.~Toft, and L.~M. Favrholdt.
\newblock {\em Graph edge coloring}.
\newblock Wiley Series in Discrete Mathematics and Optimization. John Wiley \&
  Sons, Inc., Hoboken, NJ, 2012.

\bibitem[SV19]{SuVu}
Hsin-Hao Su and Hoa~T. Vu.
\newblock Towards the locality of {V}izing’s theorem.
\newblock In {\em Proceedings of the 51st Annual ACM SIGACT Symposium on Theory
  of Computing}, STOC 2019, page 355–364, New York, NY, USA, 2019.
  Association for Computing Machinery.

\bibitem[Tho22a]{thornton2022algebraic}
Riley Thornton.
\newblock An algebraic approach to {B}orel {CSP}s.
\newblock {\em arXiv:2203.16712}, 2022.

\bibitem[Tho22b]{thornton2021orienting}
Riley Thornton.
\newblock Orienting {B}orel graphs.
\newblock {\em Proceedings of the American Mathematical Society},
  150(4):1779--1793, 2022.

\bibitem[T{\'o}t21]{TothShcreier}
L{\'a}szlo~M. T{\'o}th.
\newblock Invariant {S}chreier decorations of unimodular random networks.
\newblock {\em Annales Henri Lebesgue}, 4:1705--1726, 2021.

\bibitem[TV21]{todorvcevic2021complexity}
Stevo Todor{\v{c}}evi{\'c} and Zolt{\'a}n Vidny{\'a}nszky.
\newblock A complexity problem for {B}orel graphs.
\newblock {\em Invent. Math.}, 226:225–249, 2021.

\bibitem[Viz65]{Vizing}
Vadim~G. Vizing.
\newblock Critical graphs with given chromatic class.
\newblock {\em Metody Diskret. Analiz.}, 5:9--17, 1965.

\bibitem[Wei21]{weilacher2021borel}
Felix Weilacher.
\newblock Borel edge colorings for finite dimensional groups.
\newblock {\em arXiv:2104.14646}, 2021.

\bibitem[Wil14]{wilson2014four}
Robin Wilson.
\newblock {\em Four Colors Suffice: How the Map Problem Was Solved-Revised
  Color Edition}, volume~30.
\newblock Princeton university press, 2014.

\bibitem[Wor99]{wormald1999models}
Nicholas~C. Wormald.
\newblock Models of random regular graphs.
\newblock {\em London Mathematical Society Lecture Note Series}, pages
  239--298, 1999.

\bibitem[Zhu20]{zhuk2020proof}
Dmitriy Zhuk.
\newblock A proof of the {CSP} dichotomy conjecture.
\newblock {\em Journal of the ACM (JACM)}, 67(5):1--78, 2020.

\bibitem[Zuc06]{zuckerman2006linear}
David Zuckerman.
\newblock Linear degree extractors and the inapproximability of max clique and
  chromatic number.
\newblock In {\em Proceedings of the thirty-eighth annual ACM symposium on
  Theory of computing}, pages 681--690, 2006.

\end{thebibliography}

\end{document}